\pdfoutput=1

\documentclass[EJP,preprint]{ejpecp}

\SHORTTITLE{Swapping and switching}

\TITLE{Swapping and Switching for Markov Processes:\\
	Symmetrisation and Fast-Exchange Limits%
	\support{Supported by the Marie Sk\l{}odowska-Curie Actions (MSCA) Staff
		Exchange project SIMBAD, REA Grant Agreement No.~101131463.}}

\AUTHORS{%
	Jos\'e~Juli\'an~D\'iaz-P\'erez\footnote{Department of Functional Analysis,
		Faculty of Mathematics and Computer Science, University of Havana, Havana
		10400, Cuba; and Group of Complex Systems and Statistical Physics, Faculty
		of Physics, University of Havana, Havana 10400, Cuba.
		\BEMAIL{diazperezjj44@gmail.com}}
	\and
	Roberto~Mulet\footnote{Group of Complex Systems and Statistical Physics and
		Department of Theoretical Physics, Faculty of Physics, University of
		Havana, Havana 10400, Cuba.
		\BEMAIL{mulet@fisica.uh.cu}}}


\AMSSUBJ{60J25}
\AMSSUBJSECONDARY{60J27 ; 60J60 ; 60J76 ; 60F17 ; 65C05 ; 65C40}

\VOLUME{0}
\YEAR{2026}
\PAPERNUM{0}
\DOI{10.1214/YY-TN}

\ABSTRACT{In parallel tempering, exchanging the positions of the replicas
	(swapping) and exchanging the assignment of dynamics to
	the replicas (switching) are often used interchangeably. We study the two
	descriptions as stochastic processes, for Markov species (the dynamics that are
	exchanged) and bounded measurable exchange rates. To guarantee the Markovianity
	of each process under the exchanges, switching records both the positions and
	the label that tracks which species each replica carries; swapping records only
	the positions. Driving both with one exchange clock places them on a common probability space
	and lets us write swapping as switching with the label forgotten. This relation
	holds pathwise at every time and, once decomposed over the label, yields an
	exact identity between their transition kernels. This characterisation lets us
	relate, in a precise way, where the two processes agree and where they differ.
	We describe mainly the relation between their decay rates at finite exchange
	rate, and between their distributions and observables at finite and infinite
	exchange rate. The relation between the Markovianity of the two is particularly
	interesting. At finite rate, switching needs the label to be Markov, while
	swapping is Markov on its own. In the infinite-rate limit this reverses:
	swapping becomes non-tight, while switching becomes Markov on the physical
	coordinates alone. We illustrate all these relations on a solvable benchmark of
	two Brownian motions with drift under a constant exchange rate. The rich
	behaviour already visible in this simple case, and developed throughout the
	paper, suggests that the swapping mechanism is of interest in its own right,
	beyond the balance conditions needed in parallel tempering.}

\newcommand{\dif}{\mathrm{d}}
\newcommand{\PP}{\mathbb{P}}
\newcommand{\EE}{\mathbb{E}}
\newcommand{\RR}{\mathbb{R}}
\newcommand{\FF}{\mathcal{F}}
\newcommand{\LL}{\mathcal{L}}
\newcommand{\Var}{\mathrm{Var}}
\newcommand{\Cov}{\mathrm{Cov}}
\newcommand{\Corr}{\mathrm{Corr}}
\newcommand{\swap}{\mathrm{swap}}
\newcommand{\rs}{\mathrm{rs}}   
\newcommand{\Occ}{\mathcal{O}} 

\begin{document}
	
	\section{Introduction}
	In replica-exchange Monte Carlo, or parallel tempering
	\cite{earl2005,geyer1991,hukushima1996,swendsen1986}, $N$ replicas of a
	system evolve at different temperatures, and from time to time two of them
	exchange their configurations with a probability that depends on their
	energies. During a run one may follow a temperature and record the
	configuration it holds, or follow a replica and record the temperature it
	carries. These two readings, which we call \emph{swapping} and
	\emph{switching}, respectively, are usually treated as interchangeable in
	practice \cite{bussi2014,weigel2012,surjanovic2025}. This practical
	equivalence, however, must be used with care. This paper establishes the exact relationship between the two descriptions in continuous time, where exchanges are proposed at the rings of a Poisson clock and accepted according to bounded measurable rates. We replace the dynamics at each temperature with arbitrary Markov processes, which we refer to as \emph{species}. For the analysis of the
	fast-exchange limit we add continuity, compact-containment and
	well-posedness conditions. The resulting framework encompasses the main classes of models considered in the literature, including jump processes on finite state spaces \cite{doll2018}, such as spin systems under Glauber dynamics, lattice walkers \cite{diazperez2025}, and Langevin replicas \cite{chen2019,dong2022,dupuis2012}.
	
	The relation between the two descriptions has been used implicitly since
	\cite{dupuis2012,plattner2011}, where the infinite-swapping limit is
	obtained from the switching description, because the swapping process has
	no c\`adl\`ag limit as the exchange rate tends to infinity. In
	\cite{dupuis2012} swapping and switching are called the particle-swapped and
	the temperature-swapped processes, respectively; we prefer the term
	switching, which agrees with the terminology of regime-switching processes
	\cite{xiyinzhu2019,yinzhu2010}. The same pair of descriptions underlies
	\cite{doll2018,lu2019}, and similar pairs of schemes appear in discrete time
	in \cite{chodera2011,latz2021}. In all these works the two descriptions are
	related in law, through empirical measures or in the limit, and for
	specific dynamics. We ask, instead, what the exact relation between the two
	processes is at a fixed exchange rate, which quantities they share and
	which they do not, and what each of them becomes in the fast-exchange
	limit.
	
	Our main contributions fall into three groups. The first extends the
	convergence of the switching positions to the dynamics of infinite
	swapping, and the non-tightness and the empirical-measure limit of the
	swapping process, obtained for Langevin replicas in
	\cite{dupuis2012,lu2019}, to general species, exchange rates and connected
	exchange graphs. The second turns relations known only in law, and for
	specific models \cite{doll2018,dupuis2012,lu2019}, into a pathwise coupling
	at every rate, and establishes the Markov structure of both descriptions,
	for which only the general fact that the positions of a regime-switching
	process need not be Markov was available \cite{yinzhu2010}. The third
	consists of results that, to our knowledge, are new: the splitting of the
	decay rate of switching, the limits of the finite-dimensional laws of the
	swapping process and an exactly solvable Brownian benchmark with drifts.
	
To state these results precisely, let $\mathbf X_t=(X^1_t,\dots,X^N_t)$ denote the swapping process, in which each coordinate retains the dynamics of its species and jumps whenever an exchange is accepted. The switching process is the pair $(\widehat{\mathbf X}_t,\Lambda_t)$, where the $k$-th coordinate of $\widehat{\mathbf X}_t$ represents the position of replica $k$ and remains unchanged at exchange times. Here, $\Lambda_t$ is a permutation of $\{1,\dots,N\}$, called the \emph{label}, which keeps track of the species carried by each replica, with replica $k$ carrying species $\Lambda_t(k)$. We construct both descriptions on a common probability space, driven by the same exchange clock. The following identity then holds pathwise at every exchange rate:
\begin{equation}\label{eq:intro-coupling}
	\mathbf X_t=\Lambda_t\widehat{\mathbf X}_t,
	\qquad\text{that is,}\qquad
	X^m_t=\widehat X^{\Lambda_t^{-1}(m)}_t ,
\end{equation}
(see Theorem~\ref{thm:N-symmetrisation}).
	Decomposing over the value of $\Lambda_t$ expresses the transition kernel of
	the swapping process as a sum of pieces of the law of the switching
	process, an identity we call the \emph{symmetrisation identity}. The pathwise coupling places on a common probability space the two types
	of simulation histories considered in \cite{doll2012}, which follow either
	a fixed temperature or a replica across temperatures. It also allows us to
	study joint functionals of the configurations and the label. 
	Figure~\ref{fig:sw_trajectories} shows one sample path read both ways for
	the benchmark of Section~\ref{sec:brownian}.
	
	\begin{figure}[htbp]
		\centering
		\includegraphics[width=\textwidth]{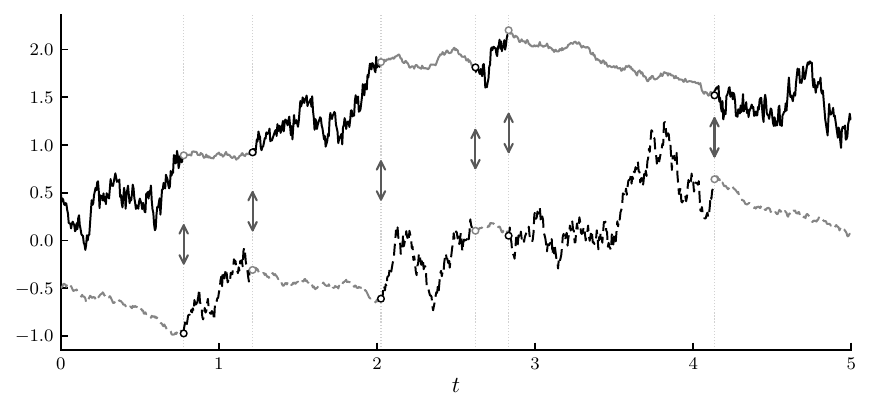}
		\caption{%
			\textbf{Same physical trajectories seen as swapping and as switching.}
			The black and grey curves follow the swapping process $\mathbf X_t$,
			started at $X^1_0=0.5$, $X^2_0=-0.5$, with drifts
			$(\mu_1,\mu_2)=(0.5,-0.5)$, volatilities $(\sigma_1,\sigma_2)=(1.0,0.2)$
			and constant exchange rate $s=1$. The exchanges are marked by vertical
			double-headed arrows. Following the solid and dashed lines instead
			gives the corresponding switching process $\widehat{\mathbf X}_t$, in
			which the replicas exchange their drifts and volatilities at the same
			instants.
		}
		\label{fig:sw_trajectories}
	\end{figure}
	
	The swapping process $\mathbf X_t$ is Markov on its own, whereas the
	positions $\widehat{\mathbf X}_t$ of the switching process
	are in general not Markov without the
	label. They fail to be Markov when their past reveals the current label
	and the label changes their future law on a set of positions reached with
	two different values of the label (Corollary~\ref{cor:non-markov}), as for
	two Brownian species with different volatilities
	(Corollary~\ref{cor:bench-non-markov}). The relation between the two
	descriptions is that of a quotient. Renumbering the replicas changes
	$\Lambda_t$ but leaves $\mathbf X_t$ unchanged, and the observables of the
	switching process that do not depend on the numbering are exactly the
	functions of $\mathbf X_t$ (Proposition~\ref{prop:markov-unique}).
	
	Every invariant law of the swapping process lifts to one of the switching
	process, under which the label is uniform and independent of the swapping
	coordinate (Proposition~\ref{prop:invariant}). In this setting, the decay rates of the two descriptions can be compared. The
	convergence of parallel tempering to equilibrium is usually studied
	through the $L^2$ decay rate of the process that exchanges configurations
	\cite{leeshen2023,madraszheng2003,woodard2009,zheng2003}. At a fixed
	exchange rate, the functions of $\mathbf X_t$ form an invariant subspace of
	the switching semigroup, and as a consequence the $L^2$ decay rate of
	switching splits as
	\begin{equation}\label{eq:intro-gap}
		\lambda_\rs=\min(\lambda_\swap,\lambda_\perp),
	\end{equation}
	where $\lambda_\swap$ is the decay rate of the swapping process and
	$\lambda_\perp$ that of the observables with zero conditional mean given the swapping coordinate (Theorem~\ref{thm:gap-decomposition}). Switching can therefore
	relax more slowly than swapping, never faster, and only through these observables. The two rates can be compared explicitly for $N$ identical
	species exchanged at constant rates, for which
	$\lambda_\swap=\lambda_1$ and
	$\lambda_\rs=\min(\lambda_1,\mathrm{gap}(L_\gamma))$, where $\lambda_1$ is
	the rate of a single species and $L_\gamma$ the Laplacian of the exchange
	graph, assumed connected (Corollary~\ref{cor:identical}), so that
	switching is slower exactly when the label relaxes more slowly than a
	single species.
	
	When the exchange rates are multiplied by $s\to\infty$, the label of the
	switching process moves much faster than the replicas, and we write
	$\widehat{\mathbf X}^{(s)}$ and $\mathbf X^{(s)}$ for the two descriptions.
	With the positions frozen at $\mathbf y$, the label is a Markov chain on the
	symmetric group $\mathfrak S_N$ of permutations of $\{1,\dots,N\}$, with
	invariant law $\nu_{\mathbf y}$. Under continuity and compact-containment
	conditions, and for any connected exchange graph, the switching positions
	converge as a process to a Markov process $\bar{\mathbf Y}$, so that the
	Markov property they lack at finite rate is recovered in the limit, with
	generator
	\begin{equation}\label{eq:intro-avg}
		\bar\LL f(\mathbf y)=\sum_{\theta\in\mathfrak S_N}\nu_{\mathbf y}(\theta)\,
		\LL_\theta f(\mathbf y),
	\end{equation}
	where $\LL_\theta$ is the generator of the replicas when replica $k$ carries
	species $\theta(k)$ (Theorem~\ref{thm:switching-conv}). With constant rates
	$\nu_{\mathbf y}$ is uniform and each replica follows the average of the
	species, and with Metropolis rates $\nu_{\mathbf y}$ is the weight of
	infinite swapping \cite{dupuis2012,lu2019}. Our proof applies the
	averaging theorem of Kurtz \cite{kurtz1992}, with the label as the fast
	variable, and identifies the limit through the martingale problem for the
	averaged generator $\bar\LL$.
	When the exchanges preserve the product of the targets, the limit of the swapping process recovers them, each species having its own target as one-time law
	in the stationary limit, and when the preservation holds at every point the
	limit is the same for every such rule as for Metropolis rates
	(Corollary~\ref{cor:targets}).
	
	The swapping process is in general not tight (Corollary~\ref{cor:nontight}),
	but its finite-dimensional laws converge,
	\begin{equation}\label{eq:intro-fdd}
		\bigl(\mathbf X^{(s)}_{t_1},\dots,\mathbf X^{(s)}_{t_J}\bigr)
		\Rightarrow
		\bigl(\Xi_1\bar{\mathbf Y}_{t_1},\dots,\Xi_J\bar{\mathbf Y}_{t_J}\bigr),
	\end{equation}
	where, given $\bar{\mathbf Y}$, the permutations $\Xi_1,\dots,\Xi_J$ are
	independent and $\Xi_i$ has law $\nu_{\bar{\mathbf Y}_{t_i}}$
	(Theorem~\ref{thm:fdd}). In the limit the
	swapping process is therefore the averaged dynamics relabelled afresh at
	every time, and its empirical measure converges to the weighted empirical
	measure of infinite swapping. These finite-dimensional laws are given by a
	Markov transition function, as at every finite rate, but this transition
	function tends at small times to the projection onto permutation-invariant
	functions rather than to the identity, and this is what rules out a
	c\`adl\`ag limit.
	
	For two Brownian species with different drifts and volatilities and a
	constant exchange rate, the densities and moments of both descriptions are
	explicit at every rate, and the two-time correlations are explicit at long
	times (Section~\ref{sec:brownian}). The coupling, the failure of the Markov
	property of the switching positions and the fast-exchange limits can all be
	followed explicitly in this model, which also shows how the two descriptions
	compare on quantities that follow a single species or a single replica. The
	model is close to the lattice random walks with stochastic exchange solved
	exactly in \cite{diazperez2025}, but the subject here is the relation between
	swapping and switching, and the drifts change the behaviour. Read as a pairwise interaction between two particles, the exchange of
	positions already produces a rich phenomenology in this simple model, which
	suggests that swapping is of interest as a stochastic process in its own
	right, beyond parallel tempering.
	
	The paper is organised as follows. Section~\ref{sec:setup} constructs the
	two processes, proves the coupling, the quotient structure and the
	correspondence of invariant laws, and derives their consequences for decay
	rates. Section~\ref{sec:fast} treats the
	fast-exchange limit and Section~\ref{sec:brownian} the Brownian benchmark.
	Several proofs are deferred to the appendices.
	
	\medskip\noindent\textit{Notation.} We write $[N]:=\{1,\dots,N\}$ and
	$\mathfrak S_N$ for the symmetric group on $[N]$, whose products are written
	by juxtaposition, $\circ$ being kept for composition of functions, and
	$\tau_{ab}$ for the transposition of $a$ and $b$, with $\tau:=\tau_{12}$
	when $N=2$. For a set $E$, the group $\mathfrak S_N$ acts on $E^N$ by
	permuting coordinates,
	\begin{equation}\label{eq:action}
		(\theta\mathbf{x})^k:=x^{\theta^{-1}(k)},\qquad
		\theta\in\mathfrak{S}_N,
	\end{equation} 
	so that $(\theta\mathbf{x})^{\theta(k)}=x^k$,
	$\theta(\varpi\mathbf{x})=(\theta\varpi)\mathbf{x}$ for
	$\theta,\varpi\in\mathfrak{S}_N$, and $\tau_{ab}\mathbf{x}$ is $\mathbf{x}$
	with the coordinates $a$ and $b$ interchanged. We write
	$a\wedge b=\min\{a,b\}$, $a\vee b=\max\{a,b\}$, and $\Rightarrow$ for
	convergence in distribution.
	
	For a locally compact separable metric space $\mathsf F$, with its Borel
	$\sigma$-algebra, $C_b(\mathsf F)$ and $C_c(\mathsf F)$ are the bounded
	continuous and the compactly supported continuous functions on $\mathsf F$,
	$\mathcal P(\mathsf F)$ is the space of probabilities on $\mathsf F$ with
	the topology of weak convergence, and
	$\|\varrho\|_{TV}=\sup_{|f|\le1}|\varrho(f)|$ is the total variation of a
	finite signed measure, the supremum being over measurable $f$, so that the
	total variation of a difference of two probabilities is twice the usual
	distance. A path is \emph{c\`adl\`ag} if it is right-continuous with left
	limits, and the \emph{Skorokhod space} $D([0,\infty),\mathsf F)$ is the set
	of such paths, with the $J_1$ topology of Skorokhod \cite{skorokhod1956},
	see \cite[Section~VI.1]{jacodshiryaev2003}. That a process has paths in
	$D([0,\infty),\mathsf F)$ includes that it is defined, in $\mathsf F$, for
	all times, that is, that it does not explode. A \emph{Markov family} on
	$\mathsf F$ is a collection $(\PP_x)_{x\in\mathsf F}$ of laws on
	$D([0,\infty),\mathsf F)$, $\PP_x$ started at $x$, such that
	$x\mapsto\PP_x(B)$ is measurable for every Borel set $B$ and the Markov property holds in
	the natural filtration.

	\section{Swapping and switching at finite rate}\label{sec:setup}
	
	Given a state space, one Markov family per species and bounded measurable
	exchange rates, this section builds the two processes from a single
	exchange clock and relates them at a fixed rate. The state space is a
	locally compact separable metric space $E$, in which the replicas evolve.
	The two cases of interest are $E=\RR^d$, the setting of Langevin replicas
	and of the benchmark, and $E$ countable with the discrete topology, the
	setting of spin systems and lattice walkers.
	
	\begin{definition}[Species]\label{def:species}
		A \emph{species} is a time-homogeneous Markov
		family $(\PP_x)_{x\in E}$ on $E$. We write
		$P_tf(x)=\EE_xf(X_t)$ for its transition semigroup, acting on bounded
		measurable functions.
	\end{definition}
	
	Given species $1,\dots,N$, the \emph{decoupled process} on $E^N$ started
	at $\mathbf{x}=(x^1,\dots,x^N)$ has independent coordinates, coordinate
	$m$ following species $m$ from $x^m$.
	
	Species $a$ and $b$ exchange at a rate $\gamma_{ab}(\mathbf x)$, where
	$\gamma_{ab}:E^N\to[0,\infty)$, $1\le a<b\le N$, are bounded and
	measurable, and we set $\gamma_{ba}:=\gamma_{ab}$,
	$\Gamma_{ab}:=\sup_{E^N}\gamma_{ab}$ and
	$\Gamma_N:=\sum_{a<b}\Gamma_{ab}$, which we assume positive, since
	otherwise there are no exchanges. The \emph{exchange graph} $G_\gamma$
	has vertex set $[N]$ and an edge $ab$ whenever $\gamma_{ab}\not\equiv0$.
	
\begin{definition}[Exchange clock]\label{def:clock}
	The \emph{exchange clock} consists of a Poisson process of rate $\Gamma_N$,
	whose jump times $T_1<T_2<\cdots$ we call \emph{rings} (with $T_0:=0$),
	and of marks $(\{a_n,b_n\},U_n)$, $n\ge1$, which are independent of one
	another and of the ring times. The unordered pair $\{a_n,b_n\}$ equals
	$\{a,b\}$ with probability $\Gamma_{ab}/\Gamma_N$, which is positive
	exactly for the pairs with $\Gamma_{ab}>0$, and, conditionally on the pair,
	$U_n$ is uniform on $[0,\Gamma_{a_nb_n}]$. The $n$-th ring is
	\emph{accepted at} $\mathbf y\in E^N$ if $U_n\le\gamma_{a_nb_n}(\mathbf y)$.
\end{definition}
	
	A ring thus proposes the pair $\{a,b\}$ with probability
	$\Gamma_{ab}/\Gamma_N$ and is accepted at $\mathbf y$ with probability
	$\gamma_{ab}(\mathbf y)/\Gamma_{ab}$, so that the accepted exchanges of
	the pair occur at rate $\gamma_{ab}(\mathbf y)$. The two descriptions
	share this clock and differ only in what an accepted ring exchanges.
	
	\begin{definition}[Swapping process]\label{def:N-swap}
		The \emph{swapping process} $\mathbf{X}=(X^1,\dots,X^N)$ started at
		$\mathbf{x}\in E^N$ is driven by the exchange clock, with
		$\mathbf{X}_0=\mathbf{x}$. On $[T_n,T_{n+1})$ it is a copy of the decoupled
		process started at $\mathbf{X}_{T_n}$ which, given $\mathbf{X}_{T_n}$, is
		independent of the clock and of the past. At a ring accepted at
		$\mathbf{X}_{T_n-}$ the two coordinates named by the mark exchange their
		positions, $\mathbf{X}_{T_n}=\tau_{a_nb_n}\mathbf{X}_{T_n-}$, and
		otherwise $\mathbf{X}_{T_n}=\mathbf{X}_{T_n-}$.
	\end{definition}
	
	In the switching description the positions never jump, and what changes
	at an exchange is the assignment of species to replicas. Coordinate $k$
	of $\mathbf x\in E^N$ is the position of replica $k$, and a permutation
	$\theta\in\mathfrak{S}_N$, the \emph{label}, records that replica $k$
	carries species $\theta(k)$. By \eqref{eq:action}, coordinate $a$ of
	$\theta\mathbf{x}$ is then the position of species $a$, so
	$\theta\mathbf{x}$ lists the positions by species, as the swapping
	process does, and it is at $\theta\mathbf x$ that the rates are
	evaluated.
	
	\begin{definition}[Switching process]\label{def:switching}
		The \emph{switching process} $(\widehat{\mathbf{X}},\Lambda)$ on
		$E^N\times\mathfrak{S}_N$ started at $(\mathbf{x},\theta)$ is driven by
		the exchange clock, with
		$(\widehat{\mathbf{X}}_0,\Lambda_0)=(\mathbf{x},\theta)$. On
		$[T_n,T_{n+1})$ the label is constant, $\Lambda_t=\Lambda_{T_n}$, and
		$\widehat{\mathbf{X}}$ is a copy of the decoupled process in which
		replica $k$ carries species $\Lambda_{T_n}(k)$, started at
		$\widehat{\mathbf{X}}_{T_n}$ and, given
		$(\widehat{\mathbf{X}}_{T_n},\Lambda_{T_n})$, independent of the clock
		and of the past. The positions do not move at the rings,
		$\widehat{\mathbf{X}}_{T_n}=\widehat{\mathbf{X}}_{T_n-}$. At a ring
		accepted at $\Lambda_{T_n-}\widehat{\mathbf{X}}_{T_n-}$ the two species
		named by the mark exchange their replicas,
		$\Lambda_{T_n}=\tau_{a_nb_n}\Lambda_{T_n-}$, and otherwise
		$\Lambda_{T_n}=\Lambda_{T_n-}$.
	\end{definition}
	
	Since $\tau_{ab}\theta=\theta\tau_{kl}$ for $k=\theta^{-1}(a)$ and
	$l=\theta^{-1}(b)$, an accepted ring exchanges the species carried by
	replicas $k$ and $l$, at the rate
	$\gamma_{\theta(k)\theta(l)}(\theta\mathbf{x})$. The switching process is
	therefore a regime-switching process whose regime is the label, see
	\cite{xiyinzhu2019,yinzhu2010} for the general theory. Throughout the 
	paper $\widehat{\mathbf{X}}_t$ denotes the positions of the switching
	process, a process on $E^N$ distinct from the swapping process
	$\mathbf{X}_t$, and unless stated otherwise the switching process starts
	from $(\mathbf{x}_0,\mathrm{id})$.
	
	Both processes are instances of the piecing-out, or thinning, construction
	of a Markov process with an added bounded jump mechanism
	\cite{ikedanagasawawatanabe1966,meyer1975}, in which a Markov family runs
	between the rings of an independent Poisson clock and, at each ring, the
	process jumps by a measurable map with a probability given by a bounded
	rate. Appendix~\ref{app:piecing} proves the Markov property in that
	setting.
	
	\begin{theorem}\label{thm:exist}
		The swapping process and the switching process have c\`adl\`ag,
		non-explosive paths. They form time-homogeneous Markov families on $E^N$
		and on $E^N\times\mathfrak{S}_N$, respectively, whose laws are determined
		by the species and by the rates $\gamma_{ab}$.
	\end{theorem}
	
	\begin{proof}
		Both processes are piecing-outs in the sense of
		Appendix~\ref{app:piecing}, and the statement is
		Theorem~\ref{thm:piecing}.
	\end{proof}

	The swapping process is therefore Markov on the physical coordinates
	$\mathbf{x}\in E^N$ alone, while the switching process needs the label. It
	is well known that the positions of a regime-switching process are in
	general not Markov once the regime is forgotten \cite{yinzhu2010}. We write $\PP_{(\mathbf x,\theta)}$ and
	$\EE_{(\mathbf x,\theta)}$ for the law of the switching process started at
	$(\mathbf x,\theta)$ and its expectation, $\FF^{\widehat{\mathbf X}}_t$ for
	the completed natural filtration of the switching positions, started from $(\mathbf x_0,\mathrm{id})$, and
	$h^{g,u}_\theta(\mathbf x):=\EE_{(\mathbf x,\theta)}g(\widehat{\mathbf X}_u)$
	for bounded measurable $g$ on $E^N$ and $u\ge0$.
	
	\begin{proposition}\label{prop:filter}
		Let $\widehat\nu_t(\theta):=\PP(\Lambda_t=\theta\mid\FF^{\widehat{\mathbf X}}_t)$.
		The switching positions are a Markov process in their own filtration if,
		and only if, for all $t,u\ge0$ and every bounded measurable $g$, the
		variable $\sum_\theta\widehat\nu_t(\theta)\,h^{g,u}_\theta(\widehat{\mathbf X}_t)$
		is $\sigma(\widehat{\mathbf X}_t)$-measurable up to a null set.
	\end{proposition}
	
	\begin{proof}
		By the Markov property of the pair (Theorem~\ref{thm:exist}) and the
		tower property,
		$\EE[g(\widehat{\mathbf X}_{t+u})\mid\FF^{\widehat{\mathbf X}}_t]
		=\EE[h^{g,u}_{\Lambda_t}(\widehat{\mathbf X}_t)\mid\FF^{\widehat{\mathbf X}}_t]
		=\sum_\theta\widehat\nu_t(\theta)\,h^{g,u}_\theta(\widehat{\mathbf X}_t)$, and the
		Markov property of $\widehat{\mathbf X}$ in its own filtration says
		exactly that these conditional expectations are functions of
		$\widehat{\mathbf X}_t$.
	\end{proof}
	
	The process $\widehat\nu_t$ is the filter of a hidden Markov model in which the
	label is hidden and the positions are observed. The positions therefore
	fail to be Markov when their past carries information on the current label
	that the current position does not, at a position where the label matters.
	The criterion is immediate to check when the past reveals the label
	completely, which happens when the species can be distinguished from
	arbitrarily short pieces of path.
	
	\begin{corollary}\label{cor:non-markov}
		Suppose that for some $t>0$ the label $\Lambda_t$ is
		$\FF^{\widehat{\mathbf X}}_t$-measurable, and that for some $u>0$, some
		bounded measurable $g$ and some $\theta\ne\theta'$, the restrictions to
		$B:=\{h^{g,u}_\theta\ne h^{g,u}_{\theta'}\}$ of the measures
		$\upsilon_\vartheta:=\PP(\widehat{\mathbf X}_t\in\cdot\,,\ \Lambda_t=\vartheta)$,
		$\vartheta=\theta,\theta'$, are not mutually singular. Then the switching
		positions are not a Markov process in their own filtration.
	\end{corollary}
	
	\begin{proof}
		Since $\Lambda_t$ is $\FF^{\widehat{\mathbf X}}_t$-measurable,
		$\widehat\nu_t=\delta_{\Lambda_t}$, and the variable of
		Proposition~\ref{prop:filter} is $h^{g,u}_{\Lambda_t}(\widehat{\mathbf X}_t)$.
		If the positions were Markov, it would equal $\varphi(\widehat{\mathbf X}_t)$
		almost surely for some measurable $\varphi$, so that $\varphi=h^{g,u}_\vartheta$
		$\upsilon_\vartheta$-almost everywhere for $\vartheta=\theta,\theta'$. The set
		$A:=\{\mathbf x\in B:\ \varphi(\mathbf x)=h^{g,u}_\theta(\mathbf x)\}$
		would then satisfy $\upsilon_\theta(B\setminus A)=0$ and, since
		$h^{g,u}_\theta\ne h^{g,u}_{\theta'}$ on $B$, also $\upsilon_{\theta'}(A)=0$,
		so the two restrictions would be mutually singular.
	\end{proof}
	
	The hypotheses of Corollary~\ref{cor:non-markov} hold, in general, for
diffusions whose volatilities differ from one species to another, as for
Langevin replicas at distinct temperatures. The species carried by a
replica can then be read from the quadratic variation of its path over
an arbitrarily short time, so the past of the positions reveals the
current label. Since different species move a replica differently over
short times, the label also changes the future law of the positions, and
when the positions have positive densities the remaining condition is
automatic. For two Brownian species these steps are carried out
explicitly in Corollary~\ref{cor:bench-non-markov}.

	For jump processes with bounded rates, such as spin systems under
Glauber dynamics or lattice walkers, the past of the positions need not
determine the label, but Proposition~\ref{prop:filter} can be applied
directly. The past carries partial information on the label, which
changes with every jump and with the time elapsed between jumps. When
the species jump at different rates, two histories that end at the same
configuration give in general different conditional laws of the label.
Since the label also changes the rate of the next jump, the conditional
law of the future then depends on more than the current position, and
the switching positions are not Markov.

	\subsection{The coupling, the quotient structure and invariant laws}\label{ssec:sym}

	In \cite[Section~2.4]{dupuis2012} the two descriptions are related at the
	level of empirical measures, the empirical measure of the process with
	exchanged positions having the same law as that of the process with
	exchanged temperatures read through its current label, and the same
	statement underlies \cite{doll2018,lu2019}.
	Theorem~\ref{thm:N-symmetrisation}\,\textup{(i)} extends this equality to the
	level of paths, at every exchange rate and without reversibility, with both
	processes realised on one probability space and driven by one clock. In
	particular, codes that exchange configurations and codes that exchange
	temperatures, the latter read through the label, produce estimators with the same law. The path-level relation
	also gives access to what the empirical measure cannot see, such as the
	joint behaviour of an estimator and of a functional of the label, for
	instance the occupation of the temperatures by a replica \cite{doll2012}.
	It is also the basis for the comparisons of observables, decay rates and
	fast-exchange limits carried out in the rest of this section and in
	Section~\ref{sec:fast}.
	
	For $\theta\in\mathfrak{S}_N$ we write $f\circ\theta$ for the function
	$\mathbf{x}\mapsto f(\theta\mathbf{x})$ and $\theta_*\rho$ for the push-forward
	of a measure $\rho$ under $\mathbf{x}\mapsto\theta\mathbf{x}$. A function on
	$E^N$ is \emph{permutation invariant} if $f\circ\theta=f$ for every
	$\theta$, and a measure is \emph{exchangeable} if $\theta_*\rho=\rho$ for every
	$\theta$.
	
\begin{theorem}[Symmetrisation identity]
	\label{thm:N-symmetrisation}
	Let $(\widehat{\mathbf{X}}_t,\Lambda_t)$ be the switching process of
	Definition~\ref{def:switching} started at $(\mathbf{x}_0,\theta_0)$.
	\begin{enumerate}
		\item[\textup{(i)}] Define $\mathbf{X}_t:=\Lambda_t\widehat{\mathbf{X}}_t$,
		i.e.\ $X^a_t=\widehat X^{\Lambda_t^{-1}(a)}_t$. Then $\mathbf{X}$ is a
		swapping process started at $\theta_0\mathbf{x}_0$ and driven by the same
		exchange clock as $(\widehat{\mathbf{X}},\Lambda)$. In particular, it has
		the law of the swapping process started at $\theta_0\mathbf{x}_0$.
		\item[\textup{(ii)}] Consequently, writing $P_t(\mathbf{x}_0,\cdot)$
		for the transition kernel of the swapping process and
		\begin{equation}\label{eq:sector-measure}
			Q^\theta_t(\mathbf{x}_0,B)
			:=\PP_{(\mathbf{x}_0,\mathrm{id})}
			\bigl(\widehat{\mathbf{X}}_t\in B,\ \Lambda_t=\theta\bigr),
			\qquad \theta\in\mathfrak{S}_N,
		\end{equation}
		for the sub-probability measures carried by the sectors $E^N\times\{\theta\}$
		of the switching process, one has, for every $t>0$ and every
		$\mathbf{x}_0\in E^N$,
		\begin{equation}\label{eq:N-sym-kernel}
			P_t(\mathbf{x}_0,\cdot)
			=\sum_{\theta\in\mathfrak{S}_N}\theta_*Q^\theta_t(\mathbf{x}_0,\cdot).
		\end{equation}
	\end{enumerate}
\end{theorem}

\begin{proof}
	Let the switching process be driven by the exchange clock of
	Definition~\ref{def:clock}. At an accepted ring with mark
	$(\{a,b\},u)$ the label becomes $\Lambda\tau_{kl}=\tau_{ab}\Lambda$,
	with $k=\Lambda^{-1}(a)$ and $l=\Lambda^{-1}(b)$, because
	$\theta\tau_{kl}\theta^{-1}=\tau_{\theta(k)\theta(l)}$.
	
	Between rings, $\mathbf{X}=\Lambda\widehat{\mathbf{X}}$ lists the
	positions of the replicas in the order of the species they carry,
	$X^m=\widehat X^{\Lambda^{-1}(m)}$. Given the clock and the past up to the
	last ring $T_n$, the label is fixed and the coordinates of $\widehat{\mathbf{X}}$
	are independent copies of the species they carry, so $\mathbf{X}$ is there
	a decoupled process with species $m$ on coordinate $m$, started at
	$\mathbf{X}_{T_n}$ and independent of the clock and of the past given
	$\mathbf{X}_{T_n}$. At an accepted ring $\widehat{\mathbf{X}}$ does not move
	and $\Lambda$ becomes $\tau_{ab}\Lambda$, so $\mathbf{X}$ jumps to
	$\tau_{ab}\mathbf{X}_{T_n-}$, and it does so when
	$u\le\gamma_{ab}(\mathbf{X}_{T_n-})$. Hence $\mathbf{X}$ is obtained from
	the decoupled process by the construction of Definition~\ref{def:N-swap},
	with the same clock and $\mathbf X_0=\theta_0\mathbf x_0$, that is, it is a
	swapping process, and by Theorem~\ref{thm:exist} it has the law of the
	swapping process started at $\theta_0\mathbf{x}_0$. This is~(i). For~(ii),
	take $\theta_0=\mathrm{id}$ and decompose the law of $\mathbf{X}_t$ over
	the value of $\Lambda_t$. For a Borel set $B$,
	\[
	\begin{aligned}
	P_t(\mathbf{x}_0,B)
	&=\PP\bigl(\mathbf{X}_t\in B\bigr)
	=\sum_{\theta\in\mathfrak{S}_N}
	\PP\bigl(\theta\widehat{\mathbf{X}}_t\in B,\ \Lambda_t=\theta\bigr)\\
	&=\sum_{\theta\in\mathfrak{S}_N}
	Q^\theta_t\bigl(\mathbf{x}_0,\theta^{-1}B\bigr)
	=\sum_{\theta\in\mathfrak{S}_N}
	\theta_*Q^\theta_t(\mathbf{x}_0,B),
	\end{aligned}
	\]
	which is \eqref{eq:N-sym-kernel}.
\end{proof}

	When $E=\RR^d$, each $\theta$ acts on $(\RR^d)^N$ as a linear bijection with
	$|\det\theta|=1$, and every summand of \eqref{eq:N-sym-kernel} is dominated by
	$P_t(\mathbf x_0,\cdot)$. Hence $P_t(\mathbf x_0,\cdot)$ has a density $p$ if
	and only if every $Q^\theta_t(\mathbf x_0,\cdot)$ has a density $p_\theta$,
	and then
	\begin{equation}\label{eq:N-sym-identity}
		p(t,\mathbf{x}_0,\mathbf{y})
		=\sum_{\theta\in\mathfrak{S}_N}
		p_\theta\bigl(t,\mathbf{x}_0,\theta^{-1}\mathbf{y}\bigr)
	\end{equation}
	for Lebesgue-almost every $\mathbf{y}$. For two Brownian species this is the
	identity \eqref{eq:density-explicit} of Proposition~\ref{prop:density}.
	
	\begin{corollary}\label{cor:N-symmetric-obs}
		Let $F:E^N\to\RR$ be permutation invariant. On the probability
		space of the switching process, with the swapping process realised as
		$\mathbf{X}_t=\Lambda_t\widehat{\mathbf{X}}_t$ as in
		Theorem~\ref{thm:N-symmetrisation}\,\textup{(i)},
		\begin{equation}\label{eq:sym-pathwise}
			F(\mathbf{X}_t)=F\bigl(\widehat{\mathbf{X}}_t\bigr)
			\qquad\text{for every }t\ge0\text{ and every }\omega .
		\end{equation}
		Consequently the scalar processes $(F(\mathbf{X}_t))_{t\ge0}$ and
		$(F(\widehat{\mathbf{X}}_t))_{t\ge0}$ have the same law, for the swapping process started at
		$\theta_0\mathbf x_0$ and the switching process started at
		$(\mathbf x_0,\theta_0)$.
	\end{corollary}
	
	\begin{proof}
		Since $F$ is invariant under every permutation and $\Lambda_t$ takes
		values in $\mathfrak{S}_N$, we have
		$F(\Lambda_t\widehat{\mathbf{X}}_t)=F(\widehat{\mathbf{X}}_t)$
		pointwise in $(\omega,t)$. The statement in
		law follows because, by Theorem~\ref{thm:N-symmetrisation}\,\textup{(i)},
		$\Lambda\widehat{\mathbf{X}}$ and $\mathbf{X}$ have the same law as
		processes.
	\end{proof}
	
	The two descriptions therefore share the law of the whole scalar process
	$t\mapsto F(\cdot_t)$, not only its one-time laws. Exit times of
	permutation-invariant domains, occupation times of symmetric sets and
	running extrema of symmetric functions, such as the largest
	inter-particle distance, may be studied on the swapping process alone,
	without tracking the hidden permutation.
	
			The label of Theorem~\ref{thm:N-symmetrisation} depends on how the
		replicas are numbered. Renumbering them changes $\theta$ but not which
		species sits where, and the swapping process is what remains of the
		switching process once this numbering is forgotten
		(Proposition~\ref{prop:markov-unique}).
	
	The group $\mathfrak{S}_N$ acts on $E^N\times\mathfrak{S}_N$ by
	renumbering the replicas,
	\begin{equation}\label{eq:replica-action}
		\varpi\cdot(\mathbf{x},\theta):=\bigl(\varpi\mathbf{x},\ \theta\varpi^{-1}\bigr),
		\qquad\varpi\in\mathfrak{S}_N ,
	\end{equation}
		so that replica $\varpi(k)$ now holds the position $x^k$ and the species
	$\theta(k)$. We write $\Phi(\mathbf{x},\theta):=\theta\mathbf{x}$ for the
	positions listed by species, in terms of which
	Theorem~\ref{thm:N-symmetrisation}\,\textup{(i)} reads
	$\mathbf X_t=\Phi(\widehat{\mathbf X}_t,\Lambda_t)$.
	 
		\begin{proposition}\label{prop:markov-unique}
		\begin{enumerate}
			\item[\textup{(i)}] For every $\varpi\in\mathfrak S_N$, if
			$(\widehat{\mathbf{X}},\Lambda)$ starts at $(\mathbf{x},\theta)$, then
			$(\varpi\widehat{\mathbf{X}},\Lambda\varpi^{-1})$ is the switching
			process started at $\varpi\cdot(\mathbf{x},\theta)$. In other words, the
			law of the switching process is invariant under renumbering of the
			replicas.
			\item[\textup{(ii)}] A function on $E^N\times\mathfrak{S}_N$ is
			invariant under renumbering if and only if it is of the form
			$g\circ\Phi$ for a function $g$ on $E^N$.
		\end{enumerate}
	\end{proposition}
	
	\begin{proof}
		(i) Renumbering the replicas maps the construction in the proof of
		Theorem~\ref{thm:N-symmetrisation} to itself. Between rings,
		$\varpi\widehat{\mathbf{X}}$ is the decoupled process in which replica
		$\varpi(k)$ carries species $\theta(k)=(\theta\varpi^{-1})(\varpi(k))$. A
		ring with mark $(\{a,b\},u)$ names two species. It exchanges $a$ and $b$
		between the replicas that carry them, whatever their numbers, and the
		acceptance condition $u\le\gamma_{ab}(\Lambda\widehat{\mathbf{X}})$ is
		unchanged because
		$(\Lambda\varpi^{-1})(\varpi\widehat{\mathbf{X}})=\Lambda\widehat{\mathbf{X}}$.
		
				(ii) $\Phi(\varpi\cdot(\mathbf{x},\theta))=\theta\varpi^{-1}\varpi\mathbf{x}
		=\theta\mathbf{x}$, so every $g\circ\Phi$ is invariant. Conversely, if
		$G$ is invariant, taking $\varpi=\theta$ gives
		$G(\mathbf{x},\theta)=G(\theta\mathbf{x},\mathrm{id})$, that is,
		$G=g\circ\Phi$ with $g:=G(\cdot,\mathrm{id})$.
	\end{proof}
	
		Part~\textup{(ii)} says that $\Phi$ is the quotient map of the action
	\eqref{eq:replica-action}, its fibres being exactly the orbits, as one
	sees by taking for $G$ the indicator of an orbit. Together with
	part~\textup{(i)} and Theorem~\ref{thm:N-symmetrisation}\,\textup{(i)},
	this is the sense in which swapping is the quotient of switching.
	
		If the switching process is in equilibrium, so is its image under $\Phi$.
	Conversely, the construction of Theorem~\ref{thm:N-symmetrisation} can
	be run backwards, building the switching process from the swapping
	process and an independent, uniform numbering of the replicas, and in
	equilibrium this produces every invariant law of the swapping process as
	such an image.
	
	\begin{proposition}\label{prop:invariant}
		\begin{enumerate}
			\item[\textup{(i)}] If $\tilde\pi$ is an invariant probability of the
			switching process, then $\Phi_*\tilde\pi$ is an invariant probability
			of the swapping process.
			\item[\textup{(ii)}] Let $\pi$ be an invariant probability of the
			swapping process, and let the switching process start from
			$(\theta_0^{-1}\mathbf X_0,\theta_0)$, where $\mathbf X_0\sim\pi$ and
			$\theta_0$ is uniform on $\mathfrak S_N$ and independent of
			$\mathbf X_0$. Then at every time $t\ge0$ the switching process has the
			law
			\begin{equation}\label{eq:lift}
				\tilde\pi\bigl(B\times\{\theta\}\bigr):=\pi(\theta B)/N!,
				\qquad\theta\in\mathfrak S_N,
			\end{equation}
			which is therefore invariant, with $\Phi_*\tilde\pi=\pi$. Moreover,
			for every $t\ge0$ the label $\Lambda_t$ is uniform on $\mathfrak S_N$
			and independent of the path of the swapping process
			$\Lambda\widehat{\mathbf X}$.
		\end{enumerate}
		In particular, every invariant probability of the swapping process is
		the image under $\Phi$ of an invariant probability of the switching
		process. We call $\tilde\pi$ in \textup{(ii)} the \emph{lift} of $\pi$.
	\end{proposition}
	
	\begin{proof}
		(i) Let the switching process start from $\tilde\pi$. By
		Theorem~\ref{thm:N-symmetrisation}\,\textup{(i)}, applied conditionally on
		the initial point (Theorem~\ref{thm:exist}),
		$\Lambda\widehat{\mathbf X}$ is the swapping process started from
		$\Phi(\widehat{\mathbf X}_0,\Lambda_0)$, whose law is $\Phi_*\tilde\pi$,
		and at every time $t$ it has the law of
		$\Phi(\widehat{\mathbf X}_t,\Lambda_t)$, which is $\Phi_*\tilde\pi$ by
		the invariance of $\tilde\pi$.
		
		(ii) Let $\mathbf X$ be the swapping process with $\mathbf X_0\sim\pi$, driven
		by the exchange clock of Definition~\ref{def:clock}, and let $\theta_0$ be
		uniform on $\mathfrak S_N$ and independent of the clock and of $\mathbf X$.
		Set $\Lambda_t:=\varpi_t\theta_0$, where $\varpi_t\in\mathfrak S_N$ is the
		product of the transpositions of the rings accepted up to time $t$, the
		latest on the left, and $\widehat{\mathbf X}_t:=\Lambda_t^{-1}\mathbf X_t$.
		Reading the proof of Theorem~\ref{thm:N-symmetrisation} backwards,
		$(\widehat{\mathbf X},\Lambda)$ is the switching process started at
		$(\theta_0^{-1}\mathbf X_0,\theta_0)$. This step uses neither the
		invariance of $\pi$ nor the law of $\theta_0$, so the swapping process
		started at $\mathbf x$ and a label $\theta_0$ give in the same way the
		switching process started at $(\theta_0^{-1}\mathbf x,\theta_0)$. Since
		$\varpi_t$ is a functional of the clock and of the path of $\mathbf X$, and
		$\theta_0$ is uniform and independent of both, $\Lambda_t$ is uniform on
		$\mathfrak S_N$ and independent of the path of $\mathbf X$. As $\pi$ is
		invariant, $\mathbf X_t\sim\pi$, so $(\widehat{\mathbf X}_t,\Lambda_t)
		=(\Lambda_t^{-1}\mathbf X_t,\Lambda_t)$ has law $\tilde\pi$ for every $t$.
		Finally $\Phi_*\tilde\pi(B)=\sum_\theta\tilde\pi(\theta^{-1}B\times\{\theta\})
		=\pi(B)$.
	\end{proof}
	
		Along a path of the swapping process, record which replica occupies each
	coordinate. The exchanges move these identities together with the
	positions but never depend on them, so an assignment that is initially
	uniform and independent of the configuration stays so at every time,
	which is why the label is uniform under the lift. 	Each replica then carries each species with probability $1/N$ at every
	time, and its expected occupation time of each species over $[0,T]$ is
	$T/N$. This is the equal-occupancy value of \cite{doll2012}, where the
	fraction of time a single replica spends at each temperature is shown to
	approach $1/N$ in an ergodic run, and the rate of this approach,
	measured through an occupation entropy, is used to compare the
	performance of tempering methods.
	
	\subsection{Consequence for decay rates}\label{ssec:spectral}
	
	One of the quantities most often invoked in the analysis of parallel tempering is the spectral gap
	\cite{leeshen2023,madraszheng2003,woodard2009,woodard2009torpid,zheng2003}
	as a measure of the rate of convergence to equilibrium. The quotient structure
	of Section~\ref{ssec:sym} carries over to it. In functional
	form, Theorem~\ref{thm:N-symmetrisation}\,\textup{(i)} says that started
	from $(\mathbf x,\theta)$ the image $\Phi(\widehat{\mathbf X}_t,\Lambda_t)$
	has the law of the swapping process started at $\theta\mathbf x$, that is,
	\begin{equation}\label{eq:intertwining}
		P_t^{\rs}(f\circ\Phi)(\mathbf{x},\theta)=(P_t f)(\theta\mathbf{x})
	\end{equation}
	for every bounded measurable $f$ on $E^N$, every $t\ge0$ and every
	$(\mathbf x,\theta)$, where $(P_t)$ is the swapping semigroup and
	$(P_t^{\rs})$ the switching semigroup (the subscript $\rs$ stands for regime
	switching). This intertwining is Dynkin's criterion
	\cite[Theorem~10.13]{dynkin1965}, \cite{rogerspitman1981} for the quotient
	map $\Phi$, under which $\Phi(\widehat{\mathbf X}_t,\Lambda_t)$ is a Markov
	process with semigroup $(P_t)$ for every initial law of the switching
	process, and the decay rates of the two processes are compared through it.
	
	We say that a Markov semigroup $(P_t)$ with invariant probability $\pi$
	\emph{decays at rate} $\lambda>0$ in $L^2(\pi)$ if there is a constant
	$C\ge1$ such that
	$\|P_tf-\pi(f)\|_{L^2(\pi)}\le Ce^{-\lambda t}\|f-\pi(f)\|_{L^2(\pi)}$ for
	all $f\in L^2(\pi)$ and $t\ge0$, and we call the supremum of such $\lambda$
	its \emph{best decay rate}, with the convention that the supremum of the empty
	set is $0$. In the reversible case it is the spectral gap
	\cite[Theorem~4.2.5]{bakry2014}.
	
	Fix an invariant probability $\tilde\pi$ of the switching process and put
	$\pi_\Phi:=\Phi_*\tilde\pi$, which is invariant for the swapping process by
	Proposition~\ref{prop:invariant}\,\textup{(i)}. The comparison of the decay rates rests on the observables of the
	switching process that do not depend on the numbering of the replicas, which
	by Proposition~\ref{prop:markov-unique}\,\textup{(ii)} are the functions of
	$\theta\mathbf{x}$. We set
	\begin{equation}\label{eq:H-Phi}
		H_\Phi:=\bigl\{g\circ\Phi:\ g\in L^2(\pi_\Phi)\bigr\}\subset L^2(\tilde\pi),
		\qquad H_\perp:=H_\Phi^{\perp}.
	\end{equation}
	Since $\|g\circ\Phi\|_{L^2(\tilde\pi)}=\|g\|_{L^2(\pi_\Phi)}$, the subspace
	$H_\Phi$ is closed and isometric to $L^2(\pi_\Phi)$. It contains the
	constants, so every element of $H_\perp$ has mean zero under $\tilde\pi$. We
	write $\Pi_\Phi$ and $\Pi_\perp$ for the two orthogonal projections, and
	$L^2_0$ for the functions of mean zero, which we call centred, under the
	relevant invariant probability. The \emph{compression} of the switching
	semigroup to $H_\perp$ is the family of operators on $H_\perp$
	\begin{equation}\label{eq:compression}
		S^\perp_t:=\Pi_\perp P^\rs_t\Pi_\perp ,\qquad t\ge0 .
	\end{equation}
	
	Since $H_\Phi$ is the space of square-integrable functions of the swapping
	coordinate $\Phi$, the subspace $H_\perp$ consists of the
	$F\in L^2(\tilde\pi)$ with $\EE_{\tilde\pi}[F\mid\Phi]=0$, the observables
	with zero conditional mean given the swapping coordinate. 
	
	\begin{lemma}
		\label{lem:quotient-semigroup}
		Let $\tilde\pi$ be an invariant probability of the switching process, and
		let $\pi_\Phi$, $H_\Phi$, $H_\perp$ and $(S^\perp_t)$ be as in
		\eqref{eq:H-Phi} and \eqref{eq:compression}. Then
		\begin{enumerate}
			\item[\textup{(i)}] the switching semigroup acts on $L^2(\tilde\pi)$ by
			contractions, it leaves $H_\Phi$ invariant, and its restriction to
			$H_\Phi$ is conjugate to the swapping semigroup on $L^2(\pi_\Phi)$
			through the isometry $g\mapsto g\circ\Phi$;
			\item[\textup{(ii)}] $(S^\perp_t)$ is a semigroup on $H_\perp$.
		\end{enumerate}
	\end{lemma}
	
	\begin{proof}
		By Jensen's inequality and the invariance of $\tilde\pi$ and of $\pi_\Phi$
		(Proposition~\ref{prop:invariant}\,\textup{(i)}),
		$\|P^\rs_tF\|^2_{L^2(\tilde\pi)}\le\tilde\pi(P^\rs_tF^2)=\|F\|^2_{L^2(\tilde\pi)}$,
		and in the same way $P_t$ is a contraction on $L^2(\pi_\Phi)$. The identity
		$P^\rs_t(f\circ\Phi)=(P_tf)\circ\Phi$ of \eqref{eq:intertwining}
		holds for bounded $f$, and both sides are continuous in $f$ for the norm of
		$L^2(\pi_\Phi)$, so it extends to $L^2(\pi_\Phi)$. This gives the invariance of
		$H_\Phi$ and the conjugacy. Finally
		$S^\perp_{t+u}=\Pi_\perp P^\rs_t(\Pi_\Phi+\Pi_\perp)P^\rs_u\Pi_\perp=S^\perp_tS^\perp_u$,
		because $P^\rs_tH_\Phi\subset H_\Phi$ gives $\Pi_\perp P^\rs_t\Pi_\Phi=0$.
	\end{proof}
	
	We say that $(S^\perp_t)$ decays at rate $\lambda$ on $H_\perp$ if
	$\|S^\perp_th\|\le Ce^{-\lambda t}\|h\|$ for some $C\ge1$, all
	$h\in H_\perp$ and $t\ge0$, and we write $\lambda_\perp$ for the supremum
	of such $\lambda$, with $\lambda_\perp=+\infty$ when $H_\perp=\{0\}$.
	
	\begin{theorem}
		\label{thm:gap-decomposition}
		In the setting of Lemma~\ref{lem:quotient-semigroup}, let $\lambda_\rs$ be
		the best decay rate of the switching semigroup in $L^2(\tilde\pi)$,
		$\lambda_\swap$ that of the swapping semigroup in $L^2(\pi_\Phi)$, and
		$\lambda_\perp$ that of the semigroup $(S^\perp_t)$ on $H_\perp$.
		\begin{enumerate}
			\item[\textup{(i)}] If the switching semigroup decays at rate $\lambda$
			in $L^2(\tilde\pi)$ with constant $C$, the swapping semigroup decays at
			rate $\lambda$ in $L^2(\pi_\Phi)$ with the same constant.
			\item[\textup{(ii)}] $\lambda_\rs=\min(\lambda_\swap,\lambda_\perp)$.
		\end{enumerate}
	\end{theorem}
	
	\begin{proof}
		(i) For $g\in L^2(\pi_\Phi)$ centred, $g\circ\Phi$ is centred and has the
		same norm, so by Lemma~\ref{lem:quotient-semigroup}
		\[
		\|P_tg\|_{L^2(\pi_\Phi)}=\|P^\rs_t(g\circ\Phi)\|_{L^2(\tilde\pi)}
		\le Ce^{-\lambda t}\|g\|_{L^2(\pi_\Phi)} .
		\]
		
		(ii) By (i), $\lambda_\rs\le\lambda_\swap$, and
		$\|S^\perp_th\|\le\|P^\rs_th\|$ for $h\in H_\perp$ gives
		$\lambda_\rs\le\lambda_\perp$. For the converse, there is nothing to
		prove if $\min(\lambda_\swap,\lambda_\perp)=0$, so assume it is positive
		and fix $0<\lambda<\lambda'<\min(\lambda_\swap,\lambda_\perp)$. Write $S^\Phi_t$
		for the restriction of $P^\rs_t$ to $H_\Phi\cap L^2_0$, which decays at
		rate $\lambda'$ with some constant $C_\Phi$ by the conjugacy of
		Lemma~\ref{lem:quotient-semigroup}, and let $C_\perp$ be a constant for
		which $S^\perp_t$ decays at rate $\lambda'$. Put
		$\mathcal R_t:=\Pi_\Phi P^\rs_t\Pi_\perp$. It takes values in
		$H_\Phi\cap L^2_0$, because
		$\langle\mathcal R_th,1\rangle=\langle P^\rs_th,\Pi_\Phi1\rangle
		=\tilde\pi(P^\rs_th)=\tilde\pi(h)=0$ for $h\in H_\perp$, and
		$\|\mathcal R_t\|\le1$. Applying $P^\rs_{t+u}$ to $h\in H_\perp$ and
		projecting on $H_\Phi$ gives
		$\mathcal R_{t+u}=S^\Phi_t\mathcal R_u+\mathcal R_tS^\perp_u$, and by
		induction, for a step $r>0$ and $n\ge1$,
		\[
		\mathcal R_{nr}=\sum_{k=0}^{n-1}S^\Phi_{(n-1-k)r}\,\mathcal R_r\,S^\perp_{kr},
		\qquad
		\|\mathcal R_{nr}\|\le n\,C_\Phi C_\perp\,e^{-\lambda'(n-1)r}.
		\]
		For $t=nr+u$ with $u\in[0,r)$,
		$\mathcal R_t=S^\Phi_u\mathcal R_{nr}+\mathcal R_uS^\perp_{nr}$, so
		$\|\mathcal R_t\|\le C'(1+t)e^{-\lambda't}$ with $C'$ depending on
		$r,C_\Phi,C_\perp,\lambda'$. A centred $F\in L^2(\tilde\pi)$ splits as
		$F=F_\Phi+F_\perp$ with $F_\Phi\in H_\Phi\cap L^2_0$, and
		$P^\rs_tF=S^\Phi_tF_\Phi+\mathcal R_tF_\perp+S^\perp_tF_\perp$. Hence
		$\|P^\rs_tF\|\le C''(1+t)e^{-\lambda't}\|F\|\le C'''e^{-\lambda t}\|F\|$,
		so $\lambda$ is a decay rate of the switching semigroup. As $\lambda$
		was arbitrary, $\lambda_\rs\ge\min(\lambda_\swap,\lambda_\perp)$.
	\end{proof}
	
	Theorem~\ref{thm:gap-decomposition} does not require $\tilde\pi$ to be the
	lift of $\pi_\Phi$, but by Proposition~\ref{prop:invariant}\,\textup{(ii)} every
	invariant probability $\pi$ of the swapping process is of the form
	$\pi_\Phi$, with $\tilde\pi$ its lift \eqref{eq:lift}. Under the lift the
	label is uniform and independent of the swapping coordinate, so every
	centred function of the label is orthogonal to $H_\Phi$, and the relaxation
	of the label is one of the modes of $H_\perp$, though not the only one.
	
	A case in which the separation of Theorem~\ref{thm:gap-decomposition} can be
	seen in an explicit way is that of $N$ identical species exchanged at
	constant rates. For constants $\gamma_{ab}\ge0$, let $L_\gamma$ be the
	weighted Laplacian of $G_\gamma$, $(L_\gamma)_{ab}=-\gamma_{ab}$ for $a\ne b$ and
	$(L_\gamma)_{aa}=\sum_{b\ne a}\gamma_{ab}$, and $\mathrm{gap}(L_\gamma)$ the
	smallest non-zero eigenvalue of $L_\gamma$.

	\begin{corollary}\label{cor:identical}
	Let all species be the same Markov family, with invariant probability
	$\pi_1$ and best decay rate $\lambda_1$ in $L^2(\pi_1)$, and let the rates
	be constants $\gamma_{ab}$, with $G_\gamma$ connected. The probability
	$\pi_0:=\pi_1^{\otimes N}$ is invariant for the swapping process, and its lift \eqref{eq:lift} is
	$\tilde\pi=\pi_0\otimes\mathrm{Unif}(\mathfrak S_N)$, with
	$\pi_\Phi=\pi_0$. Let $\lambda_\rs$, $\lambda_\swap$ and $\lambda_\perp$
	be as in Theorem~\ref{thm:gap-decomposition} for this $\tilde\pi$. Then
	\[
	\lambda_\swap=\lambda_1,
	\qquad
	\lambda_\rs=\min\bigl(\lambda_1,\,\mathrm{gap}(L_\gamma)\bigr).
	\]
	In particular $\lambda_\rs<\lambda_\swap$ when
	$\mathrm{gap}(L_\gamma)<\lambda_1$, and then
	$\lambda_\perp=\mathrm{gap}(L_\gamma)$.
\end{corollary}
	
	\begin{proof}
		See Appendix~\ref{app:identical}.
	\end{proof}
	
	The paths of the swapping process jump at the exchanges, but its semigroup
	decays at the rate of a single species, so $\lambda_\swap=\lambda_1$ does
	not depend on the exchange rates. Since multiplying all the rates by a
	factor multiplies $\mathrm{gap}(L_\gamma)$ by the same factor, slow
	exchanges, with $\mathrm{gap}(L_\gamma)\le\lambda_1$, give
	$\lambda_\rs=\mathrm{gap}(L_\gamma)$, and the switching process relaxes no
	faster than its label, while $\lambda_\rs=\lambda_\swap$ as soon as
	$\mathrm{gap}(L_\gamma)\ge\lambda_1$. For identical species, the functions
	of the label alone, on which the occupation-based diagnostics of \cite{doll2012}
	and the round trips of the replicas through the temperatures
	\cite{syed2022} are built, relax at the rate $\mathrm{gap}(L_\gamma)$, which
	carries no information on the rate $\lambda_1$ at which the observables of
	the swapping process relax.
	
	\begin{remark}[Discrete time]\label{rem:discrete}
		Although continuous time opens lines of analysis such as the fast-exchange
		limit of Section~\ref{sec:fast}, most results of Section~\ref{sec:setup}
		also hold in discrete time. Let the species be Markov
		kernels $P_1,\dots,P_N$ on $E$. At each step the swapping chain moves coordinate $a$ by $P_a$, independently, then draws a pair $\{a,b\}$ with probability $q_{ab}$, where $\sum_{a<b}q_{ab}\le1$ and no pair is drawn with the remaining probability, and exchanges the coordinates of the drawn pair with probability $\gamma_{ab}(\mathbf X)\in[0,1]$. The switching chain moves
		replica $k$ by
		$P_{\Lambda(k)}$ and then, for the drawn pair $\{a,b\}$, replaces $\Lambda$ by $\tau_{ab}\Lambda$ with
		probability $\gamma_{ab}(\Lambda\widehat{\mathbf X})$. If the move of
		species $a$ is realised by the same random map in both chains, applied to
		coordinate $a$ and to replica $\Lambda^{-1}(a)$, and the pair and the
		acceptance are shared, then $\mathbf X_n=\Lambda_n\widehat{\mathbf X}_n$
		at every step by the computation in the proof of
		Theorem~\ref{thm:N-symmetrisation}. Corollary~\ref{cor:N-symmetric-obs} and
		Propositions~\ref{prop:markov-unique} and~\ref{prop:invariant} follow as
		before, and
		Lemma~\ref{lem:quotient-semigroup} and
		Theorem~\ref{thm:gap-decomposition} hold with the powers $P^n$ in place of
		$P_t$ and decay rates defined by geometric decay in $n$.
	\end{remark}
	
	\section{The fast-exchange limit}\label{sec:fast}
	
	We now replace the rates $\gamma_{ab}$ by $s\gamma_{ab}$ and let
	$s\to\infty$, as the practice of frequent exchange suggests
	\cite{dupuis2012,rosta2009}, and write $\widehat{\mathbf X}^{(s)}$,
	$\Lambda$ and $\mathbf X^{(s)}$ for the corresponding processes, the label
	depending on $s$ as well. The rates $s\gamma_{ab}$ are realised by the clock of
	Definition~\ref{def:clock} run at rate $s\Gamma_N$, with the same marks.
	Sections~\ref{ssec:fast-hyp} and~\ref{ssec:fast-ass} introduce the averaged
	dynamics and the assumptions, and Section~\ref{ssec:fast} identifies the
	limits of both descriptions,
	their Markov structure and when they recover the targets.
	
	\subsection{The frozen exchange chain and the averaged generator}\label{ssec:fast-hyp}

	Let $\mathcal A$ be an operator that maps a class of test functions into
	bounded measurable functions. A \emph{solution of the martingale problem}
	for $\mathcal A$ is a process $Y$ with paths in $D([0,\infty),\mathsf F)$, where $\mathsf F$ is $E$ or $E^N$, adapted
	to a filtration, such that, for every test function $f$,
	$f(Y_t)-f(Y_0)-\int_0^t\mathcal Af(Y_u)\dif u$ is a martingale. The
	problem is \emph{well posed} if, for every initial distribution, it has a
	solution and all its solutions have the same law. This is the formulation
	of \cite[Chapter~4]{ethierkurtz1986}.
	
	From now on we fix a class $\mathcal D\subset C_c(E)$ that is closed under
	products, contains a countable family separating the points of $E$, and
	contains, for every compact $K$ and open $U\supset K$, a function $g$ with
	$0\le g\le1$, $g=1$ on $K$ and support in $U$. On $E=\RR^d$ we take
	$\mathcal D=C_c^2(\RR^d)$, and on a countable $E$ the functions of finite
	support.
	
	\begin{definition}\label{def:generator}
		A linear operator $A$ from $\mathcal D$ into the bounded measurable
		functions on $E$ is a \emph{generator} of a species $(\PP_x)_{x\in E}$
		if, for every $x\in E$, the law $\PP_x$ is a solution of the martingale
		problem for $A$ started at $x$.
	\end{definition}
	
	The species are given together with generators $A_1,\dots,A_N$, where $A_m$
	is a generator of species $m$. The \emph{test functions} on $E^N$ form the
	space $\mathcal D_N$ spanned by the constants and by the products
	$\bigotimes_kf_k$ with $f_k\in\mathcal D$, and $(A_m)_{(k)}$ denotes $A_m$
	acting on the coordinate $k$, with $(A_m)_{(k)}1:=0$.
	
	If the positions are frozen at $\mathbf y\in E^N$ in
	Definition~\ref{def:switching}, only the label moves. It is then a Markov
	chain on $\mathfrak S_N$, which we call the \emph{exchange chain frozen at}
	$\mathbf y$, and at unit speed its generator is
	\begin{equation}\label{eq:frozen-gen}
		(\mathcal Q_{\mathbf y}g)(\theta)=\sum_{a<b}\gamma_{ab}(\theta\mathbf y)
		\bigl[g(\tau_{ab}\theta)-g(\theta)\bigr],\qquad g:\mathfrak S_N\to\RR .
	\end{equation}
	The chain jumps from $\theta$ to $\tau_{ab}\theta$ at rate
	$\gamma_{ab}(\theta\mathbf y)$, the rate evaluated at the positions
	$\theta\mathbf y$ of the species. It is therefore a nearest-neighbour random
	walk, which we also call the frozen walk, on the Cayley graph of
	$\mathfrak S_N$ generated by the transpositions $\tau_{ab}$ with $ab$ an
	edge of $G_\gamma$. It depends on the rates and on $\mathbf y$, but not on
	the dynamics of the species.
	
	A set of transpositions generates $\mathfrak S_N$ if and only if the graph
	formed by their pairs is connected \cite[Section~3.10]{godsilroyle2001}.
	Hence, when $G_\gamma$ is connected and $\gamma_{ab}>0$ on $E^N$ for every
	edge $ab$ of $G_\gamma$, the frozen walk is irreducible for every
	$\mathbf y$ and has a unique invariant probability $\nu_{\mathbf y}$, which
	charges every permutation. Replacing the rates by $s\gamma_{ab}$ multiplies
	the generator by $s$. This leaves $\nu_{\mathbf y}$ unchanged and only
	speeds up the relaxation to it, so with the positions held at $\mathbf y$
	the law of the label at any fixed positive time converges to
	$\nu_{\mathbf y}$ as $s\to\infty$.
	
	\begin{lemma}\label{lem:frozen}
		Let the rates $\gamma_{ab}:E^N\to[0,\Gamma_{ab}]$ be continuous and, for
		every $\mathbf y\in E^N$, let the exchange chain frozen at $\mathbf y$,
		with generator \eqref{eq:frozen-gen}, have a unique invariant probability
		$\nu_{\mathbf y}$ on $\mathfrak S_N$. Then
		$\mathbf y\mapsto\nu_{\mathbf y}$ is continuous,
		\begin{equation}\label{eq:nu-shift}
			\nu_{\sigma\mathbf y}(\theta)=\nu_{\mathbf y}(\theta\sigma),
			\qquad\sigma,\theta\in\mathfrak S_N ,
		\end{equation}
		and for every compact
		$B\subset E^N$ there is $\mathfrak d_B\in(0,1]$ such that, for all
		$\mathbf y\in B$, $\theta\in\mathfrak S_N$ and $r\ge0$,
		\begin{equation}\label{eq:frozen-mixing}
			\bigl\|e^{r\mathcal Q_{\mathbf y}}(\theta,\cdot)-\nu_{\mathbf y}\bigr\|_{TV}
			\le2(1-\mathfrak d_B)^{\lfloor r\rfloor}.
		\end{equation}
	\end{lemma}
	
	\begin{proof}
		A finite chain with a unique invariant probability has a single closed
		class, accessible from every state. For a chain in continuous time with
		generator $\mathcal Q$, the entry $e^{\mathcal Q}(\theta,\theta')$ is
		positive if and only if $\theta'$ is accessible from $\theta$. Hence every
		row of $e^{\mathcal Q_{\mathbf y}}$ charges every state of the closed class,
		any two rows overlap, and the Dobrushin coefficient
		$c(P):=\tfrac12\max_{\theta,\theta'}\|P(\theta,\cdot)-P(\theta',\cdot)\|_{TV}$
		of $e^{\mathcal Q_{\mathbf y}}$ is less than $1$ \cite{seneta2006}. With the
		normalisation of the Notation, the total variation of a difference of two
		probabilities is twice the usual distance, whence the factor $\tfrac12$ in
		$c$ and the factor $2$ in \eqref{eq:frozen-mixing}. Since the rates are
		continuous, the entries of $\mathcal Q_{\mathbf y}$ are continuous in $\mathbf y$, and the
		matrix exponential is continuous, so $c(e^{\mathcal Q_{\mathbf y}})$ is
		continuous and its maximum over $B$ is some $1-\mathfrak d_B<1$. The
		coefficient $c(P)$ is the contraction coefficient of $P$ for the total
		variation of differences of probabilities \cite{seneta2006}. Write
		$r=n+r'$ with $n=\lfloor r\rfloor$. Since $\nu_{\mathbf y}$ is invariant,
		$e^{r\mathcal Q_{\mathbf y}}(\theta,\cdot)-\nu_{\mathbf y}$ is the image
		under $e^{n\mathcal Q_{\mathbf y}}$ of the difference of probabilities
		$e^{r'\mathcal Q_{\mathbf y}}(\theta,\cdot)-\nu_{\mathbf y}$, whose total
		variation is at most $2$ and is contracted $n$ times, which is
		\eqref{eq:frozen-mixing}. Any limit point of $\nu_{\mathbf y_n}$
		with $\mathbf y_n\to\mathbf y$ is invariant for $\mathcal Q_{\mathbf y}$,
		because the entries of $\mathcal Q_{\mathbf y}$ are continuous, hence equal
		to $\nu_{\mathbf y}$. Finally, for $\sigma\in\mathfrak S_N$ the map
		$\theta\mapsto\theta\sigma$ carries the chain frozen at $\sigma\mathbf y$ to
		the chain frozen at $\mathbf y$, since both move from $\theta$,
		respectively $\theta\sigma$, by a left multiplication by $\tau_{ab}$ at rate
		$\gamma_{ab}(\theta\sigma\mathbf y)$. The image of $\nu_{\sigma\mathbf y}$ is
		therefore invariant for the chain frozen at $\mathbf y$, and uniqueness
		gives \eqref{eq:nu-shift}.
	\end{proof}
	
	The constant $\mathfrak d_B$ plays, for rates that depend on the positions,
	the role of $\mathrm{gap}(L_\gamma)$ in Corollary~\ref{cor:identical}, since
	it controls the relaxation of the frozen label uniformly over the positions
	in $B$.
	
	Write $\LL_\theta:=\sum_k(A_{\theta(k)})_{(k)}$ for the generator of the
	positions when replica $k$ carries species $\theta(k)$. When the frozen chain has a unique invariant probability $\nu_{\mathbf y}$, we set
	\begin{equation}\label{eq:barL-def}
		\bar\LL f(\mathbf y):=\sum_{\theta\in\mathfrak S_N}\nu_{\mathbf y}(\theta)\,
		\LL_\theta f(\mathbf y)
	\end{equation}
	for test functions $f$. 
With constant rates and $G_\gamma$ connected, the walk moves between
$\theta$ and $\tau_{ab}\theta$ at the same rate in both directions, so
$\nu_{\mathbf y}$ is uniform for every $\mathbf y$ and
\begin{equation}\label{eq:average-L}
	\bar\LL=\sum_{k=1}^N\bar A_{(k)},\qquad
	\bar A:=\frac1N\sum_{m=1}^NA_m ,
\end{equation}
that is, every replica follows the average of the species.

If the species have invariant probabilities $\pi_m$, their
\emph{targets}, replica exchange requires the swapping process to
preserve $\pi_0:=\bigotimes_m\pi_m$. The decoupled process preserves it,
so the requirement bears on the exchanges alone, whose generator is, for
bounded measurable $f$,
\begin{equation}\label{eq:S-op}
	\mathcal Sf(\mathbf x):=\sum_{a<b}\gamma_{ab}(\mathbf x)
	\bigl(f(\tau_{ab}\mathbf x)-f(\mathbf x)\bigr).
\end{equation}
A ring of the clock moves the swapping process by the Markov kernel
$\mathcal J:=I+\Gamma_N^{-1}\mathcal S$. We say that the exchanges
\emph{preserve} $\pi_0$ if $\pi_0(\mathcal Sf)=0$ for every bounded
measurable $f$. Then $\mathcal J$ preserves $\pi_0$, and so does the
swapping process at every $s$, its semigroup being a mixture of products
of the semigroup of the decoupled process and of $\mathcal J$.

When the $\pi_m$ have positive densities $h_m$ with respect to a
$\sigma$-finite measure on $E$, a change of variables shows that this
condition makes the weights $\theta\mapsto h_0(\theta\mathbf y)$,
$h_0:=\bigotimes_mh_m$, stationary for the frozen walk at almost every
$\mathbf y$, and at every $\mathbf y$ when $h_0$ and the $\gamma_{ab}$ are
continuous and the measure charges every open set. For $G_\gamma$
connected and $\gamma_{ab}>0$ on $E^N$ for every edge $ab$, this gives
\begin{equation}\label{eq:nu-db}
	\nu_{\mathbf y}(\theta)=\frac{h_0(\theta\mathbf y)}
	{\sum_{\theta'\in\mathfrak S_N}h_0(\theta'\mathbf y)} ,
	\qquad\theta\in\mathfrak S_N ,
\end{equation}
independently of the exchange graph, which is the weight of infinite
swapping \cite{dupuis2012,lu2019}. This holds in particular for the
Metropolis rates
$\gamma_{ab}(\mathbf x)=\Gamma_{ab}\min\{1,h_0(\tau_{ab}\mathbf x)/h_0(\mathbf x)\}$
\cite{dong2022,dupuis2012,earl2005,lu2019}, which are in detailed balance
with $\pi_0$ at every point, although detailed balance is not necessary
\cite{itoh2013,manousiouthakis1999,suwatodo2010}.

	\subsection{Assumptions}\label{ssec:fast-ass}
	
	To obtain a limit we assume three conditions, which we formulate through
	the generators of the species: the positions do not escape while their
	species change arbitrarily fast (F1), the generators and the rates are
	regular enough to freeze the positions while the label equilibrates (F2),
	and the averaged dynamics is well defined (F3).
	
	Since $E$ is locally compact and separable,
	there is a continuous function $w:E\to[0,\infty)$ with compact sublevel
	sets, and $\mathcal V(\mathbf y):=\sum_kw(y^k)$ is continuous, invariant
	under permutations, and has compact sublevel sets
	$B_K:=\{\mathcal V\le K\}$.
	
	\begin{assumption}[Fast-exchange hypotheses]\label{ass:fast}
		The species have generators $A_m$ on the class $\mathcal D$ above, and the
		following hold.
		\begin{enumerate}
			\item[\textup{(F1)}] \emph{Compact containment.} For every $T<\infty$
			and every compact $K_0\subset E^N$,
			\[
			\lim_{K\to\infty}\ \sup_{s>0}\ \sup_{(\mathbf x,\theta)\in K_0\times\mathfrak S_N}
			\PP_{(\mathbf x,\theta)}\Bigl(\sup_{t\le T}
			\mathcal V\bigl(\widehat{\mathbf X}^{(s)}_t\bigr)\ge K\Bigr)=0 .
			\]
			\item[\textup{(F2)}] \emph{Continuity.} For every $m$ and every
			$f\in\mathcal D$, $A_mf$ is bounded and continuous, and every
			$\gamma_{ab}$ is continuous.
			\item[\textup{(F3)}] \emph{Averaged dynamics.} For every
			$\mathbf y\in E^N$ the exchange chain frozen at $\mathbf y$ has a unique
			invariant probability $\nu_{\mathbf y}$ on $\mathfrak S_N$, and the
			martingale problem for $\bar\LL$ is well posed.
		\end{enumerate}
	\end{assumption}
	
		(F1) holds when the species share a Lyapunov function, that is, a
	continuous $v:E\to[1,\infty)$ with compact sublevel sets such that
	$e^{-ct}v(X_t)$ is a local supermartingale under every species for some
	$c\ge0$. Since the positions do not move at the rings, conditioning on
	the exchange clock as in the proof of
	Lemma~\ref{lem:mart-switching}\,\textup{(i)} makes
	$e^{-ct}\sum_kv(\widehat X^k_t)$ a non-negative supermartingale, and the
	maximal inequality gives (F1) with $w=v$. The choice of $w$ is
	irrelevant, because $\{\mathcal V<K\}$ has compact closure. The first half
	of (F3) holds when $G_\gamma$ is connected and $\gamma_{ab}>0$ on $E^N$
	for every edge $ab$, and it fails when $G_\gamma$ is not connected, since
	the label then stays in a coset of a proper subgroup of $\mathfrak S_N$
	(Remark~\ref{rem:components}). The second half is not inherited from the
	species, because the weights with which $\bar\LL$ averages their
	generators depend on the positions.
	
	All three conditions hold for the Langevin replicas
	$\dif X=-\nabla U(X)\,\dif t+\sqrt{2\mathsf T_m}\,\dif W$ of
	\cite{dupuis2012,lu2019}, with $\nabla U$ locally Lipschitz,
	$\langle x,-\nabla U(x)\rangle\le C(1+|x|^2)$ and the Metropolis rates of
	$h_m\propto e^{-U/\mathsf T_m}$. Indeed $1+|x|^2$ is a common Lyapunov
	function, the rates are continuous and positive, and by \eqref{eq:nu-db}
	$\bar\LL$ is a non-exploding diffusion with locally Lipschitz
	coefficients, whose martingale problem on
	$C^2_c$ is well posed \cite[Sections~5.2 and~5.4]{karatzasshreve1991}.
	It is then well posed on $\mathcal D_N$ as well. Indeed, every
	$f\in C^2_c(\RR^{dN})$ is the limit of functions $f_n\in\mathcal D_N$ with
	$f_n\to f$ and $\bar\LL f_n\to\bar\LL f$ boundedly and pointwise, obtained
	by approximating $f$ and its derivatives up to order two uniformly on a box
	by polynomials and multiplying by a product of cut-offs. Hence every
	solution of the martingale problem on $\mathcal D_N$ solves it on $C^2_c$.
	The three conditions also hold on a countable $E$ for species with bounded jump rates,
	as for spin systems in finite volume and lattice walkers, when $G_\gamma$
	is connected and $\gamma_{ab}>0$ on $E^N$ for every edge $ab$. Then (F1)
	is proved in Appendix~\ref{app:switching}, (F2) holds because the rates
	are bounded and every function is continuous for the discrete topology,
	and $\bar\LL$ generates a jump process with bounded
	rates, whose martingale problem is well posed
	\cite[Chapter~4]{ethierkurtz1986}.
	
	\subsection{Limits of the two descriptions}\label{ssec:fast}
	
	Throughout the rest of this section Assumption~\ref{ass:fast} is in force,
	$N$ is fixed, $s\to\infty$, and $T<\infty$ is a horizon, so that all times
	are in $[0,T]$. The switching process starts from a random point
	$(\widehat{\mathbf X}_0,\Lambda_0)$ with a law $\rho^{(s)}_0$ on
	$E^N\times\mathfrak S_N$ whose first marginal converges weakly, as
	$s\to\infty$, to a probability $\bar\rho_0$ on $E^N$, the default being
	$\rho^{(s)}_0=\delta_{(\mathbf x_0,\mathrm{id})}$. By
	Theorem~\ref{thm:N-symmetrisation}\,\textup{(i)} we realise the swapping
	process in replica coordinates as
	\begin{equation}\label{eq:replica-rep}
		\mathbf{X}^{(s)}_t=\Lambda_t\,\widehat{\mathbf{X}}^{(s)}_t .
	\end{equation}
	With $d$ a complete metric compatible with the topology of $E$, which is
	Polish, we write
	$d_N(\mathbf y,\mathbf y'):=\max_kd(y^k,y'^k)$, a complete metric on $E^N$
	that is invariant under permutations.
	
	We say that laws indexed by $s$ on a Polish space are \emph{tight as
		$s\to\infty$} if for every $\varepsilon>0$ there are $s_0$ and a compact set
	whose complement has mass at most $\varepsilon$ under each of them with
	$s\ge s_0$. By \textup{(F1)} and the convergence of the first marginals of
	$\rho^{(s)}_0$ (Appendix~\ref{app:switching}),
	\begin{equation}\label{eq:cc}
		\lim_{K\to\infty}\ \limsup_{s\to\infty}\
		\PP\Bigl(\sup_{t\le T}\mathcal V\bigl(\widehat{\mathbf X}^{(s)}_t\bigr)\ge K\Bigr)=0 ,
	\end{equation}
	that is, for every $\varepsilon>0$ there are $K$ and $s_0$ for which the
	probability is at most $\varepsilon$ for every $s\ge s_0$. Along a sequence
	$s_n\to\infty$ the supremum over $n$ may replace the upper limit in
	\eqref{eq:cc}, because the range of a single c\`adl\`ag path on $[0,T]$ is
	relatively compact. The results below are stated for $s\to\infty$, and they
	hold along every sequence $s_n\to\infty$ with initial laws whose first
	marginals converge, with the same proofs.
	
	We write $(\FF_t)_{t\ge0}$ for the filtration generated by the exchange clock and by
	the switching process up to time $t$, so that the exchange marks after time $t$ are
	independent of $\FF_t$.
	
	Since the exchanges never move its paths, the switching description
	converges at the level of processes.
	
	\begin{theorem}\label{thm:switching-conv}
		Let Assumption~\ref{ass:fast} hold, let the initial laws
		$\rho^{(s)}_0$ have first marginals converging weakly to $\bar\rho_0$, and
		let $\bar{\mathbf Y}$ be the solution of the martingale problem for
		$\bar\LL$ with initial law $\bar\rho_0$.
		\begin{enumerate}
			\item[\textup{(i)}] $\widehat{\mathbf{X}}^{(s)}\Rightarrow\bar{\mathbf{Y}}$
			as $s\to\infty$, in $D([0,\infty),E^N)$ and in $D([0,T],E^N)$ for every
			$T<\infty$, and in
			$C([0,T],E^N)$ when the species have continuous paths. Moreover
			$\bar{\mathbf Y}$ has no fixed times of discontinuity,
			$\PP(\bar{\mathbf Y}_t\neq\bar{\mathbf Y}_{t-})=0$ for every $t>0$.
			\item[\textup{(ii)}] $\bar{\mathbf Y}$ is a Markov process on $E^N$, and
			its law is invariant under renumbering of the replicas. If $\bar{\mathbf Y}$
			starts at $\mathbf x$, then $\sigma\bar{\mathbf Y}$ has the law of
			$\bar{\mathbf Y}$ started at $\sigma\mathbf x$, for every
			$\sigma\in\mathfrak S_N$, that is, the transition semigroup
			$\bar P_tf(\mathbf x):=\bar\EE_{\mathbf x}f(\bar{\mathbf Y}_t)$ satisfies
			$\bar P_t(f\circ\sigma)=(\bar P_tf)\circ\sigma$.
		\end{enumerate}
	\end{theorem}
	
	Part~(i) is an application of the averaging theorem of Kurtz
	\cite{kurtz1992} (Lemma~\ref{lem:kurtz}), with the label as the fast
	variable, whose occupation measure is identified with $\nu_{\mathbf y}$,
	and part~(ii) follows from the well-posedness of the martingale problem
	for $\bar\LL$. Both proofs are in Appendix~\ref{app:switching}. The
	boundedness of $A_mf$ in \textup{(F2)} can be removed by localising at
	the exit times of compact sets, at the cost of a longer proof that we do
	not include. Averaging with a fast variable whose rates depend on the 
	slow one is known as fully coupled averaging, see \cite{faggionato2010,pakdaman2012}, \cite[Chapter~12]{yinzhu2010}, and \cite{yinzhang2013} for two-time-scale switching. For two Langevin replicas with Metropolis rates
	this is the convergence of the positions in \cite[Theorem~2.3]{dupuis2012}.
	Here the slow variables are Markov families given only through
	\textup{(F1)}--\textup{(F3)}, and the fast variable lives on the group
	generated by an arbitrary exchange graph.

	By part~(ii), the switching positions become Markov in the limit, although
	at finite rate they are in general not Markov without their label
	(Proposition~\ref{prop:filter}). The invariance of the switching process
	under renumbering of the replicas
	(Proposition~\ref{prop:markov-unique}\,\textup{(i)}) survives in the limit
	as the equivariance of $\bar{\mathbf Y}$.
	
	When the species can be told apart from short pieces of path, as in the
	benchmark (Corollary~\ref{cor:bench-non-markov}), the label can be read from
	the past of the positions at every $s$. The label, however, moves on the
	time scale $1/s$. At any fixed time lag it has therefore forgotten its past
	and is close to $\nu$ at the current positions, the equilibrium of the
	frozen exchange chain (Lemma~\ref{lem:local-eq}). Knowing the current label
	then says almost nothing about the positions a fixed time later, and this is
	why the limit is Markov.
	
	\begin{lemma}\label{lem:local-eq}
		Under the assumptions of Theorem~\ref{thm:switching-conv}, let $0\le t'<t$
		and, for each $s$, let $\mathcal U=\mathcal U^{(s)}$ be an
		$\FF_{t'}$-measurable random variable with $|\mathcal U|\le1$ and
		$Z=Z^{(s)}$ an $\FF_{t'}$-measurable random element of a Polish space
		$\mathcal Z$, the laws of the $Z^{(s)}$ being tight as $s\to\infty$. For every
		$g:\mathfrak S_N\to\RR$ and every $f\in C_b(\mathcal Z\times E^N)$,
		\begin{equation}\label{eq:local-eq}
			\EE\bigl[\mathcal U\,g(\Lambda_t)\,f\bigl(Z,\widehat{\mathbf{X}}^{(s)}_t\bigr)\bigr]
			-\EE\bigl[\mathcal U\,\nu_{\widehat{\mathbf{X}}^{(s)}_t}(g)\,
			f\bigl(Z,\widehat{\mathbf{X}}^{(s)}_t\bigr)\bigr]
			\ \xrightarrow[s\to\infty]{}\ 0 .
		\end{equation}
	\end{lemma}
	
	\begin{proof}
		See Appendix~\ref{app:switching}.
	\end{proof}
	
	In other words, inside an expectation that also involves the past up to
	time $t'$ and the current positions, the label at time $t$ may be replaced
	by an independent draw from $\nu$ at the current positions.
	
	For bounded measurable $f$ on $E^N$ and a probability $\rho$ on $E^N$ we
	set
	\begin{equation}\label{eq:W-def}
		\mathcal Wf(\mathbf y):=\sum_{\theta\in\mathfrak S_N}\nu_{\mathbf y}(\theta)\,
		f(\theta\mathbf y),
		\qquad
		(\mathcal W\rho)(f):=\rho(\mathcal Wf).
	\end{equation}
	The \emph{weighted symmetrisation} $\mathcal W\rho$ is the law of
	$\Xi\mathbf Y$ when $\mathbf Y\sim\rho$ and, given $\mathbf Y$, the
	permutation $\Xi$ has law $\nu_{\mathbf Y}$. By Lemma~\ref{lem:frozen},
	$\mathcal Wf$ is bounded and continuous when $f$ is, and by
	\eqref{eq:nu-shift} it is permutation invariant. Since $\mathcal Wg=g$ and
	$\mathcal W(fg)=(\mathcal Wf)\,g$ for permutation-invariant $g$, the
	operator $\mathcal W$ is a Markov projection onto the permutation-invariant
	functions, $\mathcal W^2=\mathcal W$, and $\mathcal W\delta_{\mathbf y}$
	depends on $\mathbf y$ only through its orbit.
	
	The swapping process is in general not tight
	(Corollary~\ref{cor:nontight}), but its laws at finitely many times
	converge. By \eqref{eq:replica-rep} these laws involve the label at those
	times, and by Lemma~\ref{lem:local-eq} the label at each of them is, in the
	limit, a fresh draw from $\nu$ at the current positions, independent of
	its values at the earlier times. In the limit the swapping process is
	therefore the averaged dynamics relabelled independently at every time.
	Its finite-dimensional laws still come from a Markov transition function,
	which however does not tend to the identity at small times.
	
	\begin{theorem}\label{thm:fdd}
		Under the assumptions of Theorem~\ref{thm:switching-conv}, let
		$0<t_1<\dots<t_J$.
		\begin{enumerate}
			\item[\textup{(i)}] As $s\to\infty$, in distribution,
			\[
			\bigl(\widehat{\mathbf X}^{(s)}_{t_i},\Lambda_{t_i}\bigr)_{i\le J}
			\Rightarrow\bigl(\bar{\mathbf Y}_{t_i},\Xi_i\bigr)_{i\le J},
			\]
			where, conditionally on $\bar{\mathbf Y}$, the permutations
			$\Xi_1,\dots,\Xi_J\in\mathfrak S_N$ are independent and $\Xi_i$ has law
			$\nu_{\bar{\mathbf Y}_{t_i}}$. In particular
			$(\mathbf X^{(s)}_{t_1},\dots,\mathbf X^{(s)}_{t_J})\Rightarrow
			(\Xi_1\bar{\mathbf Y}_{t_1},\dots,\Xi_J\bar{\mathbf Y}_{t_J})$.
			\item[\textup{(ii)}] The operators $\bar P^{\swap}_r:=\bar P_r\mathcal W$,
			$r>0$, with $\bar P_r$ as in Theorem~\ref{thm:switching-conv}, form a
			Markov transition function. The limits in \textup{(i)} are the
			finite-dimensional laws at positive times of a Markov process with
			initial law $\bar\rho_0$ and this transition function.
		\end{enumerate}
	\end{theorem}
	
	\begin{proof}
		(i) We write $\widehat{\mathbf X}$ for $\widehat{\mathbf X}^{(s)}$. Since
		$\mathfrak S_N$ is finite, the first convergence amounts to
		\begin{equation}\label{eq:fdd-target}
			\EE\Bigl[\prod_{i\le J}g_i(\Lambda_{t_i})\,
			\Psi\bigl(\widehat{\mathbf X}_{t_1},\dots,\widehat{\mathbf X}_{t_J}\bigr)\Bigr]
			\to\EE\Bigl[\prod_{i\le J}\nu_{\bar{\mathbf Y}_{t_i}}(g_i)\,
			\Psi\bigl(\bar{\mathbf Y}_{t_1},\dots,\bar{\mathbf Y}_{t_J}\bigr)\Bigr]
		\end{equation}
		for all $g_i:\mathfrak S_N\to\RR$ and $\Psi\in C_b((E^N)^J)$, and the second
		follows by the continuous mapping theorem.
		
		\emph{Continuity in the initial point.} For $J\ge2$,
		$\Psi\in C_b\bigl((E^N)^{J-1}\times E^N\bigr)$ and $r>0$ put
		\[
		\mathfrak u^{(s)}(z,\mathbf x,\theta):=
		\EE_{(\mathbf x,\theta)}\Psi\bigl(z,\widehat{\mathbf X}^{(s)}_r\bigr),
		\qquad
		\bar{\mathfrak u}(z,\mathbf x):=\bar\EE_{\mathbf x}\Psi\bigl(z,\bar{\mathbf Y}_r\bigr),
		\]
		with $\EE_{(\mathbf x,\theta)}$ as after Theorem~\ref{thm:exist} and
		$\bar\EE_{\mathbf x}$ as in Theorem~\ref{thm:switching-conv}\,\textup{(ii)}.
		We claim that $\bar{\mathfrak u}$ is continuous and that
		$\mathfrak u^{(s)}\to\bar{\mathfrak u}$ uniformly on compact subsets of
		$(E^N)^{J-1}\times E^N$ and in $\theta\in\mathfrak S_N$.
		Theorem~\ref{thm:switching-conv}\,\textup{(i)} uses the initial laws only
		through the convergence of their first marginals, and it holds along
		sequences. Hence if $(z_n,\mathbf x_n)\to(z,\mathbf x)$,
		$\theta_n\in\mathfrak S_N$ and $s_n\to\infty$, the switching process with
		rates $s_n\gamma$ started at $(\mathbf x_n,\theta_n)$ converges to
		$\bar{\mathbf Y}$ started at $\mathbf x$, and since $r$ is almost surely a
		continuity point of the limit, the Skorokhod representation gives
		$\mathfrak u^{(s_n)}(z_n,\mathbf x_n,\theta_n)\to\bar{\mathfrak u}(z,\mathbf x)$.
		Taking $(z_n,\mathbf x_n)$ constant gives pointwise convergence. Given
		$(z_n,\mathbf x_n)\to(z,\mathbf x)$, choose $s_n\to\infty$ with
		$|\mathfrak u^{(s_n)}(z_n,\mathbf x_n,\mathrm{id})-\bar{\mathfrak u}(z_n,\mathbf x_n)|\le1/n$.
		Then $\bar{\mathfrak u}(z_n,\mathbf x_n)\to\bar{\mathfrak u}(z,\mathbf x)$, so
		$\bar{\mathfrak u}$ is continuous. If the convergence were not uniform on
		some compact set, there would be $\varepsilon>0$ and a sequence
		$(z_n,\mathbf x_n,\theta_n,s_n)$ with $(z_n,\mathbf x_n)$ converging,
		$s_n\to\infty$ and
		$|\mathfrak u^{(s_n)}(z_n,\mathbf x_n,\theta_n)-\bar{\mathfrak u}(z_n,\mathbf x_n)|\ge\varepsilon$,
		which contradicts the two facts just proved.
		
		\emph{Induction.} We argue by induction on $J$. For $J=1$,
		Lemma~\ref{lem:local-eq} with $\mathcal U=1$ and a constant $Z$ replaces
		$g_1(\Lambda_{t_1})$ by $\nu_{\widehat{\mathbf X}_{t_1}}(g_1)$, and
		Theorem~\ref{thm:switching-conv}\,\textup{(i)} applies to the bounded
		continuous function $\mathbf y\mapsto\nu_{\mathbf y}(g_1)\Psi(\mathbf y)$ at
		the fixed time $t_1$. Let $J\ge2$ and assume \eqref{eq:fdd-target} for
		$J-1$ times. Put
		$Z:=(\widehat{\mathbf X}_{t_1},\dots,\widehat{\mathbf X}_{t_{J-1}})$ and
		$\mathcal U:=\prod_{i<J}g_i(\Lambda_{t_i})$, with $|g_i|\le1$ without loss of
		generality, both $\FF_{t_{J-1}}$-measurable, the laws of $Z$ being tight
		as $s\to\infty$ by \eqref{eq:cc}. Lemma~\ref{lem:local-eq} at $t=t_J$ with
		$t'=t_{J-1}$ replaces $g_J(\Lambda_{t_J})$ by
		$\nu_{\widehat{\mathbf X}_{t_J}}(g_J)$, so it suffices to treat $g_J\equiv1$
		with $\Psi$ replaced by
		$\tilde\Psi(z,\mathbf y):=\nu_{\mathbf y}(g_J)\Psi(z,\mathbf y)$, which is
		bounded and continuous. By the Markov property of the switching process at
		$t_{J-1}$ (Theorem~\ref{thm:exist}),
		$\EE[\tilde\Psi(Z,\widehat{\mathbf X}_{t_J})\,|\,\FF_{t_{J-1}}]
		=\mathfrak u^{(s)}(Z,\widehat{\mathbf X}_{t_{J-1}},\Lambda_{t_{J-1}})$, with
		$\mathfrak u^{(s)}$ as above for $\tilde\Psi$ in place of $\Psi$ and
		$r=t_J-t_{J-1}$. By the claim and the tightness of $Z$, whose last
		component is $\widehat{\mathbf X}_{t_{J-1}}$, this differs in $L^1$ by
		$o(1)$ from $\Psi'(Z):=\bar{\mathfrak u}(Z,\widehat{\mathbf X}_{t_{J-1}})$, a
		bounded continuous function of $Z$. The induction hypothesis applied to
		$\Psi'$ gives convergence to
		$\EE\bigl[\mathfrak w(\bar{\mathbf Y}_{t_1},\dots,\bar{\mathbf Y}_{t_{J-1}})\,
		\Psi'(\bar{\mathbf Y}_{t_1},\dots,\bar{\mathbf Y}_{t_{J-1}})\bigr]$, where
		$\mathfrak w(\mathbf y_1,\dots,\mathbf y_{J-1}):=\prod_{i<J}\nu_{\mathbf y_i}(g_i)$
		is bounded and continuous by Lemma~\ref{lem:frozen}. Since
		$\bar{\mathbf Y}$ is Markov by
		Theorem~\ref{thm:switching-conv}\,\textup{(ii)},
		$\Psi'(\bar{\mathbf Y}_{t_1},\dots,\bar{\mathbf Y}_{t_{J-1}})
		=\EE[\tilde\Psi(\bar{\mathbf Y}_{t_1},\dots,\bar{\mathbf Y}_{t_J})\mid
		\bar{\mathbf Y}_u,u\le t_{J-1}]$ by the definition of $\bar{\mathfrak u}$.
		The limit therefore equals
		$\EE\bigl[\mathfrak w(\bar{\mathbf Y}_{t_1},\dots,\bar{\mathbf Y}_{t_{J-1}})\,
		\tilde\Psi(\bar{\mathbf Y}_{t_1},\dots,\bar{\mathbf Y}_{t_J})\bigr]$,
		which is the right-hand side of \eqref{eq:fdd-target} for $J$ times.
		
		(ii) By the properties of $\mathcal W$ stated after \eqref{eq:W-def}, the
		function $\mathcal Wf$ is permutation invariant for every bounded measurable
		$f$, and by Theorem~\ref{thm:switching-conv}\,\textup{(ii)} so is
		$\bar P_r\mathcal Wf$, since
		$(\bar P_r\mathcal Wf)\circ\sigma=\bar P_r(\mathcal Wf\circ\sigma)=\bar P_r\mathcal Wf$.
		As $\mathcal W$ fixes permutation-invariant functions,
		$\bar P_r\mathcal W\bar P_u\mathcal W=\bar P_r\bar P_u\mathcal W=\bar P_{r+u}\mathcal W$.
		By~(i), given $\bar{\mathbf Y}$ the variables $\Xi_i\bar{\mathbf Y}_{t_i}$
		are independent with laws $\mathcal W\delta_{\bar{\mathbf Y}_{t_i}}$, so for
		$f_1,\dots,f_J\in C_b(E^N)$ we have
		$\EE\prod_if_i(\Xi_i\bar{\mathbf Y}_{t_i})=\EE\prod_i\mathcal Wf_i(\bar{\mathbf Y}_{t_i})$.
		Conditioning successively on $\bar{\mathbf Y}$ up to the times
		$t_{J-1},\dots,t_1$ with the Markov property, and using at each step that
		$(\mathcal Wf_i)\,h=\mathcal W(f_ih)$ for the permutation-invariant function
		$h=\bar P^{\swap}_{t_{i+1}-t_i}\bigl(f_{i+1}\cdots\bigr)$, gives
		\begin{equation}\label{eq:fdd-markov}
			\EE\prod_{i\le J}f_i\bigl(\Xi_i\bar{\mathbf Y}_{t_i}\bigr)
			=\bar\rho_0\Bigl(\bar P^{\swap}_{t_1}\bigl(f_1\,\bar P^{\swap}_{t_2-t_1}
			\bigl(f_2\cdots\bar P^{\swap}_{t_J-t_{J-1}}f_J\bigr)\bigr)\Bigr),
		\end{equation}
		which is the Markov property of the limits with transition function
		$(\bar P^{\swap}_r)_{r>0}$.
	\end{proof}
	
	For $J=1$, part~\textup{(i)} gives
	$\EE f(\widehat{\mathbf X}^{(s)}_t)\to\bar\rho_t(f)$ and
	$\EE f(\mathbf X^{(s)}_t)\to(\mathcal W\bar\rho_t)(f)$ for $t>0$ and
	$f\in C_b(E^N)$, where $\bar\rho_t:=\mathrm{Law}(\bar{\mathbf Y}_t)$. At a
	fixed time the two descriptions therefore differ in the limit by one
	weighted symmetrisation, $\mathcal W\bar\rho_t-\bar\rho_t$, and the initial
	label plays no role at positive times.
	
	\begin{corollary}\label{cor:non-normal}
		Under the assumptions of Theorem~\ref{thm:switching-conv}, and with
		$\bar P^{\swap}_r$ as in Theorem~\ref{thm:fdd}, the function
		$\bar P^{\swap}_rf$ is permutation invariant for every bounded measurable
		$f$, and $\bar P^{\swap}_rf\to\mathcal Wf$ pointwise as $r\downarrow0$ for
		every $f\in C_b(E^N)$. In particular, the transition function
		$(\bar P^{\swap}_r)_{r>0}$ does not tend to the identity as $r\downarrow0$
		when $E$ has at least two points and $\gamma_{ab}>0$ on $E^N$ for every
		edge $ab$ of $G_\gamma$.
	\end{corollary}
	
	\begin{proof}
		The permutation invariance was shown in the proof of
		Theorem~\ref{thm:fdd}\,\textup{(ii)}. For $f\in C_b(E^N)$ the function
		$\mathcal Wf$ is bounded and continuous, and $\bar{\mathbf Y}$ is
		right-continuous at $0$, so $\bar P_r\mathcal Wf\to\mathcal Wf$ pointwise
		by bounded convergence. For the last statement, the graph $G_\gamma$ is
		connected by \textup{(F3)}, as in the proof of
		Corollary~\ref{cor:nontight}, so $\nu_{\mathbf y}$ charges every
		permutation. Take $\mathbf y\in E^N$ with $y^1\ne y^2$ and
		$f\in C_b(E^N)$ with $0\le f\le1$, $f(\mathbf y)=1$ and $f(\theta\mathbf y)=0$
		for the finitely many $\theta$ with $\theta\mathbf y\ne\mathbf y$. Then
		$\mathcal Wf(\mathbf y)=\nu_{\mathbf y}(\{\theta:\theta\mathbf y=\mathbf y\})<1
		=f(\mathbf y)$, because $\tau_{12}\mathbf y\ne\mathbf y$.
	\end{proof}
	
	The finite-dimensional limits of the swapping process thus come from a
Markov transition function, as at every finite rate, but one that tends
to the projection $\mathcal W$ rather than to the identity as
$r\downarrow0$, so it is not that of a Markov family on
$D([0,\infty),E^N)$. Such non-standard transition functions are classical
on countable state spaces \cite{chung1967,doob1942}.
	
		This lack of normality is what rules out a c\`adl\`ag limit. Suppose that a
	process $\mathbf Z$ with these finite-dimensional laws were right-continuous
	at some $t>0$, and let $g,h\in C_b(E^N)$. Then
	\[
	\EE\bigl[h(\mathbf Z_t)g(\mathbf Z_t)\bigr]
	=\lim_{r\downarrow0}\EE\bigl[h(\mathbf Z_t)g(\mathbf Z_{t+r})\bigr]
	=\lim_{r\downarrow0}\EE\bigl[h(\mathbf Z_t)\,\bar P^{\swap}_rg(\mathbf Z_t)\bigr]
	=\EE\bigl[h(\mathbf Z_t)\,\mathcal Wg(\mathbf Z_t)\bigr].
	\]
	The first equality follows from the right continuity at $t$ and bounded
	convergence, the second from the Markov property \eqref{eq:fdd-markov} with
	$J=2$, and the third from the pointwise convergence
	$\bar P^{\swap}_rg\to\mathcal Wg$ of Corollary~\ref{cor:non-normal} and
	bounded convergence. Since $h$ is arbitrary, $\mathcal Wg(\mathbf Z_t)=g(\mathbf Z_t)$
	almost surely. As $\mathcal Wg$ is constant on orbits, this fails as soon as
	the relabelling at time $t$ is genuinely random, that is, as soon as
	$\mathcal W\delta_{\bar{\mathbf Y}_t}$ fails to be a point mass with
	positive probability. This is the mechanism behind
	Corollary~\ref{cor:nontight}\,\textup{(i)}.

	\begin{corollary}\label{cor:nontight}
		Under the assumptions of Theorem~\ref{thm:switching-conv}, let $T>0$.
		\begin{enumerate}
			\item[\textup{(i)}] Let $\gamma_{ab}(\mathbf x)>0$ for every
			$\mathbf x\in E^N$ and every edge $ab$ of $G_\gamma$, and let
			$\PP(\bar Y^k_t\ne\bar Y^l_t)>0$ for some $t\in[0,T)$ and $k\ne l$. Then
			no sequence $\mathbf X^{(s_n)}$ with $s_n\to\infty$ converges in
			distribution in $D([0,T],E^N)$. In particular the laws of the
			$\mathbf X^{(s)}$ are not tight as $s\to\infty$.
			\item[\textup{(ii)}] The empirical measure
			$\eta^{(s)}_T:=T^{-1}\int_0^T\delta_{\mathbf X^{(s)}_t}\dif t$ of the
			swapping process converges in distribution, as a random element of
			$\mathcal P(E^N)$, to
			\begin{equation}\label{eq:weighted-emp}
				\bar\eta_T:=\frac1T\int_0^T\sum_{\theta\in\mathfrak S_N}
				\nu_{\bar{\mathbf Y}_t}(\theta)\,\delta_{\theta\bar{\mathbf Y}_t}\,\dif t .
			\end{equation}
		\end{enumerate}
	\end{corollary}
	
	\begin{proof}
		See Appendix~\ref{app:switching}.
	\end{proof}
	
	The non-tightness was observed for two Langevin replicas in
	\cite[Section~2]{dupuis2012}, which is why the infinite-swapping limit is
	constructed there from the switching description, and $\bar\eta_T$ is the
	weighted empirical measure of infinite swapping
	\cite[Theorem~2.3]{dupuis2012}. For one species $m$, typically the coldest
	one, Theorem~\ref{thm:fdd}\,\textup{(i)} applied to functions of the $m$-th
	coordinate gives, for $f_1,\dots,f_J\in C_b(E)$,
	\begin{equation}\label{eq:one-replica}
		\EE\prod_{i\le J}f_i\bigl(X^{(s),m}_{t_i}\bigr)\longrightarrow
		\EE\prod_{i\le J}\Bigl(\sum_{l=1}^Np_{lm}\bigl(\bar{\mathbf Y}_{t_i}\bigr)\,
		f_i\bigl(\bar Y^l_{t_i}\bigr)\Bigr),
	\end{equation}
	with $p_{lm}(\mathbf y):=\nu_{\mathbf y}(\{\theta:\theta(l)=m\})$ the
	probability that replica $l$ carries species $m$, and accordingly the
	$m$-th marginal of $\bar\eta_T$,
	$T^{-1}\int_0^T\sum_lp_{lm}(\bar{\mathbf Y}_t)\,\delta_{\bar Y^l_t}\,\dif t$,
	is the weighted estimator of species $m$.
	
	\begin{corollary}\label{cor:targets}
		Under the assumptions of Theorem~\ref{thm:switching-conv}, let the species
		have invariant probabilities $\pi_m$ and let the exchanges preserve
		$\pi_0=\bigotimes_m\pi_m$ (Section~\ref{ssec:fast-hyp}). Let $\tilde\pi_0$ be the lift \eqref{eq:lift} of
		$\pi_0$ and $\bar\pi:=(N!)^{-1}\sum_\theta\theta_*\pi_0$ its first marginal.
		\begin{enumerate}
			\item[\textup{(i)}] For every $\theta\in\mathfrak S_N$,
			$\tilde\pi_0(\dif\mathbf y\times\{\theta\})=\nu_{\mathbf y}(\theta)\,\bar\pi(\dif\mathbf y)$:
			under $\tilde\pi_0$ the label given the positions has law $\nu_{\mathbf y}$.
			\item[\textup{(ii)}] $\bar\pi$ is invariant for $\bar{\mathbf Y}$ and
			$\mathcal W\bar\pi=\pi_0$. Hence, if $\bar{\mathbf Y}$ starts from
			$\bar\pi$, the finite-dimensional limits of Theorem~\ref{thm:fdd} are
			those of a stationary process with one-time law $\pi_0$, species $m$ has
			one-time law $\pi_m$ in the limit, and $\EE\,\bar\eta_T=\pi_0$ for every
			$T>0$.
		\end{enumerate}
	\end{corollary}
	
	\begin{proof}
		(i) By \eqref{eq:lift},
		$\int G\,\dif\tilde\pi_0=(N!)^{-1}\sum_\theta\int G(\theta^{-1}\mathbf x,\theta)\,\pi_0(\dif\mathbf x)$
		for bounded measurable $G$ on $E^N\times\mathfrak S_N$. Take
		$G(\mathbf y,\theta):=\varphi(\mathbf y)(\mathcal Q_{\mathbf y}g)(\theta)$, with
		$\varphi$ bounded measurable on $E^N$ and $g$ on $\mathfrak S_N$, and
		$f(\mathbf x):=\sum_\theta\varphi(\theta^{-1}\mathbf x)g(\theta)$. The
		substitution $\theta'=\tau_{ab}\theta$ gives
		$f(\tau_{ab}\mathbf x)=\sum_\theta\varphi(\theta^{-1}\mathbf x)g(\tau_{ab}\theta)$,
		so \eqref{eq:frozen-gen} and \eqref{eq:S-op} give
		$\int G\,\dif\tilde\pi_0=(N!)^{-1}\pi_0(\mathcal Sf)=0$. Each measure 
		$\tilde\pi_0(\cdot\times\{\theta\})$ is dominated by $\bar\pi$, so
		$\tilde\pi_0(\dif\mathbf y\times\{\theta\})=\kappa_{\mathbf y}(\theta)\,\bar\pi(\dif\mathbf y)$
		for a probability kernel $\kappa$ from $E^N$ to $\mathfrak S_N$, and the
		identity reads
		$\int\varphi(\mathbf y)\sum_\theta\kappa_{\mathbf y}(\theta)(\mathcal Q_{\mathbf y}g)(\theta)\,\bar\pi(\dif\mathbf y)=0$
		for every $\varphi$. As $\mathfrak S_N$ is finite, for $\bar\pi$-almost
		every $\mathbf y$ the probability $\kappa_{\mathbf y}$ is invariant for
		$\mathcal Q_{\mathbf y}$, hence equal to $\nu_{\mathbf y}$ by \textup{(F3)}.
		
		(ii) The swapping process preserves $\pi_0$ at every $s$
		(Section~\ref{ssec:fast-hyp}), so by
		Proposition~\ref{prop:invariant}\,\textup{(ii)} the switching
		process started from $\tilde\pi_0$ has law $\tilde\pi_0$ at every time and
		every $s$, and $\widehat{\mathbf X}^{(s)}_t\sim\bar\pi$.
		Theorem~\ref{thm:switching-conv}\,\textup{(i)} with $\rho^{(s)}_0=\tilde\pi_0$
		gives $\bar{\mathbf Y}_t\sim\bar\pi$ for every $t$. By (i),
		$(\mathcal W\bar\pi)(f)=\sum_\theta\int f(\theta\mathbf y)\,\tilde\pi_0(\dif\mathbf y\times\{\theta\})
		=(\Phi_*\tilde\pi_0)(f)=\pi_0(f)$. Since $\bar\pi\bar P_r=\bar\pi$, the
		right-hand side of \eqref{eq:fdd-markov} with $\bar\rho_0=\bar\pi$ depends
		on the times only through their differences, and its one-time marginal is
		$\mathcal W\bar\pi=\pi_0$, whose $m$-th marginal is $\pi_m$. Finally
		$\EE\,\bar\eta_T=T^{-1}\int_0^T\mathcal W\bar\pi\,\dif t=\pi_0$ by
		\eqref{eq:weighted-emp} and Fubini.
	\end{proof}
	
	For Metropolis rates this is the consistency of the weighted empirical
	measure of infinite swapping \cite{dupuis2012,lu2019,plattner2011,yu2016},
	and for a graph with several components (Remark~\ref{rem:components}) the
	limit is the partial infinite swapping of \cite[Section~6.1]{dupuis2012}.
	When the preservation condition holds at every point, as in
	Section~\ref{ssec:fast-hyp} for continuous $h_0$ and $\gamma_{ab}$,
	$\nu_{\mathbf y}$ is given by \eqref{eq:nu-db} for every such rule, so
	$\bar\LL$, $\mathcal W$ and all the limits of this section are those of the
	Metropolis rates: exchange rules that preserve $\pi_0$ without detailed
	balance, non-reversible ones included, can differ from Metropolis only at
	finite rate. When $\pi_0$ has a positive density, (i) recovers
	\eqref{eq:nu-db} for $\bar\pi$-almost every $\mathbf y$ in any case.
	
\begin{remark}[Several connected components]\label{rem:components}
	If $G_\gamma$ is not connected, the label started at $\theta_0$ stays in
	the coset $\mathsf H\theta_0$ of the subgroup $\mathsf H\subset\mathfrak S_N$
	generated by the transpositions of the edges, which is the product of the
	symmetric groups of the components. Renumbering the replicas by $\theta_0$
	(Proposition~\ref{prop:markov-unique}\,\textup{(i)}) brings the label into
	$\mathsf H$. When the initial label is concentrated on one coset, the
	results of this section hold with the same proofs, with $\mathfrak S_N$
	replaced by $\mathsf H$. Concretely, the first half of \textup{(F3)} is
	required of the frozen chain on $\mathsf H$, the permutations in
	\eqref{eq:nu-shift}, in the map $\mathcal W$ and in the orbits range over
	$\mathsf H$, and in Corollary~\ref{cor:nontight}\,\textup{(i)} the replicas
	$k$ and $l$ must carry species of the same component. The limit depends on
	the coset, and when the initial label charges several cosets it is the
	corresponding mixture.
\end{remark}
	
	\section{An exactly solvable benchmark}\label{sec:brownian}
We now treat two Brownian motions with constant drifts, exchanged at a
constant rate. For this model the transition densities and most of the
observables introduced above can be computed explicitly. Since Brownian
motion has no invariant probability, the model has no targets. It therefore
illustrates the parts of the theory that do not involve equilibrium. These
are the symmetrisation identity, the failure of the Markov property of the
switching positions and the fast-exchange limit with constant rates
\eqref{eq:average-L}, together with the non-tightness of the swapping
process and the structure of its finite-dimensional limits. The model
satisfies Assumption~\ref{ass:fast} with the Lyapunov function $v(x)=1+x^2$.
	
	\begin{definition}[Two-species Brownian benchmark]\label{def:bench}
		Let $E=\RR$, let the two species be
		$\dif X=\mu_a\dif t+\sigma_a\dif W$ with $\mu_a\in\RR$ and $\sigma_a>0$,
		$a=1,2$, and let the exchange rate be the constant $s>0$, that is,
		$\gamma_{12}\equiv1$ with rate $s\gamma_{12}$. We write $(X_t,Y_t)$ for
		the swapping pair and $(\widehat X_t,\widehat Y_t)$ for the positions of
		the switching process, with label $\Lambda_t$, both started from
		$(x_0,y_0)$ and the latter with $\Lambda_0=\mathrm{id}$, and set
		\[
		d_0:=x_0-y_0,\qquad
		\bar\mu:=\tfrac{\mu_1+\mu_2}{2},\qquad
		\delta:=\tfrac{\mu_1-\mu_2}{2},\qquad
		\sigma_{\rm eff}^2:=\tfrac{\sigma_1^2+\sigma_2^2}{2},\qquad
		\sigma_{\min}:=\sigma_1\wedge\sigma_2 .
		\]
	\end{definition}
	
	Here $\sigma_a$ is the volatility, so the quadratic variation per unit time
	is $\sigma_a^2$, and the label takes values in
	$\mathfrak S_2=\{\mathrm{id},\tau\}$. Since the rate is constant, the accepted exchanges do not depend on the
	positions. The label $\Lambda$ is then a two-state Markov chain with jump
	rate $s$, independent of the driving Brownian motions $W^1,W^2$, and the
	switching pair can be realised as
	\begin{equation}\label{eq:bench-switching}
		\widehat X_t=x_0+\int_0^t\mu_{\Lambda_u(1)}\dif u+\int_0^t\sigma_{\Lambda_u(1)}\dif W^1_u,
		\qquad
		\widehat Y_t=y_0+\int_0^t\mu_{\Lambda_u(2)}\dif u+\int_0^t\sigma_{\Lambda_u(2)}\dif W^2_u .
	\end{equation}
	Under the coupling of Theorem~\ref{thm:N-symmetrisation}, $(X_t,Y_t)$ equals
	$(\widehat X_t,\widehat Y_t)$ on $\{\Lambda_t=\mathrm{id}\}$ and
	$(\widehat Y_t,\widehat X_t)$ on $\{\Lambda_t=\tau\}$. We also set
	$R_u:=\mathbf 1_{\{\Lambda_u=\mathrm{id}\}}-\mathbf 1_{\{\Lambda_u=\tau\}}$.
	
	\subsection{Transition densities and the symmetrisation identity}\label{ssec:bench-density}
	
	The transition densities can be obtained from the forward equation with jumps,
	a system of coupled Fokker--Planck equations with exchange terms, by taking
	the Fourier transform in the space variables and the Laplace transform in
	time, as for the lattice version of the model in \cite{diazperez2025}. The
	swapping process is a jump diffusion and the switching process a hybrid
	switching diffusion, and we refer to \cite{kunita2019} and
	\cite{yinzhu2010}, respectively, for the conditions under which their
	transition densities solve these forward equations. This yields the integral
	representation of Proposition~\ref{prop:density} below, including the
	parametrisation by $z\in(0,1)$ around which it is organised. The switching
	decomposition of Section~\ref{sec:setup} gives a more probabilistic
	derivation, in which each term is the contribution of a value of the
	occupation time of the label, and this is the proof given in
	Appendix~\ref{app:density}. In the proposition,
	$p(x,y,t,\Lambda_t{=}\theta)$ denotes the density of the sector measure
	$Q^\theta_t$ of Theorem~\ref{thm:N-symmetrisation}.
	
	\begin{proposition}\label{prop:density}
		In the benchmark of Definition~\ref{def:bench}, for $t>0$ and
		$(x,y)\in\RR^2$, the densities $p(x,y,t,\Lambda_t{=}\theta)$ of the
		sector measures $Q^\theta_t$ of \eqref{eq:sector-measure} of the switching
		pair are
		\begin{align}
			\notag p(x,y,t,\Lambda_t{=}\mathrm{id}) &=\frac{e^{-st}}{2\pi\sigma_1\sigma_2 t}\,
			\exp\!\Bigl(-\tfrac{(x-x_0-\mu_1t)^2}{2\sigma_1^2 t}-\tfrac{(y-y_0-\mu_2t)^2}{2\sigma_2^2 t}\Bigr)\\
			&\quad+\frac{se^{-st}}{4\pi}\sum_{\chi\in\{\pm 1\}}\!\!\int_0^1\!\frac{(1+\chi\sqrt z)I_1(st\sqrt{1-z})\,\mathcal T(x,y,z,t,\chi)}{\sqrt{z(1-z)}}\dif z,
			\label{eq:density-sector1}\\
			p(x,y,t,\Lambda_t{=}\tau) &=\frac{se^{-st}}{4\pi}\sum_{\chi\in\{\pm 1\}}\!\!\int_0^1\!\frac{I_0(st\sqrt{1-z})\,\mathcal T(x,y,z,t,\chi)}{\sqrt z}\dif z,
			\label{eq:density-sector2}
		\end{align}
		so that its full density is
		$\widehat p=p(\cdot,\Lambda_t{=}\mathrm{id})+p(\cdot,\Lambda_t{=}\tau)$, while the
		swapping density is
		\begin{equation}\label{eq:density-explicit}
			p^{\swap}(x,y,t)=p(x,y,t,\Lambda_t{=}\mathrm{id})+p(y,x,t,\Lambda_t{=}\tau).
		\end{equation}
		Here, suppressing the dependence on $(x_0,y_0)$,
		\begin{equation}\label{eq:Theta-def}
			\mathcal T(x,y,z,t,\chi)=\frac{\exp\!\Bigl(-\tfrac{\bigl(x-x_0-t\tfrac{(\mu_1+\mu_2)+\chi\sqrt z(\mu_1-\mu_2)}{2}\bigr)^2}{t\left[(\sigma_1^2+\sigma_2^2)+\chi\sqrt z(\sigma_1^2-\sigma_2^2)\right]}-\tfrac{\bigl(y-y_0-t\tfrac{(\mu_1+\mu_2)-\chi\sqrt z(\mu_1-\mu_2)}{2}\bigr)^2}{t\left[(\sigma_1^2+\sigma_2^2)-\chi\sqrt z(\sigma_1^2-\sigma_2^2)\right]}\Bigr)}
			{\sqrt{(\sigma_1^2+\sigma_2^2)^2-z(\sigma_1^2-\sigma_2^2)^2}},
		\end{equation}
		and $I_0,I_1$ are modified Bessel functions of the first kind.
	\end{proposition}
	
	\begin{proof}
		See Appendix~\ref{app:density}.
	\end{proof}
	
	The representation is simpler than it looks. Each term is a Gaussian
	density averaged over the time $a\in[0,t]$ that replica $1$
	spends carrying species $1$, written $a=\tfrac t2(1+\chi\sqrt z)$: the first
	line of \eqref{eq:density-sector1} is the contribution of the event of no
	exchange, of probability $e^{-st}$, the Bessel factors are the law of that time on each
	sector, and $\mathcal T/(\pi t)$ is the Gaussian density with the drifts and
	variances accumulated over it. The swapping density differs only by the
	transposition of the arguments on the sector $\{\Lambda_t=\tau\}$. Both
	sector densities are positive everywhere.
	
	\begin{corollary}
		\label{cor:bench-non-markov}
		In the benchmark of Definition~\ref{def:bench} with
		$\sigma_1\ne\sigma_2$, the hypotheses of Corollary~\ref{cor:non-markov}
		hold at every $t>0$, so the switching positions are not a Markov process
		in their own filtration.
	\end{corollary}
	
	\begin{proof}
		See Appendix~\ref{app:density}.
	\end{proof}
	
	Two Gaussian densities recur below. The no-swap density
	\begin{equation}\label{eq:p0-defs}
		p_0(x,y,t):=\frac{1}{2\pi\sigma_1\sigma_2t}
		\exp\left(-\frac{(x-x_0-\mu_1t)^2}{2\sigma_1^2t}
		-\frac{(y-y_0-\mu_2t)^2}{2\sigma_2^2t}\right)
	\end{equation}
	is the density of the pair when no exchange occurs, each coordinate keeping
	its own coefficients. The fully averaged density
	\begin{equation}\label{eq:p0bar-def}
		\bar p_0(x,y,t):=\frac{1}{2\pi\sigma_{\rm eff}^2t}
		\exp\left(-\frac{(x-x_0-\bar\mu t)^2+(y-y_0-\bar\mu t)^2}{2\sigma_{\rm eff}^2t}\right)
	\end{equation}
	is the product of two Gaussians started at $x_0$ and $y_0$, both with drift
	$\bar\mu$ and variance $\sigma_{\rm eff}^2$ per unit time. These are the two
	extremes of the behaviour of the model.
	
	\begin{proposition}\label{prop:regimes}
		In the benchmark of Definition~\ref{def:bench}, let $\widehat p$ and
		$p^{\swap}$ be the densities of the switching and swapping pairs
		(Proposition~\ref{prop:density}) and $p_0,\bar p_0$ as in
		\eqref{eq:p0-defs} and \eqref{eq:p0bar-def}. With all convergences in
		$L^1(\RR^2)$, the following hold.
		\begin{enumerate}
			\item[\rm(i)] \emph{(Slow exchange, $st\to0$.)} Both descriptions are close
			to $p_0(\cdot,t)$, namely
			$\|p^{\swap}(\cdot,t)-p_0(\cdot,t)\|_{L^1}$ and
			$\|\widehat p(\cdot,t)-p_0(\cdot,t)\|_{L^1}$ are at most
			$2(1-e^{-st})$, which tends to $0$.
			\item[\rm(ii)] \emph{(Fast exchange, $s\to\infty$, $t$ fixed.)}
			The two limits differ by one symmetrisation,
			\begin{equation}\label{eq:fast-compare}
				p^{\swap}\longrightarrow
				p_\infty:=\tfrac12\bigl[\bar p_0(x,y,t)+\bar p_0(y,x,t)\bigr],
				\qquad
				\widehat p\longrightarrow\bar p_0(x,y,t),
			\end{equation}
			as $s\to\infty$ at fixed $t$.
			\item[\rm(iii)] \emph{(Long times, $s>0$ fixed.)} Let $G^{(s)}_t$ be
			the normal density on $\RR^2$ with mean $(\bar m_t,\bar m_t)$,
			$\bar m_t:=\tfrac{x_0+y_0}{2}+\bar\mu t$, and covariance
			$t\,\Sigma_s$, where
			\begin{equation}\label{eq:Sigma-s}
				\Sigma_s:=\sigma_{\rm eff}^2\,I
				+\frac{(\mu_1-\mu_2)^2}{4s}
				\begin{pmatrix}1&-1\\-1&1\end{pmatrix}.
			\end{equation}
			Then, as $t\to\infty$,
			\begin{equation}\label{eq:long-time}
				\bigl\|p^{\swap}(\cdot,t)-G^{(s)}_t\bigr\|_{L^1}\longrightarrow0,
				\qquad
				\bigl\|\widehat p(\cdot,t)-G^{(s)}_t\bigr\|_{L^1}\longrightarrow0,
			\end{equation}
			so that the two descriptions become indistinguishable in $L^1$ at
			every exchange rate.
		\end{enumerate}
	\end{proposition}
	
	\begin{proof}
		See Appendix~\ref{app:asymptotics}.
	\end{proof}
	
	For $t\ll1/s$ an exchange is unlikely and both
	descriptions are close to $p_0$. For $t\gg1/s$ and $s$ large the
	switching pair is
	already a single averaged cloud centred at $(x_0,y_0)+\bar\mu t(1,1)$,
	while the swapping pair is close to the balanced mixture of this cloud and
	of its reflection, centred at $(y_0,x_0)+\bar\mu t(1,1)$. The two limits
	coincide if and only if $x_0=y_0$, that is, if and only if $\bar p_0(\cdot,t)$
	is exchangeable, as Theorem~\ref{thm:fdd} with $J=1$ predicts with
	$\mathcal W$ the uniform average over $\mathfrak S_2$, and in position space fast exchange
	does not blur the two coordinates but duplicates them across the diagonal.
	For long times the spread grows like $\sqrt t$ while the offset between
	the two centres stays of order one, so by (iii) both descriptions approach
	the same exchangeable Gaussian $G^{(s)}_t$ at every fixed $s>0$, with no
	limit $s\to\infty$ involved.
	
	\subsection{Moments and two-time structure}\label{ssec:bench-moments}
	
	\begin{proposition}\label{prop:moments-compare}
		In the benchmark of Definition~\ref{def:bench}, write $e_t:=e^{-2st}$. For
		the switching pair,
		\begin{align}
			\EE\widehat X_t&=x_0+\bar\mu t+\frac{\delta}{2s}(1-e_t),
			\label{eq:rs-mean}\\
			\Var\bigl(\widehat X_t\bigr)&=\sigma_{\rm eff}^2\,t
			+\frac{\sigma_1^2-\sigma_2^2}{4s}(1-e_t)+\delta^2\Omega_t,
			\qquad
			\Omega_t:=\frac1s\Bigl[t-\frac{1-e_t}{s}+\frac{1-e_t^2}{4s}\Bigr],
			\label{eq:rs-var}\\
			\Cov\bigl(\widehat X_t,\widehat Y_t\bigr)&=-\delta^2\Omega_t,
			\label{eq:rs-cov}
		\end{align}
		and the corresponding formulas for $\widehat Y$ are obtained by
		exchanging the indices $1,2$ and $x_0,y_0$. For the swapping pair,
		\begin{align}
			\EE X_t&=\EE\widehat X_t-\frac{d_0}{2}(1-e_t)
			=\frac{x_0+y_0}{2}+\frac{d_0}{2}e_t+\bar\mu t
			+\frac{\delta}{2s}(1-e_t),
			\label{eq:mean_X}\\
			\Var(X_t)&=\Var\bigl(\widehat X_t\bigr)+\Delta_t,
			\label{eq:var_X}\\
			\Cov(X_t,Y_t)&=\Cov\bigl(\widehat X_t,\widehat Y_t\bigr)-\Delta_t,
			\qquad
			\Delta_t:=\frac{d_0^2}{4}\bigl(1-e_t^2\bigr)
			+\frac{d_0\delta}{2s}\bigl(1-e_t\bigr)^2 .
			\label{eq:cov}
		\end{align}
		In particular
		\begin{equation}\label{eq:mean_diff}
			\EE X_t-\EE Y_t=d_0e_t+\frac{\mu_1-\mu_2}{2s}(1-e_t),
		\end{equation}
		relaxing at the rate $2s$ whatever the other parameters.
	\end{proposition}
	
	\begin{proof}
		See Appendix~\ref{app:density}.
	\end{proof}
	
	The mean difference $-\tfrac{d_0}{2}(1-e_t)$ in \eqref{eq:mean_X} involves no
	drift and no volatility, because on the sector $\{\Lambda_t=\tau\}$, where
	the two descriptions differ, the occupation time is symmetric about $t/2$
	and the drift contributions of the two regimes cancel. At the level of
	first moments the whole discrepancy is thus the symmetrisation of the
	initial data, which as $s\to\infty$ becomes the difference
	$\mathcal W\bar\rho_t-\bar\rho_t$ of the two one-time limits of
	Theorem~\ref{thm:fdd}. For $t\gg(2s)^{-1}$ both variances grow linearly, at the rate
	\begin{equation}\label{eq:Deff}
		D_{\mathrm{eff}}=\sigma_{\rm eff}^2+\frac{(\mu_1-\mu_2)^2}{4s},
	\end{equation}
	which we call the effective diffusivity.
	
	\begin{corollary}
		\label{cor:asym-compare}
		In the benchmark of Definition~\ref{def:bench}, for every $s>0$ and every
		$t\ge0$,
		\begin{equation}\label{eq:var_diff}
			\Var(X_t)-\Var(Y_t)
			=\Var\bigl(\widehat X_t\bigr)-\Var\bigl(\widehat Y_t\bigr)
			=\frac{\sigma_1^2-\sigma_2^2}{2s}\bigl(1-e^{-2st}\bigr),
		\end{equation}
		which tends to $(\sigma_1^2-\sigma_2^2)/(2s)$ as $t\to\infty$, and, as
		$t\to\infty$,
		\begin{equation}
			\Corr(X_t,Y_t)\longrightarrow
			c_\infty(s)=-\frac{\zeta}{1+\zeta}
			\ \longleftarrow\
			\Corr\bigl(\widehat X_t,\widehat Y_t\bigr),
			\qquad
			\zeta:=\frac{(\mu_1-\mu_2)^2}{2s(\sigma_1^2+\sigma_2^2)},
			\label{eq:rho_infty}
		\end{equation}
		so the two descriptions agree in both cases. The mean asymmetries, in
		contrast, differ by exactly the initial separation,
		\begin{equation}\label{eq:persistent_asym}
			\EE X_t-\EE Y_t\ \longrightarrow\ \frac{\mu_1-\mu_2}{2s},
			\qquad
			\EE\widehat X_t-\EE\widehat Y_t\ \longrightarrow\
			d_0+\frac{\mu_1-\mu_2}{2s}.
		\end{equation}
	\end{corollary}
	
	\begin{proof}
		For \eqref{eq:var_diff}, $\Delta_t$ is invariant under
		$(x_0,y_0,\mu_1,\mu_2)\mapsto(y_0,x_0,\mu_2,\mu_1)$, so the difference
		terms cancel in $\Var X_t-\Var Y_t$, which therefore equals its switching
		counterpart for every $t$. The latter is read off \eqref{eq:rs-var}, since the
		$\delta^2\Omega_t$ term is common to both coordinates and $e_t=e^{-2st}$. For
		\eqref{eq:rho_infty}, both variances grow like $D_{\mathrm{eff}}t$ and
		both covariances like $-\tfrac{\delta^2}{s}t$, because $\Omega_t\sim t/s$ and
		$\Delta_t$ is bounded, and the ratio is
		$-\tfrac{\delta^2/s}{D_{\mathrm{eff}}}=-\zeta/(1+\zeta)$.
		Equation~\eqref{eq:persistent_asym} follows from \eqref{eq:mean_diff}
		and \eqref{eq:rs-mean}.
	\end{proof}
	
	The quantities that coincide measure how unevenly the coefficients have
	been distributed between the coordinates, which the exchange does in the
	same way in both pictures. The mean asymmetry is the
	only one that differs, and it differs exactly by the initial separation.
	When $\mu_1\ne\mu_2$, $\zeta\to\infty$ and $c_\infty(s)\to-1$ as
	$s\to0^+$, whereas at $s=0$ the coordinates are independent and their
	correlation vanishes at every time. Hence
	\[
	\lim_{s\to0^+}\lim_{t\to\infty}\Corr(X_t,Y_t)=-1,
	\qquad
	\lim_{t\to\infty}\lim_{s\to0^+}\Corr(X_t,Y_t)=0 ,
	\]
	and the two limits do not commute. At long times the one-time laws of the
	two descriptions become indistinguishable in $L^1$
	(Proposition~\ref{prop:regimes}\,\textup{(iii)}), since the offset $d_0$
	between their mean asymmetries is negligible on the diffusive scale
	$\sqrt t$. The two-time structure, by contrast, still distinguishes them.
	
	\begin{proposition}\label{prop:cross}
		In the benchmark of Definition~\ref{def:bench}, with $\zeta$ as in
		\eqref{eq:rho_infty} and $D_{\mathrm{eff}}$ as in \eqref{eq:Deff}, set
		\begin{equation}\label{eq:corr-factor-bench}
			\phi_s(r):=\frac{1}{1+\zeta}
			\Bigl[\frac{1+e^{-2sr}}{2}+\zeta\,e^{-2sr}\Bigr],
			\qquad r\ge0 .
		\end{equation}
		Then, as $t\to\infty$, uniformly in $r\ge0$,
		\begin{align}
			\Corr\bigl(\widehat X_t,\widehat X_{t+r}\bigr)
			&=\sqrt{\frac{t}{t+r}}+O(t^{-1}),
			\label{eq:bench-sw-corr}\\
			\Corr\bigl(X_t,X_{t+r}\bigr)
			&=\phi_s(r)\,\sqrt{\frac{t}{t+r}}+O(t^{-1}).
			\label{eq:corr-auto}
		\end{align}
		In particular, at fixed lag the switching correlation tends to $1$ and
		the swapping one to $\phi_s(r)$, which decreases at rate $2s$ from
		$\phi_s(0)=1$ to
		$\phi_s(\infty)=\frac{1}{2(1+\zeta)}=\frac{\sigma_{\rm eff}^2}{2D_{\mathrm{eff}}}\le\frac12$.
		At proportional lag $r=\alpha t$, with $\alpha>0$,
		\begin{equation}\label{eq:prop-scale}
			\lim_{t\to\infty}\Corr\bigl(\widehat X_t,\widehat X_{(1+\alpha)t}\bigr)
			=\frac{1}{\sqrt{1+\alpha}},
			\qquad
			\lim_{t\to\infty}\Corr\bigl(X_t,X_{(1+\alpha)t}\bigr)
			=\frac{\sigma_{\rm eff}^2}{2D_{\mathrm{eff}}}\cdot
			\frac{1}{\sqrt{1+\alpha}} .
		\end{equation}
	\end{proposition}
	
	\begin{proof}
		By Proposition~\ref{prop:moments-compare}, $\Omega_t=t/s+O(1)$ and
		$\Delta_t=O(1)$. Hence $\Var(\widehat X_t)$ and $\Var(X_t)$ are both
		$D_{\mathrm{eff}}t+O(1)$ and
		$\Cov(X_t,Y_t)=-(D_{\mathrm{eff}}-\sigma_{\rm eff}^2)t+O(1)$, with
		$O(1)$ terms bounded for all $t\ge0$. Consequently
		$\sqrt{\Var(X_t)\Var(X_{t+r})}=D_{\mathrm{eff}}\sqrt{t(t+r)}\,
		\bigl(1+O(t^{-1})\bigr)$ uniformly in $r$, and the same holds for the
		switching pair. It therefore suffices to show that the two covariances
		are $D_{\mathrm{eff}}t+O(1)$ and $D_{\mathrm{eff}}\,t\,\phi_s(r)+O(1)$.
		
		For the switching pair, write
		$\widehat X_{t+r}-\widehat X_t=\int_t^{t+r}\mu_{\Lambda_u(1)}\dif u
		+\int_t^{t+r}\sigma_{\Lambda_u(1)}\dif W^1_u$. Conditionally on $\Lambda$ the
		stochastic integral is centred and independent of $\widehat X_t$, so
		\[
		\Cov\bigl(\widehat X_t,\widehat X_{t+r}\bigr)
		=\Var\bigl(\widehat X_t\bigr)
		+\delta^2\int_0^t\!\!\int_t^{t+r}\Cov(R_u,R_w)\dif w\dif u ,
		\]
		and the bound $0\le\Cov(R_u,R_w)=e^{-2s(w-u)}(1-e^{-4su})\le e^{-2s(w-u)}$
		for $u<w$ bounds the double integral by $(2s)^{-2}$.
		
		For the swapping pair, \eqref{eq:mean_X} started from $(x,y)$ and the
		Markov property give
		$\EE[X_{t+r}\mid\FF_t]=\tfrac{1+e_r}{2}X_t+\tfrac{1-e_r}{2}Y_t+c(r)$, with
		$e_r=e^{-2sr}$ and $c(r)$ deterministic, so
		\[
		\Cov(X_t,X_{t+r})=\tfrac{1+e_r}{2}\Var(X_t)+\tfrac{1-e_r}{2}\Cov(X_t,Y_t)
		=D_{\mathrm{eff}}\,t\Bigl[e_r+(1-e_r)\frac{\sigma_{\rm eff}^2}{2D_{\mathrm{eff}}}\Bigr]+O(1).
		\]
		The bracket equals $\phi_s(r)$ because
		$D_{\mathrm{eff}}/\sigma_{\rm eff}^2=1+\zeta$. The particular cases follow
		by letting $t\to\infty$ at fixed $r$, and at $r=\alpha t$, where $e_r\to0$.
	\end{proof}
	
	At fixed lag the switching marginal behaves like a Brownian motion, whose
	correlation $\sqrt{t/(t+r)}$ also tends to $1$, whereas the swapping
	marginal keeps a genuine dependence on the lag. At proportional lag both
	decorrelate with the same Brownian profile $(1+\alpha)^{-1/2}$, the
	swapping one damped by $\phi_s(\infty)$. The exchange does not change the
	rate at which memory is lost with the scale, but within a time of order
	$1/s$ it removes at least half of it. The mechanism is that the jump of the swapping coordinate at an exchange
	is \emph{macroscopic on the diffusive scale}. At an exchange the
	coordinate moves by $|X_t-Y_t|$, which is of order $\sqrt t$, so a single
	swap displaces it by an amount comparable to its own fluctuation however
	large $t$ is, which no process with continuous paths can do.
	
	The benchmark also shows the lack of normality of
	Corollary~\ref{cor:non-normal} in closed form. At a fixed lag $r>0$,
	$\phi_s(r)\to\tfrac12$ as $s\to\infty$, since $\zeta\to0$ and
	$e^{-2sr}\to0$, while $\phi_s(0)=1$ for every $s$. The same jump appears at
	fixed times in the limit of Theorem~\ref{thm:fdd}\,\textup{(i)}. With
	constant rates $\nu_{\mathbf y}$ is uniform, so the replicas that carry
	species $1$ at the times $t$ and $t+r$ are independent and uniform on
	$\{1,2\}$, and they coincide with probability $\tfrac12$. The coordinates
	of $\bar{\mathbf Y}$ are independent Brownian motions with drift $\bar\mu$
	and variance $\sigma_{\rm eff}^2$ per unit time, started at $x_0$ and $y_0$,
	so the limit $(X^\infty_t,X^\infty_{t+r})$ of $(X_t,X_{t+r})$ satisfies
	\[
	\Cov\bigl(X^\infty_t,X^\infty_{t+r}\bigr)=\tfrac12\sigma_{\rm eff}^2t,
	\qquad
	\Var\bigl(X^\infty_t\bigr)=\sigma_{\rm eff}^2t+\tfrac14d_0^2,
	\qquad r>0 .
	\]
	Its correlation tends, as $r\downarrow0$, to
	$\sigma_{\rm eff}^2t/(2\sigma_{\rm eff}^2t+\tfrac12d_0^2)\le\tfrac12$, while
	it equals $1$ at $r=0$, and for large $t$ it behaves like
	$\tfrac12\sqrt{t/(t+r)}$, in agreement with $\phi_s(r)\to\tfrac12$.
	
\appendix
\section{Construction of the two processes}\label{app:piecing}

Let $\mathsf F$ be a locally compact separable metric space,
$(\PP^0_x)_{x\in\mathsf F}$ a Markov family on
$D([0,\infty),\mathsf F)$, the \emph{base process}, and
$\Theta_j:\mathsf F\to\mathsf F$, $j$ in a finite set $\mathsf J$, measurable maps with bounded
measurable rates $\gamma_j:\mathsf F\to[0,\infty)$, with
$\Gamma_j:=\sup_{\mathsf F}\gamma_j$ and $\Gamma:=\sum_j\Gamma_j>0$. The
\emph{clock} is a Poisson process of rate $\Gamma$ with rings
$T_1<T_2<\cdots$, $T_0:=0$, and independent marks $(j_n,U_n)$, independent
of the ring times, where $j_n=j$ with probability $\Gamma_j/\Gamma$ among
the indices with $\Gamma_j>0$ and $U_n$ is then uniform on
$[0,\Gamma_{j_n}]$.

\emph{Construction.} The space $\mathsf F$ is Polish, hence so is
$D([0,\infty),\mathsf F)$, and since $x\mapsto\PP^0_x$ is a
probability kernel, the randomisation lemma
\cite[Lemma~3.22]{kallenberg2002} gives a measurable map
$\mathfrak p:\mathsf F\times[0,1]\to D([0,\infty),\mathsf F)$ such that
$\mathfrak p(x,V)$ has law $\PP^0_x$ whenever $V$ is uniform on $[0,1]$.
Let $V_0,V_1,\dots$ be independent and uniform on $[0,1]$, independent of
the clock. The \emph{piecing-out} $\xi$ of the base process by
$(\Theta_j,\gamma_j)_j$, started at $x\in\mathsf F$, is defined
recursively by $\xi_0:=x$ and, for $n\ge0$, by the \emph{$n$-th piece}
$\omega^n:=\mathfrak p(\xi_{T_n},V_n)$,
\[
\xi_t:=\omega^n_{t-T_n}\quad\text{for }t\in[T_n,T_{n+1}),\qquad
\xi_{T_{n+1}}:=
\begin{cases}
	\Theta_{j_{n+1}}\bigl(\xi_{T_{n+1}-}\bigr)
	&\text{if }U_{n+1}\le\gamma_{j_{n+1}}\bigl(\xi_{T_{n+1}-}\bigr),\\
	\xi_{T_{n+1}-}&\text{otherwise,}
\end{cases}
\]
where $\xi_{T_{n+1}-}$ is the left limit of $\omega^n$ at $T_{n+1}-T_n$. By
construction $\xi$ is a measurable function of $x$, of the clock and of
$(V_n)_{n\ge0}$, and on each $[T_n,T_{n+1})$ it is a copy of the base
process started at $\xi_{T_n}$ which, given $\xi_{T_n}$, is independent of
the clock and of the past, as in Definition~\ref{def:N-swap}.

The swapping process is the piecing-out with $\mathsf F=E^N$, the decoupled
process as base, the maps $\tau_{ab}$ and the rates $\gamma_{ab}$. The
switching process is the piecing-out with
$\mathsf F=E^N\times\mathfrak{S}_N$, the base process in which the label is
frozen at $\theta$ and replica $k$ follows species $\theta(k)$,
independently of the other replicas, which is a Markov family because its
coordinates are independent Markov families, the maps
$\Theta_{ab}(\mathbf{x},\theta)=(\mathbf{x},\tau_{ab}\theta)$ and the rates
$(\mathbf{x},\theta)\mapsto\gamma_{ab}(\theta\mathbf{x})$, whose suprema are
again the $\Gamma_{ab}$. In both cases the clock is the exchange clock of
Definition~\ref{def:clock}, with the pairs $\{a,b\}$ as indices, and
Theorem~\ref{thm:exist} is part~\textup{(i)} of the following result
applied to these two piecing-outs. Part~\textup{(ii)} records two facts of
the construction that are used in Appendix~\ref{app:switching}.

\begin{theorem}[Piecing-out]\label{thm:piecing}
	Let $\FF_t$ be the $\sigma$-algebra generated by the rings up to time
	$t$, with their marks, and by $\xi_u$, $u\le t$.
	\begin{enumerate}
		\item[\textup{(i)}] For every $x\in\mathsf F$, the piecing-out started
		at $x$ is a c\`adl\`ag, non-explosive process, and the resulting
		family is a time-homogeneous Markov family on $\mathsf F$, with respect
		to $(\FF_t)$ and to the natural filtration of $\xi$, whose law is
		determined by the base process and by $(\Theta_j,\gamma_j)_j$.
		\item[\textup{(ii)}] For every $n\ge0$, conditionally on the whole clock
		and on $V_0,\dots,V_{n-1}$, the $n$-th piece has law
		$\PP^0_{\xi_{T_n}}$. For every $t\ge0$, the rings after $t$ with their
		marks, shifted by $t$, form a marked Poisson process with the law of the
		clock, independent of $\FF_t$.
	\end{enumerate}
\end{theorem}

\begin{proof}
	(ii) The point $\xi_{T_n}$ is a measurable function of $x$, of the
	clock up to $T_n$ and of $V_0,\dots,V_{n-1}$, and $V_n$ is independent of
	the clock and of $V_0,\dots,V_{n-1}$, so the first statement follows from
	the choice of $\mathfrak p$. For the second, on the event
	$\{T_n\le t<T_{n+1}\}$ the variables $\xi_u$, $u\le t$, are measurable
	functions of the clock restricted to $[0,t]$ and of $V_0,\dots,V_n$.
	Hence $\FF_t$ is contained in the $\sigma$-algebra generated by the
	clock on $[0,t]$ and by $(V_n)_{n\ge0}$, which is independent of the
	rings after $t$ and their marks, because a marked Poisson process with
	independent marks has independent increments. The shifted rings after
	$t$ form a marked Poisson process with the law of the clock by the
	memorylessness of the exponential law.

	(i) \emph{Paths.} The times $T_{n+1}-T_n$ are independent and
	exponential of rate $\Gamma$, so $T_n\to\infty$ almost surely and a
	bounded time interval meets only finitely many pieces. On
	$[T_n,T_{n+1})$ the path of $\xi$ is the path of the $n$-th piece
	shifted by $T_n$, which is c\`adl\`ag. At $T_{n+1}$ it has the left
	limit of the $n$-th piece, and it is right-continuous because the next
	piece starts at $\xi_{T_{n+1}}$. The paths of $\xi$ are thus c\`adl\`ag
	on $[0,\infty)$ with values in $\mathsf F$, and $\xi$ does not explode.

	\emph{Markov property.} Fix $t\ge0$ and $n\ge0$, put
	$\mathcal I_n:=\{T_n\le t<T_{n+1}\}\in\FF_t$, and let $\mathcal H_n$ be the
	$\sigma$-algebra generated by the whole clock and by $V_0,\dots,V_{n-1}$.
	The variable $r:=t-T_n$ and the event $\mathcal I_n$ are $\mathcal H_n$-measurable,
	and by~(ii) the $n$-th piece has law $\PP^0_{\xi_{T_n}}$ given
	$\mathcal H_n$. The base family is Markov at every fixed time, and the
	maps $(r,\omega)\mapsto(\omega_{\cdot\wedge r},\omega_{r+\cdot})$ are
	jointly measurable on $[0,\infty)\times D([0,\infty),\mathsf F)$, so, by
	conditioning first on $\mathcal H_n$ and then applying the Markov property
	at the frozen time $r$, the remainder $(\omega^n_{r+u})_{u\ge0}$ has, on
	$\mathcal I_n$ and conditionally on $\mathcal H_n$ and on
	$(\omega^n_{u\wedge r})_{u\ge0}$, the law $\PP^0_{\omega^n_r}=\PP^0_{\xi_t}$.
	On $\mathcal I_n$, the $\sigma$-algebra $\FF_t$ is contained in the one
	generated by $\mathcal H_n$ and $(\omega^n_{u\wedge r})_{u\ge0}$, so the
	remainder has law $\PP^0_{\xi_t}$ given $\FF_t$ and the whole clock. This
	conditional law does not involve the clock, so given $\FF_t$ the remainder
	is independent of the rings after $t$, whose conditional law is that of
	the clock by~(ii). The process after $t$ is obtained from the remainder,
	from the rings after $t$ with their marks and from $V_{n+1},V_{n+2},\dots$,
	which are independent of everything else, by the rule that defines the
	piecing-out started at $\xi_t$. Since the events $\mathcal I_n$, $n\ge0$, cover the
	probability space up to a null set, the conditional law of
	$(\xi_{t+u})_{u\ge0}$ given $\FF_t$ is the law of the piecing-out started at
	$\xi_t$. It depends on $\FF_t$ only through $\xi_t$ and does not depend on
	$t$, which gives the Markov property and its time-homogeneity, in the
	natural filtration of $\xi$ as well by the tower property. The Markov
	property of the base family is used only at times that are fixed given
	$\mathcal H_n$, because the clock that sets the restart times is
	independent of the pieces.

	\emph{Law.} Since $\xi$ is a measurable function of $x$, of the clock
	and of $(V_n)$, the map $x\mapsto\PP_x(B)$ is measurable for every
	Borel set $B$ by Fubini's theorem. By induction over the pieces, using~(ii),
	the joint law of the clock and of the pieces is determined by the base
	family and by $(\Theta_j,\gamma_j)_j$, which fix the law of the clock
	through the $\Gamma_j$, whatever the choice of $\mathfrak p$. The same holds for
	the law of $\xi$.
\end{proof}

\section{Proof of Corollary~\ref{cor:identical}}\label{app:identical}

\begin{proof}[Proof of Corollary~\ref{cor:identical}]
	Write $P^1_t$ for the semigroup of the common species and
	$P^{\mathrm{dec}}_t=(P^1_t)^{\otimes N}$ for that of the decoupled
	process, and recall that $P^1_t$ is a contraction of $L^2(\pi_1)$
	because $\pi_1$ is invariant.

	\emph{Swapping.} Since the species coincide, the decoupled process
	started at $\tau_{ab}\mathbf x$ has the law of $\tau_{ab}$ applied to the
	decoupled process started at $\mathbf x$, so $P^{\mathrm{dec}}_t$ commutes
	with $f\mapsto f\circ\tau_{ab}$, hence with the kernel
	$\mathcal J=I+\Gamma_N^{-1}\mathcal S$ of a ring, with $\mathcal S$ as in
	\eqref{eq:S-op} for the constants $\gamma_{ab}$. Expanding the semigroup of
	the swapping process over the rings of the clock,
	\[
	\begin{aligned}
	P_t&=\sum_{n\ge0}e^{-\Gamma_Nt}\Gamma_N^n\int_{0<t_1<\dots<t_n<t}
	P^{\mathrm{dec}}_{t_1}\mathcal J P^{\mathrm{dec}}_{t_2-t_1}\mathcal J\cdots
	\mathcal JP^{\mathrm{dec}}_{t-t_n}\,\dif t_1\cdots\dif t_n\\
	&=P^{\mathrm{dec}}_t\,e^{t\Gamma_N(\mathcal J-I)}
	=P^{\mathrm{dec}}_t\,e^{t\mathcal S}.
	\end{aligned}
	\]
	The measure $\pi_0=\pi_1^{\otimes N}$ is exchangeable, so
	$\pi_0(f\circ\tau_{ab})=\pi_0(f)$ and $\pi_0(\mathcal Sf)=0$. Hence
	$e^{t\mathcal S}$ is a Markov operator that preserves $\pi_0$, and so is
	$P_t$, which proves that $\pi_0$ is invariant for the swapping process.
	Since $\pi_0$ is exchangeable, $\pi_0(\theta B)=\pi_0(B)$ and its lift
	\eqref{eq:lift} is $\pi_0\otimes\mathrm{Unif}(\mathfrak S_N)$, with
	$\pi_\Phi=\pi_0$.

	The space $L^2_0(\pi_0)$ is the orthogonal sum, over the non-empty
	$S\subset[N]$, of the subspaces $\bigotimes_{k\in S}L^2_0(\pi_1)$ (with
	the constant $1$ on the coordinates outside $S$), each invariant under
	$P^{\mathrm{dec}}_t$, which acts there as $(P^1_t|_{L^2_0})^{\otimes|S|}$.
	If $\|P^1_tg\|\le Ce^{-\lambda t}\|g\|$ for centred $g$, then, the norm of a
	tensor product of operators on Hilbert spaces being the product of the
	norms,
	$\|(P^1_t|_{L^2_0})^{\otimes|S|}\|\le\min(1,Ce^{-\lambda t})^{|S|}\le Ce^{-\lambda t}$,
	so $P^{\mathrm{dec}}_t$ decays at rate $\lambda$ on $L^2_0(\pi_0)$ with
	the same constant, and its best decay rate there is $\lambda_1$, the
	subspaces with $|S|=1$ showing that it is not larger. The factor
	$e^{t\mathcal S}$ is a contraction of $L^2_0(\pi_0)$, so $P_t$ decays at
	every rate below $\lambda_1$. Conversely, for $g\in L^2_0(\pi_1)$ the
	function $f(\mathbf x)=\sum_kg(x^k)$ is permutation invariant, so
	$\mathcal Sf=0$ and $P_tf=\sum_k(P^1_tg)(x^k)$, whose norm is
	$\sqrt N\,\|P^1_tg\|$ while $\|f\|=\sqrt N\,\|g\|$. On these functions
	$P_t$ decays exactly as $P^1_t$, so $\lambda_\swap=\lambda_1$.

	\emph{Switching.} The positions of the switching process are the
	decoupled process whatever the label, and since the rates are constant the
	label is an independent chain on $\mathfrak S_N$ with generator
	$(\mathcal Qg)(\theta):=\sum_{a<b}\gamma_{ab}\bigl[g(\tau_{ab}\theta)-g(\theta)\bigr]$,
	the generator of the interchange process on the exchange graph. Under the
	identification $L^2(\tilde\pi)=L^2(\pi_0)\otimes L^2(\mathrm{Unif})$ the
	switching semigroup is therefore $P^{\mathrm{dec}}_t\otimes e^{t\mathcal Q}$.
	The centred functions of $L^2(\tilde\pi)$ split orthogonally into the
	three invariant subspaces
	$\bigl(L^2_0(\pi_0)\otimes1\bigr)\oplus\bigl(1\otimes L^2_0\bigr)\oplus
	\bigl(L^2_0(\pi_0)\otimes L^2_0\bigr)$, on which the semigroup acts as
	$P^{\mathrm{dec}}_t$, as $e^{t\mathcal Q}$ and as their tensor product. The
	interchange process is reversible with respect to the uniform law, so the
	best decay rate of $e^{t\mathcal Q}$ on $L^2_0$ is its spectral gap, which
	is $\mathrm{gap}(L_\gamma)$ by the theorem of Caputo, Liggett and
	Richthammer \cite{caputo2010}, whose interchange process multiplies by
	transpositions on the right and is conjugate to ours by
	$\theta\mapsto\theta^{-1}$. By the product of the norms again, the third
	subspace decays at every rate below $\lambda_1+\mathrm{gap}(L_\gamma)$.
	The best decay rate on an orthogonal sum of invariant subspaces being the
	smallest of the rates of the summands,
	$\lambda_\rs=\min(\lambda_1,\mathrm{gap}(L_\gamma))$.

	Finally, if $\mathrm{gap}(L_\gamma)<\lambda_1$,
	Theorem~\ref{thm:gap-decomposition}\,\textup{(ii)} gives
	$\min(\lambda_1,\lambda_\perp)=\mathrm{gap}(L_\gamma)$, that is,
	$\lambda_\perp=\mathrm{gap}(L_\gamma)$.
\end{proof}

\section{Proofs for the fast-exchange limit}\label{app:switching}

Assumption~\ref{ass:fast} is in force throughout this appendix.

\subsection{Preliminaries}

\emph{Test functions.} The functions of $\mathcal D_N$ are the constants and
the linear combinations of products $\bigotimes_kf_k$ with
$f_k\in\mathcal D$, and $\LL_\theta1=0$,
$\LL_\theta f=\sum_k(A_{\theta(k)})_{(k)}f$ for a product $f$. The
convention $(A_m)_{(k)}1:=0$ is consistent when $1\in\mathcal D$, as for
finite $E$. Indeed, the martingale problem for $f=1$ shows that
$\int_0^tA_m1(X_u)\dif u$ is a continuous martingale of finite variation,
hence zero, so $A_m1(X_u)=0$ for almost every $u$, almost surely, and since
$A_m1$ is continuous by \textup{(F2)} and the paths are right-continuous,
$A_m1(x)=0$ for every starting point $x$. Thus $\LL_\theta$ maps
$\mathcal D_N$ into $C_b(E^N)$ by \textup{(F2)}.

The functions of $\mathcal D_N$ separate the points of $E^N$. If
$\mathbf y\ne\mathbf y'$ differ in coordinate $k$, take $g\in\mathcal D$
separating $y^k$ from $y'^k$ and, for the other coordinates, functions of
$\mathcal D$ equal to $1$ on a compact set containing all the coordinates of
both points. The same construction with a countable separating family of
$\mathcal D$ and an exhausting sequence of compact sets gives a countable
family of products that separates the points of $E^N$.

\emph{Density.} Every $f\in C_b(E^N)$ is the limit, uniformly on compact
sets and with $\sup_n\|f_n\|_\infty\le\|f\|_\infty+1$, of functions
$f_n\in\mathcal D_N$. Indeed, let $\mathcal K_n$ be compact sets with
$\mathcal K_n\subset\operatorname{int}\mathcal K_{n+1}$ and $\bigcup_n\mathcal K_n=E^N$, and let
$\mathsf c_n\in\mathcal D_N$ be a product of functions of $\mathcal D$ with
$0\le\mathsf c_n\le1$ and $\mathsf c_n=1$ on $\mathcal K_n$, whose support $\mathcal K'_n$ is compact.
The restrictions of $\mathcal D_N$ to $\mathcal K'_n$ form an algebra that contains
the constants and separates the points, so by the Stone--Weierstrass theorem
there is $\tilde f_n\in\mathcal D_N$ with $|\tilde f_n-f|\le1/n$ on $\mathcal K'_n$. Then
$f_n:=\tilde f_n\mathsf c_n\in\mathcal D_N$ vanishes outside $\mathcal K'_n$, is bounded by
$\|f\|_\infty+1/n$ on $\mathcal K'_n$, and satisfies $|f_n-f|\le1/n$ on $\mathcal K_n$; every
compact set lies in some $\mathcal K_n$.

\emph{Proof of \eqref{eq:cc}.} A family that converges as $s\to\infty$ need
not be tight as a whole, but the first marginals of $\rho^{(s)}_0$ are tight
as $s\to\infty$. Indeed, given $\varepsilon>0$ take a compact $\mathcal K$ with
$\bar\rho_0(\mathcal K)\ge1-\varepsilon$ and a compact $\mathcal K'$ whose
interior contains $\mathcal K$; by the portmanteau theorem, applied to this
interior, the mass of $\mathcal K'$ under the first marginal of
$\rho^{(s)}_0$ is eventually at least $1-2\varepsilon$.
So for $\varepsilon>0$ there are $s_0$ and a compact $K_0$ with
$\rho^{(s)}_0(K_0^c\times\mathfrak S_N)\le\varepsilon$ for every $s\ge s_0$,
and conditioning on the initial point and applying \textup{(F1)} gives
\eqref{eq:cc}.

\emph{Compact containment on a countable space.} Let $E$ be countable, with
the jump rates of the species bounded by $\bar q$, let
$\widetilde P_m:=I+A_m/\bar q$ be the kernels of the uniformised chains, with
$A_m$ the rate matrix of species $m$, and $\bar P:=N^{-1}\sum_m\widetilde P_m$,
so that $\widetilde P_m\le N\bar P$. By uniformisation, the switching process
can be realised with each replica moving at the rings of its own Poisson
clock of rate $\bar q$, independent of the other clocks and of the exchange
clock, by a step drawn from $\widetilde P_m(x,\cdot)$, where $m$ is the
species it carries and $x$ its position. Let $\mathsf S^k_0=x^k,\mathsf S^k_1,\dots$ be the
successive positions of replica $k$ at its own rings. Given everything that
precedes it, the $i$-th step of replica $k$ has a law dominated by
$N\bar P(x,\cdot)$, and by induction on $j$,
$\PP\bigl((\mathsf S^k_i)_{i\le j}\in A\bigr)\le N^j\,\bar\PP_{x^k}\bigl((\mathsf S_i)_{i\le j}\in A\bigr)$,
where $\bar\PP_x$ is the law of the chain $(\mathsf S_i)$ with kernel $\bar P$
started at $x$. If $\mathcal V(\widehat{\mathbf X}^{(s)}_t)\ge K$ for some
$t\le T$, then $w(\widehat X^k_t)\ge K/N$ for some $k$, and either the clock
of replica $k$ rings more than $j$ times in $[0,T]$ or one of
$\mathsf S^k_0,\dots,\mathsf S^k_j$ satisfies $w\ge K/N$. Hence, for every $j\ge1$,
\[
\PP_{(\mathbf x,\theta)}\Bigl(\sup_{t\le T}
\mathcal V\bigl(\widehat{\mathbf X}^{(s)}_t\bigr)\ge K\Bigr)
\le N\,\PP\bigl(\mathrm{Poisson}(\bar qT)>j\bigr)
+N^j\sum_{k\le N}\bar\PP_{x^k}\Bigl(\max_{i\le j}w(\mathsf S_i)\ge K/N\Bigr).
\]
The first term is small for $j$ large, and for fixed $j$ the second tends to
$0$ as $K\to\infty$, uniformly in $\theta$, in $s$ and in $\mathbf x$ in a
compact, hence finite, set $K_0$. This gives \textup{(F1)}.

\subsection{Averaging}

We use the averaging theorem of Kurtz in the following form, which is
\cite[Theorem~2.1 and Example~2.3]{kurtz1992} specialised to a finite fast
space, with conditions (K1)--(K3) corresponding to (2.1)--(2.3) there with
vanishing error terms and (K4) to (2.7).
The fast space being finite, hence compact, the relative compactness required
there of the fast variables holds automatically.

\begin{lemma}[Kurtz~\cite{kurtz1992}]\label{lem:kurtz}
	Let $E_1$ be a complete separable metric space, $E_2$ a finite set,
	$\mathcal D(\mathcal A)$ a class of bounded continuous functions on $E_1$,
	dense in the bounded continuous functions for bounded uniform convergence
	on compact sets, and $\mathcal A:\mathcal D(\mathcal A)\to C(E_1\times E_2)$.
	For $n\ge1$ let $X_n$ and $Y_n$ be processes with paths in
	$D([0,\infty),E_1)$ and $D([0,\infty),E_2)$, adapted to a filtration
	$\{\FF^n_t\}$, and let $\Upsilon_n(\dif u\times\dif y)
	:=\delta_{Y_n(u)}(\dif y)\dif u$ be the occupation measure of $Y_n$.
	Assume that
	\begin{enumerate}
		\item[\textup{(K1)}] for every $\varepsilon,T>0$ there is a compact
		$\mathcal K\subset E_1$ with
		$\inf_n\PP(X_n(t)\in\mathcal K\ \text{for}\ t\le T)\ge1-\varepsilon$;
		\item[\textup{(K2)}] for every $f\in\mathcal D(\mathcal A)$,
		$f(X_n(t))-\int_0^t\mathcal Af(X_n(u),Y_n(u))\dif u$ is an
		$\{\FF^n_t\}$-martingale;
		\item[\textup{(K3)}] for every $f\in\mathcal D(\mathcal A)$ and $T>0$
		there is $p>1$ with
		$\sup_n\EE\int_0^T|\mathcal Af(X_n(u),Y_n(u))|^p\dif u<\infty$.
	\end{enumerate}
	Then $\{(X_n,\Upsilon_n)\}$ is relatively compact in
	$D([0,\infty),E_1)\times\mathcal L_m(E_2)$, where $\mathcal L_m(E_2)$ is
	the space of measures on $[0,\infty)\times E_2$ with $\Upsilon([0,t]\times
	E_2)=t$, and for every limit point $(X,\Upsilon)$ and every
	$f\in\mathcal D(\mathcal A)$ the process
	$f(X_t)-\int_{[0,t]\times E_2}\mathcal Af(X_u,y)\,\Upsilon(\dif u\times\dif y)$
	is a martingale of the filtration
	$\FF^{X,\Upsilon}_t:=\sigma\{X_u,\ \Upsilon([0,u]\times H):\ u\le t,\ H\subset E_2\}$.
	If moreover
	\begin{enumerate}
		\item[\textup{(K4)}] there are $\beta_n\to\infty$, a set $\mathcal C$ of
		functions on $E_2$ and, for every $g\in\mathcal C$, a bounded continuous
		function $\mathcal Bg$ on $E_1\times E_2$ such that
		$g(Y_n(t))-\beta_n\int_0^t\mathcal Bg(X_n(u),Y_n(u))\dif u$ is an
		$\{\FF^n_t\}$-martingale,
	\end{enumerate}
	then $\Upsilon(\dif u\times\dif y)=\mathfrak m_u(\dif y)\dif u$ for a
	measurable family of probabilities $\mathfrak m_u$ on $E_2$ and, for every
	$g\in\mathcal C$, almost surely
	$\int\mathcal Bg(X_u,y)\,\mathfrak m_u(\dif y)=0$ for almost every $u$.
	When $\mathcal C$ is a finite-dimensional linear space and
	$g\mapsto\mathcal Bg$ is linear, applying this to a basis gives one null
	set for all $g\in\mathcal C$.
\end{lemma}

The filtration $\FF^{X,\Upsilon}$ is the one constructed in the proof of
\cite[Theorem~2.1]{kurtz1992}, and Example~2.3 there treats one function $g$
at a time. The second part rests on the following observation. Dividing the
martingale in \textup{(K4)} by $\beta_n$ gives martingales that are bounded on
compact time intervals uniformly in $n$, and the term $g(Y_n(t))/\beta_n$
tends to $0$ uniformly. Their limits are therefore the processes
$-\int_{[0,t]\times E_2}\mathcal Bg(X_u,y)\,\Upsilon(\dif u\times\dif y)$,
which are martingales with continuous paths of finite variation, hence
vanish.

\begin{lemma}\label{lem:no-fixed}
	Let $X$ be a species, with generator $A_m$, or a solution of the
	martingale problem for $\bar\LL$. Then $\PP(X_t\neq X_{t-})=0$ for every
	$t>0$.
\end{lemma}

\begin{proof}
	We treat the second case, the first being the same with $E$, $A_m$ and
	$\mathcal D$ in place of $E^N$, $\bar\LL$ and $\mathcal D_N$. Fix $t>0$ and
	let $f$ be a product of functions of $\mathcal D$, so that $f^2$ is one as
	well. By \textup{(F2)} and Lemma~\ref{lem:frozen}, $\bar\LL f$ and
	$\bar\LL f^2$ are bounded, and the martingale property at time
	$t-\varepsilon$, together with
	$(a-b)^2=a^2-b^2-2b(a-b)$, gives, for $0<\varepsilon<t$,
	\[
	\EE\bigl[\bigl(f(X_t)-f(X_{t-\varepsilon})\bigr)^2\bigr]
	=\EE\int_{t-\varepsilon}^{t}\bar\LL f^2(X_r)\dif r
	-2\,\EE\Bigl[f(X_{t-\varepsilon})\int_{t-\varepsilon}^{t}
	\bar\LL f(X_r)\dif r\Bigr]\le c\,\varepsilon .
	\]
	By Fatou's lemma along $\varepsilon=1/n$, $f(X_t)=f(X_{t-})$ almost surely.
	A countable family of such products separates points, so $X_t=X_{t-}$
	almost surely.
\end{proof}

\begin{lemma}
	\label{lem:mart-switching}
	Let $s>0$. In the filtration $(\FF_t)$ generated by the exchange clock and
	the switching process up to time $t$, with rates $s\gamma_{ab}$, the
	following hold.
	\begin{enumerate}
		\item[\textup{(i)}] For every $f\in\mathcal D_N$,
		$f(\widehat{\mathbf{X}}^{(s)}_t)-f(\widehat{\mathbf{X}}^{(s)}_0)
		-\int_0^t\LL_{\Lambda_u}f(\widehat{\mathbf{X}}^{(s)}_u)\dif u$ is a
		martingale.
		\item[\textup{(ii)}] For every $g:\mathfrak S_N\to\RR$,
		$g(\Lambda_t)-g(\Lambda_0)-s\int_0^t
		\bigl(\mathcal Q_{\widehat{\mathbf X}^{(s)}_u}g\bigr)(\Lambda_u)\dif u$ is a
		martingale.
	\end{enumerate}
\end{lemma}

\begin{proof}
	(i) It suffices to take a product $f=\bigotimes_kf_k$, and we write
	$\widehat{\mathbf X}$ for $\widehat{\mathbf X}^{(s)}$. Between two rings
	$T_n$ and $T_{n+1}$ the label is fixed, say at $\theta$, and the
	coordinates are independent copies of the species $\theta(k)$, so
	$f_k(\widehat X^k_t)=f_k(\widehat X^k_{T_n})+\mathcal M^k_t+\int_{T_n}^tA_{\theta(k)}f_k(\widehat X^k_r)\dif r$
	with independent martingales $\mathcal M^k$, bounded on compact time
	intervals. The jumps of $\mathcal M^k$ are those of $f_k(\widehat X^k)$, so
	by Lemma~\ref{lem:no-fixed} the $\mathcal M^k$ have no fixed times of
	discontinuity; this is the only use of Lemma~\ref{lem:no-fixed} for the
	species. Conditionally on $\mathcal M^k$, the jump times of $\mathcal M^k$
	are fixed times for $\mathcal M^l$, $l\ne k$, so by Fubini's theorem the
	$\mathcal M^k$ have almost surely no common jumps. For $k\ne l$ the product
	$\mathcal M^k\mathcal M^l$ of the two independent martingales is a
	martingale of the filtration that they generate jointly, so, by
	integration by parts, the covariation $[\mathcal M^k,\mathcal M^l]$ is a
	local martingale of finite variation. It is continuous because the two
	factors have no common jumps, and a continuous local martingale of finite
	variation started at $0$ vanishes. Apply now It\^o's formula to the
	multilinear function $(z_1,\dots,z_N)\mapsto\prod_kz_k$ of the
	semimartingales $f_k(\widehat X^k)$. Since at most one coordinate jumps at
	a time and the function is affine in each coordinate, the jump correction
	vanishes, and the second-order terms involve only the brackets
	$[\mathcal M^k,\mathcal M^l]$, $k\ne l$, which vanish. Hence
	\[
	\dif\prod_kf_k(\widehat X^k_t)
	=\sum_k\prod_{j\ne k}f_j(\widehat X^j_{t-})\,\dif \mathcal M^k_t
	+\LL_\theta f(\widehat{\mathbf X}_t)\,\dif t
	\]
	on each interval between rings. At the rings only the label jumps and the
	positions do not move. Let $(\mathcal G_t)$ be the filtration $(\FF_t)$
	enlarged at time $0$ by the whole clock, with its ring times and marks. In
	$\mathcal G$ the ring times are $\mathcal G_0$-measurable, and by
	Theorem~\ref{thm:piecing}\,\textup{(ii)}, conditionally on
	$\mathcal G_{T_n}$, the piece on $[T_n,T_{n+1})$ is the decoupled process
	started at $\widehat{\mathbf X}_{T_n}$ with the species given by
	$\Lambda_{T_n}$ (Definition~\ref{def:switching}). Write $M$ for the
	process in the statement. The functions $\LL_\theta f$ are bounded, so
	$|M_t|\le c(1+t)$. Conditionally on $\mathcal G_{T_n}$, the increment
	$M_{t\wedge T_{n+1}}-M_{t\wedge T_n}$, $t\ge0$, is a local martingale by the
	identity between rings, and it is bounded on compact time intervals, hence
	a martingale. Summing over the rings, every $M_{\cdot\wedge T_n}$ is a
	martingale for $\mathcal G$, and so is $M$, by bounded convergence, because
	$T_n\to\infty$. Since $M$ is $\FF$-adapted, it is a martingale for $\FF$ by
	the tower property.

	(ii) By Theorem~\ref{thm:piecing}\,\textup{(ii)}, the rings of the pair
	$\{a,b\}$ with their marks form a Poisson random measure on
	$[0,\infty)\times[0,\Gamma_{ab}]$ of intensity $s\,\dif u\,\dif x$ with
	respect to $(\FF_t)$. Such a ring, at a time $T_n$, moves the label from
	$\Lambda_{T_n-}$ to $\tau_{ab}\Lambda_{T_n-}$ when its mark is at most
	$\gamma_{ab}(\Lambda_{T_n-}\widehat{\mathbf X}_{T_n-})$, a predictable
	function. The compensator of the accepted rings of the pair is therefore
	$s\int\gamma_{ab}(\Lambda_u\widehat{\mathbf X}_u)\dif u$
	\cite[Chapter~II]{bremaud1981}, and
	\[
	g(\Lambda_t)-g(\Lambda_0)=\sum_{a<b}\int_0^t
	\bigl[g(\tau_{ab}\Lambda_{u-})-g(\Lambda_{u-})\bigr]\,\dif \mathcal N^{ab}_u,
	\]
	with $\mathcal N^{ab}$ the counting process of the accepted rings of
	$\{a,b\}$, has compensator
	\[
	s\int_0^t\sum_{a<b}\gamma_{ab}(\Lambda_u\widehat{\mathbf X}_u)
	\bigl[g(\tau_{ab}\Lambda_u)-g(\Lambda_u)\bigr]\dif u ,
	\]
	which by \eqref{eq:frozen-gen} is the integral in the statement. The process is
	bounded on compact time intervals, so the local martingale is a
	martingale.
\end{proof}

\begin{proof}[Proof of Theorem~\ref{thm:switching-conv}\,\textup{(i)}]
	Fix a sequence $s_n\to\infty$. We show that the processes
	$\widehat{\mathbf X}^{(s_n)}$ converge in distribution to $\bar{\mathbf Y}$
	in $D([0,\infty),E^N)$. Below $s$ runs over this sequence, along which
	\eqref{eq:cc} reads
	\begin{equation}\label{eq:cc-seq}
		\lim_{K\to\infty}\ \sup_{s}\
		\PP\Bigl(\sup_{t\le T}\mathcal V\bigl(\widehat{\mathbf X}^{(s)}_t\bigr)\ge K\Bigr)=0
		\qquad\text{for every }T<\infty .
	\end{equation}
	The initial laws enter the proof only through \eqref{eq:cc-seq} and the
	convergence of their first marginals.

	\emph{Averaging.} Apply Lemma~\ref{lem:kurtz} with $E_1=E^N$,
	$E_2=\mathfrak S_N$, the processes $\widehat{\mathbf X}^{(s)}$ and
	$\Lambda$, the filtration $(\FF_t)$ of Section~\ref{ssec:fast},
	$\mathcal D(\mathcal A)=\mathcal D_N$ and
	$\mathcal Af(\mathbf y,\theta):=\LL_\theta f(\mathbf y)$, which is bounded
	and continuous by \textup{(F2)}. The density required of
	$\mathcal D(\mathcal A)$ is the density statement of the preliminaries.
	Condition \textup{(K1)} is \eqref{eq:cc-seq}, \textup{(K2)} is
	Lemma~\ref{lem:mart-switching}\,\textup{(i)}, and \textup{(K3)} holds for
	every $p$ because $\mathcal Af$ is bounded. By
	Lemma~\ref{lem:mart-switching}\,\textup{(ii)}, \textup{(K4)} holds with
	$\beta_s=s$, $\mathcal C$ the functions on $\mathfrak S_N$ and
	$\mathcal Bg(\mathbf y,\theta):=(\mathcal Q_{\mathbf y}g)(\theta)$, which is
	bounded and continuous by \textup{(F2)}. Hence, with $\Upsilon^{(s)}$ the
	occupation measure of $\Lambda$, the laws of
	$(\widehat{\mathbf X}^{(s)},\Upsilon^{(s)})$ are relatively compact in
	$D([0,\infty),E^N)\times\mathcal L_m(\mathfrak S_N)$, and every limit point
	$(\mathbf Y,\Upsilon)$ has $\Upsilon(\dif u\times\dif\theta)
	=\mathfrak m_u(\dif\theta)\dif u$ with $\mathfrak m_u$ invariant for
	$\mathcal Q_{\mathbf Y_u}$, for almost every $u$, almost surely, the null
	set being the same for all $g$ because $\mathcal C$ is the
	$N!$-dimensional space of functions on $\mathfrak S_N$. Hence
	$\mathfrak m_u=\nu_{\mathbf Y_u}$ for almost every $u$, by \textup{(F3)}.
	Moreover $\mathbf Y_0\sim\bar\rho_0$, because evaluation at $0$ is
	continuous on $D([0,\infty),E^N)$.

	\emph{Identification.} By Lemma~\ref{lem:kurtz}, for every
	$f\in\mathcal D_N$,
	\[
	f(\mathbf Y_t)-\int_{[0,t]\times\mathfrak S_N}\LL_\theta f(\mathbf Y_u)\,
	\Upsilon(\dif u\times\dif\theta)
	=f(\mathbf Y_t)-\int_0^t\bar\LL f(\mathbf Y_u)\dif u
	\]
	is a martingale of $(\FF^{\mathbf Y,\Upsilon}_t)$, the equality following
	from $\Upsilon(\dif u\times\dif\theta)=\nu_{\mathbf Y_u}(\dif\theta)\dif u$
	and \eqref{eq:barL-def}. It is adapted to the natural filtration of
	$\mathbf Y$, which is smaller, so it is a martingale of the latter by the
	tower property. Thus $\mathbf Y$ solves the martingale problem for
	$\bar\LL$ with initial law $\bar\rho_0$, and by \textup{(F3)} its law is
	that of $\bar{\mathbf Y}$. The sequence $\widehat{\mathbf X}^{(s_n)}$ is
	relatively compact and all its limit points have this law, so it converges
	to it in $D([0,\infty),E^N)$, and the sequence $s_n\to\infty$ was
	arbitrary.

	If the species have continuous paths so has $\widehat{\mathbf X}^{(s)}$.
	The continuous paths form a closed subset of the Skorokhod space, on which
	the Skorokhod topology is that of uniform convergence on compact time
	intervals, so $\bar{\mathbf Y}$ has continuous paths and the convergence
	holds in $C([0,T],E^N)$. The absence of fixed discontinuities is
	Lemma~\ref{lem:no-fixed} applied to $\bar{\mathbf Y}$. Finally, the
	restriction to $[0,T]$ is continuous at every path that is continuous at
	$T$, and $\bar{\mathbf Y}$ has no fixed times of discontinuity, so the
	convergence in $D([0,\infty),E^N)$ gives the statement in $D([0,T],E^N)$
	for every $T$.
\end{proof}

\begin{proof}[Proof of Theorem~\ref{thm:switching-conv}\,\textup{(ii)}]
	Write $\bar\PP_{\mathbf x}$ for the law of $\bar{\mathbf Y}$ started at
	$\mathbf x$, which exists and is unique by \textup{(F3)}. The map
	$\mathbf x\mapsto\bar\PP_{\mathbf x}$ is weakly continuous, hence
	measurable, and so is $\mathbf x\mapsto\bar\PP_{\mathbf x}(B)$ for every
	Borel set $B$. Indeed, let $\mathbf x_n\to\mathbf x$ and choose
	$s_n\to\infty$ such that the law of $\widehat{\mathbf X}^{(s_n)}$ started
	at $(\mathbf x_n,\mathrm{id})$ is within $1/n$ of $\bar\PP_{\mathbf x_n}$
	in the Prokhorov metric, which is possible by~(i) for each fixed $n$.
	By~(i) along this sequence, with initial points converging to $\mathbf x$,
	these laws converge to $\bar\PP_{\mathbf x}$, and hence so do the
	$\bar\PP_{\mathbf x_n}$. Since $\bar\LL$ maps $\mathcal D_N$ into bounded
	functions, by \textup{(F2)} and Lemma~\ref{lem:frozen}, and its martingale
	problem is well posed, $\bar{\mathbf Y}$ is a Markov process in its natural
	filtration, with transition function
	$\bar P_r(\mathbf x,\cdot)=\bar\PP_{\mathbf x}(\bar{\mathbf Y}_r\in\cdot)$
	\cite[Theorem~4.4.2]{ethierkurtz1986}.

	For the renumbering, let the switching process start at
	$(\mathbf x,\mathrm{id})$. By
	Proposition~\ref{prop:markov-unique}\,\textup{(i)},
	$(\sigma\widehat{\mathbf X}^{(s)},\Lambda\sigma^{-1})$ is the switching
	process started at $(\sigma\mathbf x,\sigma^{-1})$. By~(i), in which the
	initial label plays no role, its positions converge to $\bar{\mathbf Y}$
	started at $\sigma\mathbf x$. They also converge to $\sigma\bar{\mathbf Y}$
	with $\bar{\mathbf Y}$ started at $\mathbf x$, by the continuous mapping
	theorem, since $\omega\mapsto\sigma\omega$ is continuous on
	$D([0,\infty),E^N)$. Hence
	$\bar P_t(f\circ\sigma)(\mathbf x)=\bar P_tf(\sigma\mathbf x)$.
\end{proof}

\subsection{Local equilibrium of the label}

\begin{proof}[Proof of Lemma~\ref{lem:local-eq}]
	We may assume $\|g\|_\infty\le1$ and $\|f\|_\infty\le1$, and we write
	$\widehat{\mathbf X}$ for $\widehat{\mathbf X}^{(s)}$ when $s$ is fixed. Two
	facts are used. First, for every family $r_s\to0$,
	\begin{equation}\label{eq:no-oscillation}
		\sup_{u\in[t-r_s,t]}d_N\bigl(\widehat{\mathbf X}^{(s)}_u,
		\widehat{\mathbf X}^{(s)}_t\bigr)\ \longrightarrow\ 0
		\quad\text{in probability.}
	\end{equation}
	Indeed, along any sequence $s_n\to\infty$,
	Theorem~\ref{thm:switching-conv}\,\textup{(i)} and the Skorokhod
	representation give versions converging almost surely in $D([0,T],E^N)$ to
	a version of $\bar{\mathbf Y}$, which is almost surely continuous at $t$ by
	Theorem~\ref{thm:switching-conv}\,\textup{(i)}. Convergence in the Skorokhod
	topology means uniform convergence on $[0,T]$ after time changes $\iota_n$
	with $\sup_u|\iota_n(u)-u|\to0$ \cite[Section~3.5]{ethierkurtz1986}, so
	for $u\in[t-r_{s_n},t]$ the points $\widehat{\mathbf X}^{(s_n)}_u$ are
	uniformly close to $\bar{\mathbf Y}_{\iota_n(u)}$ with $\iota_n(u)\to t$
	uniformly, and continuity at $t$ gives \eqref{eq:no-oscillation}. Second,
	the $\gamma_{ab}$, by \textup{(F2)}, and $\mathbf y\mapsto\nu_{\mathbf y}(g)$,
	by Lemma~\ref{lem:frozen}, are continuous, hence uniformly continuous on
	each $B_K$, and $f$ is uniformly continuous on each $\mathcal K\times B_K$
	with $\mathcal K\subset\mathcal Z$ compact. We write
	\[
	\bar\omega_K(r):=\sup\Bigl\{\sum_{a<b}|\gamma_{ab}(\mathbf y)-\gamma_{ab}(\mathbf y')|:\
	\mathbf y,\mathbf y'\in B_K,\ d_N(\mathbf y,\mathbf y')\le r\Bigr\},
	\]
	which is non-decreasing in $r$ and tends to $0$ as $r\to0$. Since
	$\mathcal V$ and $d_N$ are invariant under permutations,
	$\sum_{a<b}|\gamma_{ab}(\theta\mathbf y)-\gamma_{ab}(\theta\mathbf y')|
	\le\bar\omega_K(d_N(\mathbf y,\mathbf y'))$ for $\mathbf y,\mathbf y'\in B_K$
	and $\theta\in\mathfrak S_N$.

	Fix $\varepsilon>0$. Since the laws of $Z$ are tight as $s\to\infty$, and
	by \eqref{eq:cc}, there are $s_0$, a compact $\mathcal K$ and $K$ with
	$\PP(Z\notin\mathcal K)\le\varepsilon$ and
	$\PP(\sup_{u\le T}\mathcal V(\widehat{\mathbf X}_u)>K)\le\varepsilon$ for
	every $s\ge s_0$. Then choose by Lemma~\ref{lem:frozen} an integer $\ell$
	with $2(1-\mathfrak d_{B_K})^\ell\le\varepsilon$. Finally choose
	$\mathfrak r>0$ with $\ell\,\bar\omega_K(\mathfrak r)\le\varepsilon$ and such
	that $|f(z,\mathbf y)-f(z,\mathbf y')|\le\varepsilon$ and
	$|\nu_{\mathbf y}(g)-\nu_{\mathbf y'}(g)|\le\varepsilon$ for $z\in\mathcal K$,
	$\mathbf y,\mathbf y'\in B_K$ and $d_N(\mathbf y,\mathbf y')\le\mathfrak r$.
	For $s\ge s_0$ with $s>\ell/(t-t')$ put $u_s:=t-\ell/s>t'$ and
	$\mathbf y:=\widehat{\mathbf X}_{u_s}$, and define on $[u_s,t]$ a second
	label $\Lambda'$, with $\Lambda'_{u_s}=\Lambda_{u_s}$, driven by the same
	rings and marks as $\Lambda$ but accepting a ring on $\{a,b\}$ at time $T_n$
	when its mark is at most $\gamma_{ab}(\Lambda'_{T_n-}\mathbf y)$, that is,
	with the positions frozen at $\mathbf y$. By
	Theorem~\ref{thm:piecing}\,\textup{(ii)} the rings and marks on $(u_s,t]$
	are independent of $\FF_{u_s}$, so conditionally on $\FF_{u_s}$ the label
	$\Lambda'$ is the exchange chain frozen at $\mathbf y$ run at speed $s$ from
	$\Lambda_{u_s}$, and $\Lambda'_t$ has law
	$e^{\ell\mathcal Q_{\mathbf y}}(\Lambda_{u_s},\cdot)$, while $\mathcal U$ and
	$Z$ are $\FF_{u_s}$-measurable.

	Let
	$\mathfrak t:=\inf\{u\ge u_s:\mathcal V(\widehat{\mathbf X}_u)>K\text{ or }
	d_N(\widehat{\mathbf X}_u,\mathbf y)>\mathfrak r\}$. At a ring
	$T_n\in(u_s,\mathfrak t]$ the points $\widehat{\mathbf X}_{T_n-}$ and
	$\mathbf y$ lie in $B_K$ at distance at most $\mathfrak r$. As long as
	$\Lambda=\Lambda'$, the two labels take different decisions at such a ring
	on $\{a,b\}$ only if its mark falls between
	$\gamma_{ab}(\Lambda\widehat{\mathbf X}_{T_n-})$ and $\gamma_{ab}(\Lambda\mathbf y)$,
	which, summed over the pairs, happens at rate at most
	$s\,\bar\omega_K(\mathfrak r)$. On $\{\mathfrak t\ge t\}$ every ring of the
	window $(u_s,t]$ is of this kind, and on $\{\mathfrak t<t\}$ one of the two
	conditions that define $\mathfrak t$ is violated at some time of $[u_s,t)$.
	Over the window of length $\ell/s$ this gives
	\[
	\PP(\Lambda_t\neq\Lambda'_t)\le\ell\,\bar\omega_K(\mathfrak r)+\PP(\mathfrak t<t)
	\le2\varepsilon+\PP\Bigl(\sup_{u\in[u_s,t]}d_N(\widehat{\mathbf X}_u,
	\widehat{\mathbf X}_{u_s})>\mathfrak r\Bigr).
	\]
	Since
	$\sup_{u\in[u_s,t]}d_N(\widehat{\mathbf X}_u,\widehat{\mathbf X}_{u_s})
	\le2\sup_{u\in[u_s,t]}d_N(\widehat{\mathbf X}_u,\widehat{\mathbf X}_t)$,
	the last probability tends to $0$ by \eqref{eq:no-oscillation} with
	$r_s=\ell/s$. Now replace successively $\Lambda_t$ by $\Lambda'_t$, at a cost
	at most $2\PP(\Lambda_t\ne\Lambda'_t)$, and $f(Z,\widehat{\mathbf X}_t)$ by
	$f(Z,\mathbf y)$, at a cost at most
	$\EE|f(Z,\widehat{\mathbf X}_t)-f(Z,\mathbf y)|$, which is eventually at most
	a fixed multiple of $\varepsilon$ by \eqref{eq:no-oscillation}, the choice
	of $\mathfrak r$ and the bounds on $\PP(Z\notin\mathcal K)$ and
	$\PP(\sup\mathcal V>K)$. Conditioning on $\FF_{u_s}$,
	$\EE[\mathcal U\,g(\Lambda'_t)f(Z,\mathbf y)]=\EE\bigl[\mathcal U\,f(Z,\mathbf y)\,
	\bigl(e^{\ell\mathcal Q_{\mathbf y}}g\bigr)(\Lambda_{u_s})\bigr]$, which differs
	from $\EE[\mathcal U\,f(Z,\mathbf y)\nu_{\mathbf y}(g)]$ by at most
	$\varepsilon+2\,\PP(\mathbf y\notin B_K)\le3\varepsilon$, by
	\eqref{eq:frozen-mixing} on $\{\mathbf y\in B_K\}$ and the bound $2$ on its
	complement. Finally
	$\EE|f(Z,\mathbf y)\nu_{\mathbf y}(g)-f(Z,\widehat{\mathbf X}_t)
	\nu_{\widehat{\mathbf X}_t}(g)|$ is eventually at most a fixed multiple of
	$\varepsilon$ for the same reasons. Collecting the terms, the $\limsup$ of
	the modulus of the left-hand side of \eqref{eq:local-eq} is at most
	$C\varepsilon$ with $C$ a universal constant, and $\varepsilon$ is
	arbitrary.
\end{proof}

\subsection{Limits of the swapping process}

\begin{proof}[Proof of Corollary~\ref{cor:nontight}]
	(i) Suppose that $\mathbf X^{(s_n)}\Rightarrow\mathbf X$ in
	$D([0,T],E^N)$. Outside a countable set of times the finite-dimensional
	laws of $\mathbf X^{(s_n)}$ converge to those of $\mathbf X$. This is
	\cite[Lemma~3.7.7 and Theorem~3.7.8]{ethierkurtz1986}, stated there in
	$D([0,\infty),E^N)$ and applied after extending paths by constants after
	$T$, which maps $D([0,T],E^N)$ continuously into $D([0,\infty),E^N)$. By
	Theorem~\ref{thm:fdd}\,\textup{(i)}, for $0<t'<t''\le T$ outside this set
	the pair $(\mathbf X_{t'},\mathbf X_{t''})$ has the law of
	$(\Xi'\bar{\mathbf Y}_{t'},\Xi''\bar{\mathbf Y}_{t''})$, with $\Xi'$ and
	$\Xi''$ conditionally independent given $\bar{\mathbf Y}$, of laws
	$\nu_{\bar{\mathbf Y}_{t'}}$ and $\nu_{\bar{\mathbf Y}_{t''}}$, so that
	\[
	\EE\bigl[d_N(\mathbf X_{t'},\mathbf X_{t''})\wedge1\bigr]
	=\EE\sum_{\theta',\theta''\in\mathfrak S_N}\nu_{\bar{\mathbf Y}_{t'}}(\theta')\,
	\nu_{\bar{\mathbf Y}_{t''}}(\theta'')\,
	\bigl(d_N(\theta'\bar{\mathbf Y}_{t'},\theta''\bar{\mathbf Y}_{t''})\wedge1\bigr).
	\]
	Let $t''\downarrow t'$ outside the countable set. The left-hand side tends
	to $0$, because $\mathbf X$ is right-continuous. On the right-hand side
	$\bar{\mathbf Y}$ is right-continuous and $\mathbf y\mapsto\nu_{\mathbf y}$
	is continuous (Lemma~\ref{lem:frozen}), so by bounded convergence it tends
	to the same expression with $\bar{\mathbf Y}_{t'}$ in place of
	$\bar{\mathbf Y}_{t''}$, which is at least
	\[
	\EE\Bigl[\nu_{\bar{\mathbf Y}_{t'}}(\mathrm{id})\,\nu_{\bar{\mathbf Y}_{t'}}(\tau_{kl})\,
	\bigl(d(\bar Y^k_{t'},\bar Y^l_{t'})\wedge1\bigr)\Bigr],
	\]
	since $d_N(\mathbf y,\tau_{kl}\mathbf y)=d(y^k,y^l)$. The graph $G_\gamma$
	is connected, because otherwise the frozen chain would have a closed class
	in each coset of the subgroup of Remark~\ref{rem:components}, against the
	uniqueness in \textup{(F3)}. The rates being positive at every point on the
	edges of $G_\gamma$, the frozen chain is irreducible and $\nu_{\mathbf y}$
	charges every permutation, so this expectation is positive as soon as
	$\PP(\bar Y^k_{t'}\ne\bar Y^l_{t'})>0$. Since $\{y^k\ne y^l\}$ is open,
	right continuity and Fatou's lemma give
	$\liminf_{t'\downarrow t}\PP(\bar Y^k_{t'}\ne\bar Y^l_{t'})
	\ge\PP(\bar Y^k_t\ne\bar Y^l_t)>0$, so $\PP(\bar Y^k_{t'}\ne\bar Y^l_{t'})>0$
	for every $t'$ in an interval $(t,t+r_0)\subset(t,T)$, which contains times
	outside the countable set. This is a contradiction. Finally, if the laws were
	tight as $s\to\infty$, every sequence $s_n\to\infty$ would have a
	subsequence converging in distribution, by Prokhorov's theorem.

	(ii) Fix a sequence $s_n\to\infty$ and write $s$ for $s_n$. Let
	$f\in C_b(E^N)$ with $\|f\|_\infty\le1$, let $\bar f:=\mathcal Wf$, as
	defined after \eqref{eq:W-def}, which is bounded and continuous, and put
	$\mathcal E_t:=f(\mathbf X^{(s)}_t)-\bar f(\widehat{\mathbf X}^{(s)}_t)$, so that
	$|\mathcal E_t|\le2$. By \eqref{eq:replica-rep},
	$f(\mathbf X^{(s)}_t)=\sum_\theta\mathbf 1_{\{\Lambda_t=\theta\}}
	(f\circ\theta)(\widehat{\mathbf X}^{(s)}_t)$. For $u<t$,
	Lemma~\ref{lem:local-eq} with $t'=u$, $\mathcal U=\mathcal E_u/2$, a constant $Z$,
	$g=\mathbf 1_{\{\cdot=\theta\}}$ and $f\circ\theta$ in place of $f$, summed
	over $\theta$, gives $\EE[\mathcal E_u\mathcal E_t]\to0$. By dominated convergence,
	\[
	\EE\Bigl[\Bigl(\int_0^T\mathcal E_t\,\dif t\Bigr)^2\Bigr]
	=2\int_0^T\!\!\int_0^t\EE[\mathcal E_u\mathcal E_t]\,\dif u\,\dif t\ \longrightarrow\ 0 ,
	\]
	so $\eta^{(s)}_T(f)-T^{-1}\int_0^T\bar f(\widehat{\mathbf X}^{(s)}_t)\dif t\to0$
	in $L^2$. The map $x\mapsto\int_0^T\bar f(x_t)\dif t$ is continuous on
	$D([0,T],E^N)$, because convergence in the Skorokhod topology implies
	convergence at all but countably many times. Since
	$\bar\eta_T(f)=T^{-1}\int_0^T\bar f(\bar{\mathbf Y}_t)\dif t$,
	Theorem~\ref{thm:switching-conv}\,\textup{(i)} and the continuous mapping
	theorem give, for all $f_1,\dots,f_j\in C_b(E^N)$,
	\[
	\bigl(\eta^{(s)}_T(f_1),\dots,\eta^{(s)}_T(f_j)\bigr)\Rightarrow
	\bigl(\bar\eta_T(f_1),\dots,\bar\eta_T(f_j)\bigr).
	\]
	It remains to see that the laws of the random measures $\eta^{(s)}_T$ are
	tight. Since $\mathcal V$ is invariant under permutations,
	$\EE\,\eta^{(s)}_T(B_K^c)\le\PP\bigl(\sup_{t\le T}
	\mathcal V(\widehat{\mathbf X}^{(s)}_t)>K\bigr)$. Given $\varepsilon>0$,
	\eqref{eq:cc} along the sequence gives levels $K_i$ with
	$\EE\,\eta^{(s)}_T(B_{K_i}^c)\le\varepsilon4^{-i}$ for every $i\ge1$ and
	every $s$ of the sequence, and by Markov's inequality, with probability
	at least $1-\varepsilon$, the measure $\eta^{(s)}_T$ lies in
	$\mathcal P_\varepsilon:=\{\mu:\mu(B_{K_i}^c)\le2^{-i}\text{ for every }i\}$.
	This set is relatively compact in $\mathcal P(E^N)$ by Prokhorov's theorem,
	and closed because $\mu\mapsto\mu(B_{K_i}^c)$ is lower semicontinuous, the
	sets $B_{K_i}^c$ being open, so it is compact. The functions
	$\mu\mapsto\mathsf h(\mu(f_1),\dots,\mu(f_j))$, with $\mathsf h\in C_b(\RR^j)$ and
	$f_i\in C_b(E^N)$, form an algebra of bounded continuous functions on
	$\mathcal P(E^N)$ that separates points, so they determine the law of a
	random measure \cite[Theorem~3.4.5]{ethierkurtz1986}. Every limit point
	of the laws of the $\eta^{(s)}_T$ is therefore the law of $\bar\eta_T$,
	which gives the convergence along the sequence, and the sequence is
	arbitrary.
\end{proof}

\section{Proofs of Propositions~\ref{prop:density} and~\ref{prop:moments-compare} and of Corollary~\ref{cor:bench-non-markov}}\label{app:density}

We derive the representation from the switching decomposition of
Section~\ref{sec:setup}, which gives each term a probabilistic meaning. The
rate being constant, every ring is accepted and the label flips at every
ring of a Poisson clock of rate $s$, which by
Theorem~\ref{thm:piecing}\,\textup{(ii)} is independent of the pieces.
Pasting the Brownian increments of the successive pieces of each replica
therefore gives two independent Brownian motions $W^1,W^2$, independent of
$\Lambda$, for which \eqref{eq:bench-switching} holds. Conditionally on the
path of $\Lambda$, the stochastic integrals in \eqref{eq:bench-switching} are
Wiener integrals with deterministic integrands. The pair
$(\widehat X_t,\widehat Y_t)$ is therefore conditionally Gaussian with
independent coordinates, and its conditional law depends on the path only
through the occupation time
\begin{equation}\label{eq:bench-occ}
	\Occ_t:=\int_0^t\mathbf 1_{\{\Lambda_u=\mathrm{id}\}}\dif u\in[0,t].
\end{equation}
Explicitly, the conditional density is $g_{\Occ_t}$, where for $a\in[0,t]$
\begin{equation}\label{eq:bench-ga}
	g_a(x,y)=\frac{1}{2\pi\sqrt{v_1(a)v_2(a)}}
	\exp\!\Bigl(-\frac{(x-m_1(a))^2}{2v_1(a)}-\frac{(y-m_2(a))^2}{2v_2(a)}\Bigr),
\end{equation}
\begin{equation}\label{eq:bench-mv}
	\begin{aligned}
		m_1(a)&=x_0+\mu_1a+\mu_2(t-a), &\qquad v_1(a)&=\sigma_1^2a+\sigma_2^2(t-a),\\
		m_2(a)&=y_0+\mu_2a+\mu_1(t-a), &\qquad v_2(a)&=\sigma_2^2a+\sigma_1^2(t-a).
	\end{aligned}
\end{equation}

\emph{The law of the occupation time.} Let $\mathcal N_t$ be the number of
exchanges in $[0,t]$. Given $\mathcal N_t=n$ the exchange times are the
order statistics of $n$ uniform points of $[0,t]$, so they have joint density
$e^{-st}s^n$ on $\{0<t_1<\dots<t_n<t\}$. The label equals $\mathrm{id}$ on
the sojourns of odd rank and $\tau$ on the others, and the map from the
exchange times to the first $n$ sojourn lengths has unit Jacobian. If
$n=2k+1$, then $\Lambda_t=\tau$, and there are $k+1$ sojourns of each kind,
of total lengths $a=\Occ_t$ and $t-a$. The set of configurations with
$\Occ_t\in[a,a+\dif a]$ has volume
$\frac{a^k}{k!}\frac{(t-a)^k}{k!}\dif a$, the product of the volumes of
the two simplices of sojourn lengths. If $n=2k$ with $k\ge1$, then
$\Lambda_t=\mathrm{id}$, with $k+1$ sojourns in $\mathrm{id}$ and $k$ in
$\tau$, and the volume is $\frac{a^k}{k!}\frac{(t-a)^{k-1}}{(k-1)!}\dif a$.
Summing over $k$ with the series
$I_\nu(x)=\sum_{j\ge0}(x/2)^{2j+\nu}/(j!\,(j+\nu)!)$
\cite[9.6.10]{abramowitz1964}, the joint law of $(\Occ_t,\Lambda_t)$
consists of an atom of mass $e^{-st}$ at $(t,\mathrm{id})$, which
corresponds to no exchange in $[0,t]$, and, for $0<a<t$, of the densities
\begin{equation}\label{eq:bench-occ-law}
	\psi_1(a)=s\,e^{-st}\sqrt{\tfrac{a}{t-a}}\;I_1\!\bigl(2s\sqrt{a(t-a)}\bigr),
	\qquad
	\psi_2(a)=s\,e^{-st}\,I_0\!\bigl(2s\sqrt{a(t-a)}\bigr)
\end{equation}
on $\{\Lambda_t=\mathrm{id}\}$ and $\{\Lambda_t=\tau\}$ respectively. With
$R_u$ as in Section~\ref{sec:brownian}, $2\Occ_t-t=\int_0^tR_u\dif u$ is the
position at time $t$ of a telegraph process with unit speed and switching
rate $s$, started with positive velocity, and \eqref{eq:bench-occ-law} is
the law of Goldstein and Kac \cite{goldstein1951,kac1974} split according
to the terminal velocity; see also \cite{orsingher1990}, and
\cite{kovchegov2010} for the same law written for the occupation time.

Averaging $g_{\Occ_t}$ over this law, and applying the coupling of
Theorem~\ref{thm:N-symmetrisation} to the swapping pair, gives
\begin{equation}\label{eq:bench-mixture}
	\widehat p(x,y,t)=\EE\bigl[g_{\Occ_t}(x,y)\bigr],
	\qquad
	p^{\swap}(x,y,t)=\EE\bigl[\mathbf 1_{\{\Lambda_t=\mathrm{id}\}}g_{\Occ_t}(x,y)
	+\mathbf 1_{\{\Lambda_t=\tau\}}g_{\Occ_t}(y,x)\bigr].
\end{equation}
Indeed, for a Borel set $B$ and $\theta\in\mathfrak S_2$, the independence of
$\Lambda$ and $(W^1,W^2)$ and Tonelli's theorem give
$Q^\theta_t(B)=\EE\bigl[\mathbf 1_{\{\Lambda_t=\theta\}}\int_Bg_{\Occ_t}\bigr]
=\int_B\EE\bigl[\mathbf 1_{\{\Lambda_t=\theta\}}g_{\Occ_t}\bigr]$, and the
swapping density follows from \eqref{eq:N-sym-identity}. With the series of
$I_\nu$ and the Beta integral
$\int_0^t\bigl(a(t-a)\bigr)^k\dif a=(k!)^2t^{2k+1}/(2k+1)!$, integrals of
$\psi_1$ and $\psi_2$ against polynomials reduce to exponential series. For
instance $\int_0^t\psi_2=s\,t\,e^{-st}\sum_{k\ge0}(st)^{2k}/(2k+1)!
=e^{-st}\sinh(st)=\PP(\Lambda_t=\tau)$.

\begin{proof}[Proof of Proposition~\ref{prop:density}]
	Split the integral over $a$ in \eqref{eq:bench-mixture} at $t/2$ and
	substitute $a=a_\chi(z):=\tfrac t2(1+\chi\sqrt z)$, with
	$\chi=-1$ on $(0,t/2)$ and $\chi=+1$ on $(t/2,t)$. Then
	$2s\sqrt{a(t-a)}=st\sqrt{1-z}$, $|\dif a|=\tfrac{t}{4\sqrt z}\dif z$ and
	$\sqrt{a/(t-a)}=(1+\chi\sqrt z)/\sqrt{1-z}$. By \eqref{eq:bench-mv}, the
	quantities $2v_1(a_\chi(z))$ and $2v_2(a_\chi(z))$ are the two denominators
	in \eqref{eq:Theta-def}, the means $m_1(a_\chi(z))$ and $m_2(a_\chi(z))$ are
	the two centres there, and
	$4v_1v_2=t^2\bigl[(\sigma_1^2+\sigma_2^2)^2-z(\sigma_1^2-\sigma_2^2)^2\bigr]$,
	so that $g_{a_\chi(z)}=\mathcal T(\cdot,z,t,\chi)/(\pi t)$. The atom gives
	the first line of \eqref{eq:density-sector1}, since $g_t=p_0$, $\psi_1$
	gives its second line and $\psi_2$ gives the sector density
	\eqref{eq:density-sector2}. Exchanging the arguments on the second sector
	gives the swapping density \eqref{eq:density-explicit}, which is
	\eqref{eq:N-sym-identity} for $N=2$. Every step is an identity between
	non-negative integrands, so Tonelli's theorem justifies the interchanges.
\end{proof}

\begin{proof}[Proof of Proposition~\ref{prop:moments-compare}]
	We integrate the densities. By the mixture representation
	\eqref{eq:bench-mixture} and Fubini's theorem, for every $f$ of at most
	quadratic growth,
	\[
	\begin{aligned}
	\EE f(\widehat X_t,\widehat Y_t)&=\EE\Bigl[\iint f\,g_{\Occ_t}\dif x\dif y\Bigr],\\
	\EE f(X_t,Y_t)&=\EE f(\widehat X_t,\widehat Y_t)
	+\EE\Bigl[\iint(f\circ\tau-f)\,g_{\Occ_t}\dif x\dif y;\,\Lambda_t=\tau\Bigr].
	\end{aligned}
	\]
	Under $g_a$ the coordinates are independent Gaussians with the means
	$m_i(a)$ and variances $v_i(a)$ of \eqref{eq:bench-mv}, so the moments of
	order at most two reduce to those of $\Occ_t$. These are elementary. With
	$R_u$ as above, $\Occ_t-t/2=\frac12\int_0^tR_u\dif u$, $\EE R_u=e^{-2su}$
	and $\EE[R_uR_v]=e^{-2s|u-v|}$, so that $\EE\Occ_t=\frac t2+\frac{1-e_t}{4s}$
	and
	\[
	\Var\int_0^tR_u\dif u
	=\frac ts-\frac{1-e_t}{2s^2}-\frac{(1-e_t)^2}{4s^2},
	\qquad
	\Var\Occ_t=\frac1{4s}\int_0^t\bigl(1-e^{-2su}\bigr)^2\dif u=\frac{\Omega_t}4 .
	\]
	\emph{Switching pair.} Since $m_1(a)=x_0+\mu_2t+2\delta a$ and
	$v_1(a)=\sigma_2^2t+(\sigma_1^2-\sigma_2^2)a$, integrating $x$ and $x^2$
	gives $\EE\widehat X_t=\EE m_1(\Occ_t)$ and
	$\Var\widehat X_t=\EE v_1(\Occ_t)+\Var m_1(\Occ_t)
	=\EE v_1(\Occ_t)+4\delta^2\Var\Occ_t$, which are \eqref{eq:rs-mean} and
	\eqref{eq:rs-var}. Integrating $xy$, and using $m_2(a)=y_0+\mu_1t-2\delta a$,
	gives $\Cov(\widehat X_t,\widehat Y_t)=\Cov\bigl(m_1(\Occ_t),m_2(\Occ_t)\bigr)
	=-4\delta^2\Var\Occ_t$, which is \eqref{eq:rs-cov}.

	\emph{Swapping pair.} The correction vanishes for $f(x,y)=xy$, so
	$\EE[X_tY_t]=\EE[\widehat X_t\widehat Y_t]$. For $f(x,y)=x$ and
	$f(x,y)=x^2$ it involves, with $\varsigma:=\Occ_t-t/2$,
	\[
	m_2-m_1=-d_0-4\delta\varsigma,
	\qquad
	(v_2+m_2^2)-(v_1+m_1^2)
	=-(x_0+y_0+2\bar\mu t)(d_0+4\delta\varsigma)
	-2(\sigma_1^2-\sigma_2^2)\varsigma,
	\]
	evaluated at $\Occ_t$, and both expressions are affine in $\varsigma$. The
	density $\psi_2$ in \eqref{eq:bench-occ-law} is invariant under
	$a\mapsto t-a$, so $\EE[\varsigma;\Lambda_t=\tau]=0$, while
	$\PP(\Lambda_t=\tau)=\tfrac12(1-e_t)$. Hence
	\[
	\EE X_t=\EE\widehat X_t-\tfrac{d_0}{2}(1-e_t),
	\qquad
	\EE X_t^2=\EE\widehat X_t^2-\tfrac{d_0}{2}(x_0+y_0+2\bar\mu t)(1-e_t),
	\]
	and the same holds for $Y$ with $d_0$ replaced by $-d_0$. Subtracting
	the products of the means and inserting \eqref{eq:rs-mean} gives
	\eqref{eq:mean_X}--\eqref{eq:cov}, and \eqref{eq:mean_diff} follows from
	\eqref{eq:mean_X} and its counterpart for $Y$.
\end{proof}

\begin{proof}[Proof of Corollary~\ref{cor:bench-non-markov}]
	We check the hypotheses of Corollary~\ref{cor:non-markov} with
	$\theta=\mathrm{id}$ and $\theta'=\tau$, with the notation introduced
	before Proposition~\ref{prop:filter}. The quadratic variation
	$[\widehat X^1]$ is a limit in probability of sums of squared increments of
	$\widehat X^1$, hence adapted to the completed filtration
	$(\FF^{\widehat{\mathbf X}}_t)$, and
	$[\widehat X^1]_t-[\widehat X^1]_{t-r}=\int_{t-r}^t\sigma^2_{\Lambda_u(1)}\dif u$.
	Almost surely, therefore,
	$r^{-1}\bigl([\widehat X^1]_t-[\widehat X^1]_{t-r}\bigr)\to\sigma^2_{\Lambda_t(1)}$
	as $r\downarrow0$, because $\Lambda_{t-}=\Lambda_t$ almost surely. Since
	$\sigma_1\ne\sigma_2$, the label is read from the past of the positions,
	that is, $\Lambda_t$ is $\FF^{\widehat{\mathbf X}}_t$-measurable.
	For $g(\mathbf y)=e^{-(y^1)^2}$ and $u>0$ small enough the function
	$h^{g,u}_{\mathrm{id}}-h^{g,u}_\tau$ is continuous and not identically zero.
	Indeed, the increment of replica $1$ over a time $u$ has a law
	$\varrho_\theta$ that depends only on the initial label, with variance
	$\sigma^2_{\theta(1)}u+O(u^2)$, so that
	$\varrho_{\mathrm{id}}\ne\varrho_\tau$ for $u$ small enough, while
	$h^{g,u}_\theta(\mathbf y)=\int e^{-(y^1+z)^2}\varrho_\theta(\dif z)$ is
	continuous by bounded convergence, and convolution with a Gaussian is
	injective on finite measures, its Fourier transform having no zeros. Hence
	$B$ is a non-empty open set. Finally, the measures $\upsilon_{\mathrm{id}}$
	and $\upsilon_\tau$ have positive densities
	(Proposition~\ref{prop:density}), so their restrictions to $B$ are
	not mutually singular.
\end{proof}

\section{Proof of Proposition~\ref{prop:regimes}}\label{app:asymptotics}

We use the mixture representation \eqref{eq:bench-mixture}, with $g_a$ as
in \eqref{eq:bench-ga}, $\Occ_t$ as in \eqref{eq:bench-occ} and
$\mathcal N_t$ the number of exchanges in $[0,t]$. Since
$m_1(t/2)=x_0+\bar\mu t$, $m_2(t/2)=y_0+\bar\mu t$ and
$v_1(t/2)=v_2(t/2)=\sigma_{\rm eff}^2t$, we have $p_0=g_t$,
$\bar p_0=g_{t/2}$ and $p_\infty=\frac12(g_{t/2}+g_{t/2}\circ\tau)$.

\begin{proof}[Proof of \textup{(i)}]
	On $\{\mathcal N_t=0\}$, which has probability $e^{-st}$, we have
	$\Occ_t=t$ and $\Lambda_t=\mathrm{id}$, so both integrands in
	\eqref{eq:bench-mixture} equal $g_t$. Since all the densities involved have
	unit mass, $\|p^{\swap}-g_t\|_{L^1}$ and $\|\widehat p-g_t\|_{L^1}$ are at
	most $2\,\PP(\mathcal N_t\ge1)=2(1-e^{-st})\to0$.
\end{proof}

\begin{proof}[Proof of \textup{(ii)}]
	By \eqref{eq:bench-mixture}, $\widehat p-\bar p_0=\EE[g_{\Occ_t}-g_{t/2}]$, so
	$\|\widehat p-\bar p_0\|_{L^1}\le\EE\|g_{\Occ_t}-g_{t/2}\|_{L^1}$. By
	\eqref{eq:bench-mv} the means and the variances of $g_a$ are affine in $a$,
	and the variances are at least $\sigma_{\min}^2t>0$, so $g_a\to g_{a'}$
	pointwise as $a\to a'$, hence in $L^1$ by Scheff\'e's lemma. The function
	$a\mapsto\|g_a-g_{t/2}\|_{L^1}$ is therefore continuous on $[0,t]$, bounded
	by $2$ and zero at $t/2$. Let $\varsigma:=\Occ_t-t/2$. By the proof of
	Proposition~\ref{prop:moments-compare}, $\EE\varsigma=(1-e^{-2st})/(4s)$ and
	$\Var\varsigma=\Omega_t/4\le t/(4s)$, because $(1-e_t^2)/4\le1-e_t$. Hence
	$\Occ_t\to t/2$ in probability as $s\to\infty$, and bounded convergence
	gives $\EE\|g_{\Occ_t}-g_{t/2}\|_{L^1}\to0$, which is the second limit in
	\eqref{eq:fast-compare}. For the swapping pair,
	\[
	\begin{aligned}
	p^{\swap}-p_\infty&=\EE\bigl[g_{\Occ_t}-g_{t/2};\Lambda_t=\mathrm{id}\bigr]
	+\EE\bigl[(g_{\Occ_t}-g_{t/2})\circ\tau;\Lambda_t=\tau\bigr]\\
	&\quad+\bigl(\PP(\Lambda_t=\mathrm{id})-\tfrac12\bigr)\bigl(g_{t/2}-g_{t/2}\circ\tau\bigr).
	\end{aligned}
	\]
	Composition with $\tau$ is an isometry of $L^1$, so the first two terms
	have total norm at most $\EE\|g_{\Occ_t}-g_{t/2}\|_{L^1}$, and since
	$\PP(\Lambda_t=\mathrm{id})=\tfrac12(1+e^{-2st})$ the last one has norm at
	most $e^{-2st}$. Both tend to $0$ as $s\to\infty$ at fixed $t$.
\end{proof}

Part (iii) uses one more estimate. Let $\varphi_{m,v}$ be the normal density
on $\RR$ with mean $m$ and variance $v$ and put
$\omega_0:=|\sigma_1^2-\sigma_2^2|/(\sigma_1^2+\sigma_2^2)<1$. Pinsker's
inequality, in the form $\|\varphi-\varphi'\|_{L^1}\le\sqrt{2\,\mathrm{KL}}$,
together with
$\mathrm{KL}\bigl(\mathcal{N}(m,v)\,\|\,\mathcal{N}(m',v')\bigr)
=\tfrac12\bigl[\tfrac{v}{v'}-1-\log\tfrac{v}{v'}+\tfrac{(m-m')^2}{v'}\bigr]$
and the bound $r-1-\log r\le c_{\omega_0}(r-1)^2$ on
$[1-\omega_0,1+\omega_0]$, gives a constant $C=C(\omega_0)$ with
\begin{equation}\label{eq:app:pinsker}
	\|\varphi_{m,v}-\varphi_{m',v'}\|_{L^1}
	\le C\Bigl(\frac{|m-m'|}{\sqrt{v'}}+\Bigl|\frac{v}{v'}-1\Bigr|\Bigr)
	\qquad\text{when }\frac{v}{v'}\in[1-\omega_0,1+\omega_0].
\end{equation}
We also use that the $L^1$ distance between two product densities on
$\RR^2$ is at most the sum of the $L^1$ distances between their factors.

\begin{proof}[Proof of \textup{(iii)}]
	Fix $s>0$. For a density $f$ on $\RR^2$ write
	$\tilde f(\mathbf v):=t\,f\bigl(\bar m_t(1,1)+\sqrt t\,\mathbf v\bigr)$. This
	rescaling preserves $L^1$ distances and commutes with $\tau$, and
	$\tilde G^{(s)}_t$ is the $\mathcal N(0,\Sigma_s)$ density for every $t$.

	\emph{Step 1. A central limit theorem for the clock.} Let
	$\mathsf Z_t:=t^{-1/2}(2\Occ_t-t)=t^{-1/2}\int_0^tR_u\dif u$. The generator
	of $\Lambda$ maps $R$ to $-2sR$, so
	$\mathcal M_t:=R_t-R_0+2s\int_0^tR_u\dif u$ is a martingale. Its jumps are
	those of $R$, of size $2$ and rate $s$, so its predictable quadratic
	variation is $4st$, and the martingale central limit theorem
	\cite[Theorem~7.1.4]{ethierkurtz1986}, applied along any sequence
	$t_n\to\infty$ to $\mathcal M_{t_nu}/\sqrt{t_n}$, gives
	$\mathcal M_t/\sqrt t\Rightarrow\mathcal N(0,4s)$. Hence
	$\mathsf Z_t=(\mathcal M_t-R_t+R_0)/(2s\sqrt t)\Rightarrow\mathsf Z\sim\mathcal N(0,1/s)$.
	The convergence holds jointly with $\Lambda_t$, towards a uniform variable
	$\Lambda_\infty$ on $\{\mathrm{id},\tau\}$ independent of $\mathsf Z$. Indeed, with
	$\varkappa_t:=t^{1/4}$, the variable
	$\mathsf Z'_t:=t^{-1/2}\int_0^{t-\varkappa_t}R_u\dif u$ is
	$\FF_{t-\varkappa_t}$-measurable and
	$|\mathsf Z_t-\mathsf Z'_t|\le\varkappa_t/\sqrt t\to0$, while conditionally
	on $\FF_{t-\varkappa_t}$ the law of $\Lambda_t$ differs from the uniform one
	by at most $e^{-2s\varkappa_t}\to0$. We also note that
	$\EE\mathsf Z_t^2=\Omega_t/t+(1-e^{-2st})^2/(4s^2t)$ is bounded in $t$.

	\emph{Step 2. Freezing the variances.} For $z\in\RR$ let $\mathfrak f_z$ be
	the density on $\RR^2$ of the normal law
	$\mathcal N\bigl(\delta z(1,-1),\sigma_{\rm eff}^2I\bigr)$. By \eqref{eq:bench-mv}, with
	$\Occ_t=\tfrac t2(1+\mathsf Z_t/\sqrt t)$, the rescaled conditional density
	$\tilde g_{\Occ_t}$ is the product of two normal densities with means and
	variances
	\[
	\pm\Bigl(\frac{d_0}{2\sqrt t}+\delta\mathsf Z_t\Bigr)
	\qquad\text{and}\qquad
	\sigma_{\rm eff}^2\pm\frac{\sigma_1^2-\sigma_2^2}{2\sqrt t}\,\mathsf Z_t .
	\]
	Since $|2\Occ_t-t|\le t$, the ratio of each variance to
	$\sigma_{\rm eff}^2$ lies in $[1-\omega_0,1+\omega_0]$, so the estimate
	\eqref{eq:app:pinsker}, applied to each factor with $v'=\sigma_{\rm eff}^2$,
	gives
	\[
	\bigl\|\tilde g_{\Occ_t}-\mathfrak f_{\mathsf Z_t}\bigr\|_{L^1}
	\le C\,\frac{|d_0|+|\mathsf Z_t|}{\sqrt t},
	\]
	with $C$ depending only on $\sigma_1,\sigma_2$. Since
	$\mathfrak f_z\circ\tau=\mathfrak f_{-z}$ and $\EE|\mathsf Z_t|$ is bounded,
	\eqref{eq:bench-mixture} yields
	\[
	\bigl\|\tilde{\widehat p}_t-\EE\,\mathfrak f_{\mathsf Z_t}\bigr\|_{L^1}\to0,
	\qquad
	\bigl\|\tilde p^{\swap}_t-\EE\bigl[\mathbf 1_{\{\Lambda_t=\mathrm{id}\}}\mathfrak f_{\mathsf Z_t}
	+\mathbf 1_{\{\Lambda_t=\tau\}}\mathfrak f_{-\mathsf Z_t}\bigr]\bigr\|_{L^1}\to0 .
	\]

	\emph{Step 3. Passage to the limit.} The map $z\mapsto\mathfrak f_z$ is
	continuous from $\RR$ into $L^1(\RR^2)$ and takes values of norm one. By
	Step~1 and the Skorokhod representation, with dominated convergence in
	$L^1$, the two expectations above converge in $L^1$ to
	$\EE\,\mathfrak f_{\mathsf Z}$ and to
	$\tfrac12\EE\,\mathfrak f_{\mathsf Z}+\tfrac12\EE\,\mathfrak f_{-\mathsf Z}
	=\EE\,\mathfrak f_{\mathsf Z}$, the law of $\mathsf Z$ being symmetric.
	Finally $\EE\,\mathfrak f_{\mathsf Z}$ is the density of
	$\delta\mathsf Z(1,-1)+\sigma_{\rm eff}\,\mathbf n$, with $\mathbf n$ a
	standard normal vector of $\RR^2$ independent of $\mathsf Z$, which is
	centred Gaussian with covariance
	$\sigma_{\rm eff}^2I+\tfrac{\delta^2}{s}\bigl(\begin{smallmatrix}1&-1\\-1&1\end{smallmatrix}\bigr)=\Sigma_s$.
\end{proof}

	
	\ACKNO{The authors declare no competing interests.}

\end{document}